\documentclass[english]{smfart}
\usepackage[T1]{fontenc}
\usepackage[utf8]{inputenc}
\usepackage{hyperref,amssymb,url,upref,verbatim,xspace,smfenum}
\RequirePackage[all,ps,cmtip]{xy}
\newdir^{ (}{{}*!/-5pt/\dir^{(}}
\newdir_{ (}{{}*!/-5pt/\dir_{(}}
\RequirePackage[mathscr]{eucal}
\usepackage{mathrsfs}\let\mathcal\mathscr
\numberwithin{subsection}{section}

\AtBeginDocument{%
\def\labelitemi{$\scriptscriptstyle \bullet$}%
}

\makeatletter
\@namedef{subjclassname@2010}{\textup{2010} Mathematics Subject Classification}

\DeclareRobustCommand{\@pointrait}{\unskip\@addpunct{.}%
\mbox{}\ignorespaces}

\def\th@plain{%
\let\thm@indent\noindent
\thm@headfont{\fontfamily{ptm}\bfseries\itshape}%
\thm@notefont{\fontfamily{ptm}\bfseries\upshape}%
\thm@preskip.5\linespacing \@plus.5\linespacing
\thm@postskip\thm@preskip
\thm@headpunct{\bfseries\itshape\pointrait}
\itshape }
\makeatother

\newcounter{toto}
\def\thetoto{\arabic{toto}}
\let\oldmarginpar\marginpar
\def\marginpar#1{\refstepcounter{toto}\textsuperscript{\textup{[\thetoto]}}\oldmarginpar{\footnotesize\textsuperscript{[\thetoto]}\,#1}}

\newenvironment{enumeratei}
{\bgroup\def\theenumi{\roman{enumi}}\begin{enumerate}}
{\end{enumerate}\egroup}
\newenvironment{enumeratea}
{\bgroup\def\theenumi{\alph{enumi}}\begin{enumerate}}
{\end{enumerate}\egroup}
\let\oldtheequation\theequation
\def\numstareq{\let\oldtheequation\theequation\renewcommand{\theequation}{\oldtheequation$\,*$}}
\def\numstarstareq{\let\oldtheequation\theequation\renewcommand{\theequation}{\oldtheequation$\,**$}}
\newenvironment{starequation}{\numstareq\begin{equation*}\tag{\theequation}}{\end{equation*}\let\theequation\oldtheequation\ignorespaces}

\def\sha{\mathcal{A}}

\def\shc{\mathcal{C}}
\def\shd{\mathcal{D}}\let\cD\shd

\def\shf{\mathcal{F}}\let\cF F
\def\shg{\mathcal{G}}\let\cG G
\def\shh{\mathcal{H}}
\let\cI\shi

\def\shh{\mathcal{H}}
\def\shk{\mathcal{K}}\let\shk\shk
\def\shl{\mathcal{L}}
\def\shm{\mathcal{M}}
\def\shn{\mathcal{N}}\let\shn\shn
\def\sho{\mathcal{O}}\let\cO\sho
\def\shr{\mathcal{R}}

\def\sht{\mathcal{T}}

\newcommand{\C}{\mathbb{C}}\let\CC\C
\newcommand{\DD}{\mathbb{D}}
\newcommand{\N}{\mathbb{N}}\let\NN\N
\newcommand{\R}{\mathbb{R}}\let\RR\R
\let\QQ\Q
\newcommand{\Z}{\mathbb{Z}}\let\ZZ\Z

\newcommand{\bD}{\boldsymbol{D}}
\newcommand{\bK}{\boldsymbol{K}}

\newcommand{\bR}{\boldsymbol{R}}
\newcommand{\Rhom}{R\shhom}
\newcommand{\shhom}{\mathcal{H}\!\mathit{om}}\let\ho\shhom
\DeclareMathOperator{\rh}{\mathit{R}\shhom}
\DeclareMathOperator{\tho}{\mathit{T}\shhom}

\DeclareMathOperator{\Rh}{\mathrm{RHom}}
\DeclareMathOperator{\RH}{RH}
\let\TH\THH

\newcommand{\rb}{\mathrm{b}}
\newcommand{\rd}{\mathrm{d}}

\newcommand{\gr}{\mathrm{gr}}
\newcommand{\coh}{\mathrm{coh}}
\newcommand{\hol}{\mathrm{hol}}
\newcommand{\rhol}{\mathrm{rhol}}
\newcommand{\Mod}{\mathrm{Mod}}
\newcommand{\imin}[1]{#1^{-1}}

\newcommand{\nb}{\mathrm{nb}}
\newcommand{\op}{\mathrm{op}}
\newcommand{\red}{\mathrm{red}}
\newcommand{\sa}{\mathrm{sa}}
\newcommand{\tors}{\mathrm{tors}}
\newcommand{\lf}{\mathrm{lf}}
\newcommand{\rhog}{\boldsymbol{\rho}}
\renewcommand{\thetag}{\boldsymbol{\theta}}
\newcommand{\bg}{\boldsymbol{g}}

\newcommand{\cc}{{\C\textup{-c}}}
\newcommand{\rc}{{\R\textup{-c}}}

\newcommand{\XS}{X\times S}
\newcommand{\XpS}{X'\times S}
\newcommand{\XsS}{X^*\times S}
\newcommand{\YS}{Y\times S}
\newcommand{\ZS}{Z\times S}
\newcommand{\US}{U\times S}
\newcommand{\DXS}{\shd_{\XS/S}}
\newcommand{\DXpS}{\shd_{\XpS/S}}
\newcommand{\DXSp}{\shd_{\XS'/S'}}
\newcommand{\DXSR}{\shd_{\XS_\RR/S}}
\newcommand{\DYS}{\shd_{Y\times S/S}}
\newcommand{\DZS}{\shd_{Z\times S/S}}
\newcommand{\DUS}{\shd_{U\times S/S}}

\newcommand{\lind}[1]{\underset{#1}{\varinjlim}}

\newcommand{\lpro}[1]{\underset{#1}{\varprojlim}}

\DeclareMathOperator{\Char}{Char}

\DeclareMathOperator{\codim}{codim}
\DeclareMathOperator{\coker}{coker}
\DeclareMathOperator{\pD}{{}^\mathrm{p}\mathsf{D}}
\DeclareMathOperator{\rD}{\mathsf{D}}
\DeclareMathOperator{\rK}{\mathsf{K}}
\DeclareMathOperator{\DR}{DR}
\DeclareMathOperator{\Db}{\mathfrak{Db}}
\DeclareMathOperator{\pDR}{{}^\mathrm{p}DR}
\DeclareMathOperator{\GL}{GL}
\DeclareMathOperator{\Hom}{Hom}
\DeclareMathOperator{\id}{Id}\let\Id\id
\let\im\imm
\DeclareMathOperator{\MTM}{MTM}
\DeclareMathOperator{\Op}{Op}

\DeclareMathOperator{\Rep}{Rep}
\DeclareMathOperator{\reel}{Re}
\DeclareMathOperator{\Sol}{Sol}
\DeclareMathOperator{\pSol}{{}^\mathrm{p}Sol}
\DeclareMathOperator{\supp}{Supp}
\let\tilde\widetilde
\let\wt\widetilde
\let\wh\widehat
\let\ov\overline
\let\epsilon\varepsilon

\let\setminus\smallsetminus
\let\leq\leqslant
\let\geq\geqslant

\def\loccit{loc.\kern3pt cit.{}\xspace}
\def\cf{cf.\kern.3em}
\def\eg{e.g.\kern.3em}
\def\ie{i.e.,\ }
\def\resp{\text{resp.}\kern.3em}
\let\moins\smallsetminus
\def\wtj{\wt\jmath}
\newcommand{\Df}{{}_{\scriptscriptstyle\mathrm{D}}f}

\newcommand{\Di}{{}_{\scriptscriptstyle\mathrm{D}}i}
\newcommand{\Dpi}{{}_{\scriptscriptstyle\mathrm{D}}\pi}

\newcommand{\cbbullet}{{\raisebox{1pt}{$\sbullet$}}}
\newcommand{\sbullet}{{\scriptscriptstyle\bullet}}
\newcommand{\pOS}{p^{-1}\sho_S}
\DeclareMathOperator{\Aut}{Aut}
\DeclareMathOperator{\End}{End}

\numberwithin{equation}{section}

\theoremstyle{plain}
\newtheorem{theorem}[equation]{Theorem}
\newtheorem{proposition}[equation]{Proposition}
\newtheorem{lemma}[equation]{Lemma}
\newtheorem{corollary}[equation]{Corollary}
\newtheorem{theoremintro}{Theorem}
\newtheorem{propositionintro}[theoremintro]{Proposition}
\newtheorem{corollaryintro}[theoremintro]{Corollary}

\theoremstyle{definition}
\newtheorem{assumption}[equation]{Assumption}
\newtheorem{definition}[equation]{Definition}

\newtheorem{example}[equation]{Example}

\newtheorem{remark}[equation]{Remark}

\newtheorem*{claim*}{Claim}
\newtheorem{remintro}[theoremintro]{Remark}

\newcommand{\RedefinitSymbole}[1]{%
\expandafter\let\csname old\string#1\endcsname=#1
\let#1=\relax
\newcommand{#1}{\csname old\string#1\endcsname\,}%
}
\RedefinitSymbole{\forall} \RedefinitSymbole{\exists}

\let\ra\rightarrow
\def\to{\mathchoice{\longrightarrow}{\rightarrow}{\rightarrow}{\rightarrow}}
\def\mto{\mathchoice{\longmapsto}{\mapsto}{\mapsto}{\mapsto}}
\def\hto{\mathrel{\lhook\joinrel\to}}
\def\from{\mathchoice{\longleftarrow}{\leftarrow}{\leftarrow}{\leftarrow}}
\def\implique{\mathchoice{\Longrightarrow}{\Rightarrow}{\Rightarrow}{\Rightarrow}}

\def\To#1{\mathchoice{\xrightarrow{\textstyle\kern4pt#1\kern3pt}}{\stackrel{#1}{\longrightarrow}}{}{}}

\def\isom{\stackrel{\sim}{\longrightarrow}}

\let\oldbigoplus\bigoplus
\renewcommand{\bigoplus}{\mathop{\textstyle\oldbigoplus}\displaylimits}
\let\oldbigwedge\bigwedge
\renewcommand{\bigwedge}{\mathop{\textstyle\oldbigwedge}\displaylimits}
\let\oldbigcap\bigcap
\renewcommand{\bigcap}{\mathop{\textstyle\oldbigcap}\displaylimits}
\let\oldprod\prod
\renewcommand{\prod}{\mathop{\textstyle\oldprod}\displaylimits}

\begin{document}
\frontmatter

\title{Riemann-Hilbert correspondence for mixed~twistor $\mathcal D$-Modules}

\author[T. Monteiro Fernandes]{Teresa Monteiro Fernandes}
\address[T. Monteiro Fernandes]{Centro de Matemática e Aplicações Fundamentais -- Centro de investigação Operacional e Departamento de Matemática da FCUL, Edifício C 6, Piso 2, Campo Grande, 1700, Lisboa, Portugal}
\email{mtfernandes@fc.ul.pt}

\author[C.~Sabbah]{Claude Sabbah}
\address[C.~Sabbah]{CMLS, École polytechnique, CNRS, Université Paris-Saclay\\
F--91128 Palaiseau cedex\\
France}
\email{Claude.Sabbah@polytechnique.edu}
\urladdr{http://www.math.polytechnique.fr/perso/sabbah}

\thanks{The research of TMF was supported by Fundação para a Ciência e Tecnologia, PEst OE/MAT/\allowbreak UI0209/2011. The research of CS was supported by the grant ANR-13-IS01-0001-01 of the Agence nationale de la recherche.}



\begin{abstract}
We introduce the notion of regularity for a relative holonomic $\mathcal D$-module in the sense of \cite{MF-S12}. We prove that the solution functor from the bounded derived category of regular relative holonomic modules to that of relative constructible complexes is essentially surjective by constructing a right quasi-inverse functor. When restricted to relative $\mathcal D$-modules underlying a regular mixed twistor $\mathcal D$-module, this functor satisfies the left quasi-inverse property.
\end{abstract}
\maketitle

\tableofcontents
\mainmatter

\section*{Introduction}
Let $X$ and $S$ be complex manifolds and let $p$ be the projection $\XS\to S$. We will set $d_X=\dim_\CC X$, $d_S=\dim_\CC S$. In~\cite{MF-S12} (see Section \ref{S:1} for a reminder), we have considered a restricted notion of holomorphic family parametrized by $S$ of holonomic $\cD_X$-modules. These are coherent modules on the sheaf $\cD_{\XS/S}$ of relative differential operators whose characteristic variety, in the product $(T^*X)\times S$, is contained in $\Lambda\times S$ for some Lagrangean conic closed subset $\Lambda$ of $T^*X$. This notion is restrictive in the sense that $\Lambda$ does not vary with respect to $S$. We have also introduced the derived category of sheaves of $\pOS$-modules with $\C$\nobreakdash-constructible cohomology, also called $S$-$\C$-constructible complexes, together with the corresponding notion of perversity, and we have proved that the de\,Rham functor $\DR$ and its adjoint by duality, the solution functor $\Sol$, on the bounded derived category of $\cD_{\XS/S}$-modules with holonomic cohomology take values in the derived category of $S$-$\C$-constructible complexes. Denoting by $\pSol(\shm)$ \resp $\pDR(\shm)$ the complex $\Sol(\shm)[d_X]$ \resp $\DR(\shm)[d_X]$, these $S$-$\CC$-constructible complexes are related by duality: $\bD\pSol(\shm)=\pDR(\shm)$.

Many properties in the relative setting can be obtained from those in the ``absolute case'' (\ie when $S$ is reduced to a point), by specializing the parameter and by considering analogous properties for the restricted objects by the functors $Li_{s_o}^*$ when~$s_o$ varies in $S$. As a consequence, \emph{strictness}, that is, $\pOS$-flatness (or absence of $\pOS$-torsion if $\dim S=1$), plays an important role at various places. On the other hand, for an $S$-$\C$-constructible perverse complex $F$, the dual $S$-$\C$-constructible complex $\bD F$ needs not be perverse, and the subcategory of $S$-$\C$-constructible complexes $F$ such that $F$ and $\bD F$ are perverse is specially interesting. Both notions (strictness and perversity of $F$ and $\bD F$) are related.\enlargethispage{1.5\baselineskip}%

\begin{propositionintro}\label{prop:FDFperverse}
Assume that that $F$ and $\bD F$ are perverse. Let $(X)_{\alpha}$ be a stratification of $X$ adapted to $F$. Then, for any open strata $X_{\alpha}$, $\shh^{-d_X}i^{-1}_{\alpha}F$ is a locally free $\pOS$-module of finite rank ($d_X:=\dim X$).

Conversely, let $Y$ be a hypersurface of $X$ and let $F$ be a locally free $\pOS$-module of finite rank on $(X\setminus Y)\times S$. Then $j_!F[d_X]$ and its dual $Rj_*F[d_X]$ are perverse ($d_X=\dim_\CC X$, $j:(X\setminus Y)\times S\hto\XS$).
\end{propositionintro}

Moreover, when $F$ is the de~Rham complex or the solution complex of a holonomic $\DXS$-module, we have the following improvement of \cite[Th.\,1.2]{MF-S12}.

\begin{propositionintro}\label{P:3.3}
Let $\shm$ belong to $\rD^\rb_{\hol}(\DXS)$. Then the following conditions are equivalent.
\begin{enumerate}
\item
$\shm$ is concentrated in degree $0$ and $\shh^0(\shm)$ is strict.

\item
$\bD \shm$ is concentrated in degree $0$ and $\shh^0(\bD \shm)$ is strict.
\item
$\pSol(\shm)$ and $\pDR(\shm)=\bD\pSol(\shm)$ are perverse.
\end{enumerate}
\end{propositionintro}

Going further, it is natural to define the subcategory of \emph{regular} relative holonomic $\shd$-modules by imposing the regularity condition to each $Li_{s_o}^*\shm$ (\cf Section \ref{subsec:regularity}).

Our main objective in this article is to approach the problem of constructing a quasi-inverse functor to $\Sol$ restricted to the category of regular holonomic $\cD_{\XS/S}$-modules. In analogy with the method of Kashiwara \cite{Kashiwara84}, we introduce the functor $\RH^S$, from the derived category $\rD^\rb_\cc(\pOS)$ of $S$-$\C$-constructible complexes to the bounded derived category $\rD^\rb_\rhol(\cD_{\XS/S})$ of $\cD_{\XS/S}$-modules with regular holonomic cohomology. Roughly speaking, it is a relative version of the functor $\tho(\cdot,\sho)$ of Kashiwara \cite{Kashiwara84} using the language of ind-sheaves or sheaves on a subanalytic site (\cite{K-S01}, \cite{K-Sch06}, \cite{Prelli08}). In the locally constant case it coincides with the construction due to Deligne \cite{Deligne70}.

However, contrary to the absolute case, the behaviour by pull-back is not always controlled, due to the lack of an existence theorem of a Bernstein-Sato polynomial, so it remains conjectural that the derived category $\rD^\rb_\rhol(\cD_{\XS/S})$ is stable by inverse images. This constitutes a major obstacle to obtain an equivalence of categories as in the absolute case. Our first main result concerns essential surjectivity of $\Sol:\rD^\rb_\rhol(\cD_{\XS/S})\mto\rD^\rb_\cc(\pOS)$, when $S$ is a curve. This restriction to $\dim S=1$ is needed in order to find bases of open coverings of the subanalytic site $S_{\sa}$ formed by $\sho_S$-acyclic open subsets.

\begin{theoremintro}\label{T:11}
Assume that $\dim S=1$ and let $F\in\rD^\rb_\cc(\pOS)$. Then $\RH^S(F)\in \rD^\rb_\rhol(\DXS)$ and we have a functorial isomorphism $\pSol(\RH^S(F))\simeq F$ in $\rD^\rb_\cc(\pOS)$.
\end{theoremintro}

As a consequence, we obtain:

\begin{corollaryintro} \label{P:fin1}
Assume that $\dim S=1$ and let $F\in\rD^\rb_\cc(\pOS)$ be such that $F$ and~$\bD F$ are perverse. Set $\shm:=\RH^S(F)\in\rD^\rb_{\rhol}(\DXS)$. Then $\shm$ is concentrated in degree~$0$ and $\shh^0\shm$ is strict.
\end{corollaryintro}

Let $\sha$ be a $\QQ$-vector subspace of $\RR\times\CC$. The category $\MTM(X)_\sha$ of \emph{$\sha$-mixed twistor $\cD$\nobreakdash-modules} with KMS exponents in $\sha$ on the complex manifold $X$, together with the corresponding functors (pushforward by a projective map, duality, localization, etc.) has been introduced by T.\,Mochizuki in \cite{Mochizuki11}. Roughly speaking  (\cf\cite[\S7.1.3]{Mochizuki11} for details), for pure objects of $\MTM(X):=\MTM(X)_{\RR\times\CC}$, the set $\sha$ bounds the possible asymptotic behaviour of the norms of sections with respect to the corresponding harmonic metric, as well as the possible monodromies on nearby cycles along functions of the associated holonomic $\cD_X$\nobreakdash-modules (formal monodromies in the wild case), via the two functions
\[
\sha\times S\ni(a,\alpha,s)\mto a+2\reel(\ov\alpha s)\in\RR,\quad (a,\alpha,s)\mto\alpha-as-\ov\alpha s^2\in\CC.
\]
Let us recall the main properties we use. A more detailed reminder is given in Section \ref{subsec:MTMreminder}. An object of $\MTM(X)_\sha$ is a $W$-filtered triple whose first two components are $W$-filtered sheaves on $X\times\CC$ and the third component is a sesquilinear pairing between their restriction to $X\times S^1$ satisfying a number of properties. For our purpose, we set $S=\CC^*$ and we restrict the first two components to $\XS$, which consist then of $W$-filtered $\DXS$\nobreakdash-modules (\cf \cite{Bibi01c,Mochizuki07,Mochizuki11}). For the sake of simplicity, we shall say that a $\DXS$-module \emph{$\shm$ underlies an $\sha$-mixed twistor $\cD$-module} if it is the second $\DXS$-module of the pair. This defines a subcategory of $\Mod_\rhol(\DXS)$ (morphisms are similarly induced by morphisms in $\MTM(X)$), which is not full however, but is endowed with relative proper direct image functor, a duality functor and a localization functor (\cf \cite{Mochizuki11}). These properties are essential to prove our main application.

\begin{theoremintro}\label{C:14}
Let us fix $\sha=\RR\times\{0\}\subset\RR\times\CC$. Assume that $\shm\in\Mod_\rhol(\DXS)$ underlies an $\sha$-mixed twistor $\cD$-module. Then there exists a canonical isomorphism
\[\tag{$*$}
\shm\simeq \RH^S(\pSol(\shm))
\]
which is functorial with respect to morphisms in $\Mod(\DXS)$ between objects $\shm,\shn$ of $\Mod_\rhol(\DXS)$ underlying $\sha$-mixed twistor $\cD$-modules. Moreover, we have a natural isomorphism
\[\tag{$**$}
\Hom_{\DXS}(\shm,\shn)\simeq\Hom_{\pOS}(\pSol\shn,\pSol\shm).
\]
\end{theoremintro}

\begin{remintro}\label{rem:intro}
More generally, if $\sha\subset\RR\times\CC$ is finite-dimensional over $\QQ$, the statement of Theorem \ref{C:14} holds for all $\shm$ underlying an $\sha$-mixed twistor $\cD$-module away from the subset $S_0\subset S$ defined by the equations $\alpha-as-\ov\alpha s^2\in\ZZ$ for $(a,\alpha)\in\sha$, $\alpha\neq0$. For a given $\shm$ and locally on~$X$, only a finite number of such equations are needed to define the corresponding $S_0$, which is thus discrete in $S$ (\cf Remark \ref{rem:proviso} below).
\end{remintro}

We do not know how to characterize the essential image of the category of regular holonomic $\cD_{\XS/S}$-modules underlying an $\sha$-mixed twistor $\shd$-module by the functor $\pSol$, although we know that, for such a module $\shm$, $\pSol(\shm)$ and its dual $\pDR(\shm)$ are perverse in $\rD^\rb_\cc(\pOS)$.

This paper is organized as follows. In Section \ref{S:1}, we prove the complements to~\cite{MF-S12}, that is, Propositions \ref{prop:FDFperverse} and \ref{P:3.3}. We establish in Theorem \ref{th:Deligneext} and Corollary \ref{cor:Deligneext} the generalization of Deligne's results on the extension of a relative holomorphic connexion on the complementary of a normal crossing divisor. The construction of $\RH^S$, explained in Section \ref{sec:3}, is based on the notion of relative tempered distributions and holomorphic functions introduced in \cite{MF-P14}, which form subanalytic sheaves in the relative subanalytic site. Note that the constraint to the case $\dim S=1$ is not inconvenient for the application since mixed twistor $\shd$-modules satisfy this condition. In the locally free case, the above extension is also obtained using the functor $\RH^S$ as proved in Lemma~\ref{RHV}. Moreover, in this case, it provides an equivalence of categories (Theorem~\ref{T:D-T}). We obtain Theorem \ref{T:11} as a consequence of Lemma \ref{RHV}, and Theorem \ref{C:14} is proved in Section~\ref{subsec:C14} by reducing to \cite[Cor.\,8.6]{Kashiwara84}. In the appendix we collect various results which are essential for the remaining part of the paper.

\subsubsection*{Acknowledgements}
This work has benefited from discussions with Andrea d'Agnolo, Masaki Kashiwara, Yves Laurent and Luca Prelli, whom we warmly thank. We also thank Daniel Barlet for having kindly provided us with a proof of Lemma \ref{fibers}. We finally thank the referee for useful suggestions leading to an improvement of the presentation of the article.

\section{Some complementary results to \texorpdfstring{\cite{MF-S12}}{MFS12}}\label{S:1}

\subsection{Notation and preliminary results}
Throughout this work $X$ and $S$, unless specified, will denote complex manifolds and $p_X:\XS\to S$ will denote the projection. We will set $d_X:=\dim X$, $d_S:=\dim S$, and for any complex space $Z$, we will set similarly $d_Z=\dim Z$. We will often write $p$ instead of $p_X$ when there is no risk of ambiguity. We say that a $\pOS$-module is \emph{strict} if it is $\pOS$-flat. Given $s_o\in S$, we denote by $Li^*_{s_o} (\cbbullet)$ the derived functor on $\rD^\rb(\pOS)$ of
$$F\to F\otimes_{\pOS}p^{-1}(\sho_S/\mathfrak{m}_{s_o}),$$ where $\mathfrak{m}_{s_o}$ denotes the maximal ideal of holomorphic functions on $S$ vanishing at~$s_o$. The following results are straightforward.

\begin{lemma}\label{L:s}
Let $N\in\rD^{\geq0} (\pOS)$ and let $s_o\in S$. Then
$Li^*_{s_o}(N)\in \rD^{\geq-d_S}(X)$. If moreover, for any $k$, $\shh^k(N)$ is strict then
$Li^*_{s_o}(N)\in \rD^{\geq 0}(X)$.
\end{lemma}

\begin{comment}
\begin{proof}
For the first part of the statement it is sufficient to remark that $\sho/\mathfrak{m}_{s_o}$ admits locally a flat resolution of length equal to $d_S$.

For the second part, note that the assumption entails that
$\shh^k(Li^*_{s_o}N)=i^*_{s_o}\shh^k(N)$ for any $k$. Since
$\shh^kN=0$ for each $k<0$, we also have $\shh^k(Li^\ast_{s_o}N)=0$ for
each such~$k$.
\end{proof}
\end{comment}

\begin{lemma}\label{RGamma}
For any locally closed subset $Z$ of $\XS$, for any $F\in \rD^\rb(\pOS)$ and for any $s_o\in S$, we have $R\Gamma_Z(Li^*_{s_o}(F))\simeq Li^*_{s_o}(R\Gamma_Z(F))$.
\end{lemma}

\begin{comment}
\begin{proof}
The category $\Mod(\pOS)$ admits enough injective objects. Arguing by induction on $\dim S$, we may assume that $\dim S=1$ and that $S$ is an open subset of $\C$ with a coordinate $s$ vanishing at $s_o$. Thus, replacing $F$ by an injective resolution $F^{\sbullet}$,
$Li^*_{s_o} (R\Gamma_Z(F))$ is given by the simple complex associated to $\Gamma_Z(F^{\sbullet})\To{s}\Gamma_Z(F^{\sbullet})$. On the other hand, we have
$$Li^*_{s_o}(F)\simeq \{F^{\sbullet}\To{s} F^{\sbullet}\}.$$ The injectivity of $F^{\sbullet}$ then entails that $R\Gamma_Z(Li^*_{s_o}(F))$ is quasi-isomorphic to $\Gamma_Z(F^{\sbullet})\To{s}\Gamma_Z(F^{\sbullet})$ as desired.
\end{proof}
\end{comment}

We shall also need the following result which is contained in the proof of \cite[Prop.\,2.2]{MF-S12}:

\begin{proposition}\label{P:1}
Let $F$ belong to $\rD^\rb(\pOS)$ and assume that for every $(x_o,s_o)\in \XS$ and for every $j$, $\shh^j(F)_{(x_o,s_o)}$ is finitely generated over ${\sho_S}_{s_o}$. Assume that, for a given~$j$, $\shh^j(Li^*_{s_o} (F))=0$ for any $s_o\in S$. Then $\shh^j(F)=0$. In particular, if, for a given integer $k$ and for every $s_o$, $Li^*_{s_o} (F)\in \rD^{\geq k}(X)$ (respectively $Li^*_{s_o} (F)\in \rD^{\leq k}(X)$), then $F\in \rD^{\geq k}(\XS)$ (respectively $F\in \rD^{\leq k}(\XS)$).
\end{proposition}

\subsection{$S$-$\C$-constructibility and perversity}\label{subsec:Sconstructible}

We refer to the appendix for the notion of \emph{$S$-locally constant} sheaf on $\XS$. We have defined in \cite{MF-S12} the categories of $S$-$\R$-constructible sheaves (\resp $S$-$\C$-constructible sheaves) and the corresponding derived categories $\rD^\rb_\rc(\pOS)$ (\resp $\rD^\rb_\cc(\pOS)$). For an object $F$ of $\rD^\rb(\pOS)$, the condition that it is an object of $\rD^\rb_\rc(\pOS)$ is a local property on $X$, since it is characterized by the property that the microsupport of $F$ is contained in $\Lambda\times(T^*S)$ for some closed $\RR_+^*$-conic Lagrangean subanalytic subset $\Lambda$ of $T^*X$. Similarly, the condition that it is an object of $\rD^\rb_\cc(\pOS)$ is local, since it consists in adding that the microsupport is $\CC^\times$-conic (\cf \cite[Prop.\,2.5 \& Def.\,2.19]{MF-S12}).

The category $\rD^\rb_\cc(\pOS)$ is endowed with a natural t-structure, which however is not preserved by duality, as seen by considering a skyscraper $\pOS$-module on $\XS$. In this article, the adjective ``perverse'' refers to this t-structure. Recall (\cf \cite[Lem.\,2.5]{MF-S12}) that the category $\pD^{\leq0}_\cc(\pOS)$ \resp $\pD^{\geq0}_\cc(\pOS)$ can be defined as the full subcategory of $\rD^\rb_\cc(\pOS)$ whose objects are the $S$-$\C$-constructible bounded complexes~$F$ such that, for some adapted $\mu$-stratification $(X_\alpha)$, denoting by \hbox{$i_\alpha:X_\alpha\hto X$} the inclusion,
\begin{align*}
\forall\alpha\text{ and }\forall j>-\dim(X_\alpha),\quad\shh^j(i^{-1}_\alpha F)&=0,\\
\tag*{\resp}
\forall \alpha\text{ and }\forall j<-\dim(X_\alpha),\quad\shh^j(i^{!}_\alpha F)&=0.
\end{align*}
As usual, an object $F$ of $\rD^\rb_\cc(\pOS)$ is called \emph{perverse} if it is an object of both $\pD^{\leq0}_\cc(\pOS)$ and $\pD^{\geq0}_\cc(\pOS)$. There is a natural duality functor on $\rD^\rb_\rc(\pOS)$ and on $\rD^\rb_\cc(\pOS)$ (\cf\cite[Prop.\,2.23]{MF-S12}), but in general it does not exchange $\pD^{\leq0}_\cc(\pOS)$ and $\pD^{\geq0}_\cc(\pOS)$ and therefore does not preserve the heart of the t-structure.

\begin{lemma}\label{perv}
For a given object $F\in \pD^\rb_\cc(\pOS)$,
$F$ and $\bD F$ are perverse if and only if for all $s_o\in S$, $Li^*_{s_o}(F)$ is perverse regarded as an object of $\rD^\rb_\cc(\CC_X)$.
\end{lemma}

\begin{proof}
If $F\in \pD^{\leq 0}_\cc(\pOS)$ then $Li^*_{s_o}(F)\in \pD^{\leq 0}_\cc(\CC_X)$ and the converse holds by Proposition \ref{P:1}. The assertion then follows by \cite[Prop.\,2.28]{MF-S12}.
\end{proof}

\begin{proof}[Proof of Proposition \ref{prop:FDFperverse}]
For the first statement we note that, according to the assumption and the definition of t-structure, when $X_\alpha$ is an open stratum, $i^{-1}_{\alpha} F$ is concentrated in degree $-d_X$ and $i^{-1}_{\alpha}\shh^{-d_X}F$ is a $\pOS$-coherent module. On the other hand, according to Lemma~\ref{perv}, for any $s_o\in S$, $Li^*_{s_o} F$ is perverse, hence it is concentrated is degrees $\geq -d_X$. Recall that a coherent $\sho_S$-module $F_{x_o}$ is locally free if and only if $Li^*_{s_o} F_{x_o}$ is concentrated in degree zero for every $s_o\in S$. It follows that $i^{-1}_{\alpha}\shh^{-d_X} F$ is locally free.

Conversely, since $Li^*_{s_o}$ commutes with $j_!$, Lemma \ref{perv} implies that $j_!F[d_X]$ and its dual are perverse. On the other hand, we have
$$\bD(j_!F[d_X])\simeq R j_{\ast} R\shhom_{\pOS}(F,\pOS)[d_X].$$
Since $F$ is locally free, $\bD'(F):=R\shhom_{\pOS}(F,\pOS)$ is concentrated in degree zero and $\shh^0 \bD'(F)$ is locally free. Thus the statement follows by biduality (\cf \cite[Prop.\,2.23]{MF-S12}).
\end{proof}

\subsection{Coherent \texorpdfstring{$\DXS$}{DXS/S}-modules}
We will use a notation similar to that of \cite{Kashiwara03} for the functors on $\cD$-modules, namely $\Df_*$ denotes the pushforward by a map $f$, $\Df_!$ the ``proper pushforward'' and $\Df^*$ the pull-back (they are denoted respectively by $\DD f_*$, $\DD f_!$ and $\DD f^*$ in \cite{Kashiwara03}, but we try to avoid confusion with the duality functor).

Let $i:Z\hto X$ be the inclusion of a closed submanifold in $X$. The following adaptation of Kashiwara's result (\cf\eg \cite[\S4.8]{Kashiwara03}) is straightforward.

\begin{theorem}[Kashiwara's equivalence]\label{th:kashiwaraequivalence}
The pushforward functor $\Di_*$ induces an equivalence between the category of coherent $\DZS$-modules and that of coherent $\DXS$-modules supported on $Z\times S$. A quasi-inverse functor is $\shh^{-\codim Z}\Di^*$, and $\shh^j \Di^*\!=\!0$ for $j\neq-\codim Z$ on objects of the latter category.
\end{theorem}

The behaviour of coherence by pushforward with respect to the parameter space is obtained in the following proposition. Let $\pi:S\to S'$ be a morphism of complex manifolds. Let $\shm$ be a coherent $\DXS$-module which is $\pi$-good, that is, by definition, such that for any point $(x,s')\in\XS'$ there exists a neighborhood $U\times V'$ of $(x,s')$ such that $\shm_{U\times\pi^{-1}(V')}$ has a good filtration $F_\sbullet\shm$. The proof of the following proposition is similar to that given in \cite[\S4.7]{Kashiwara03}.

\begin{proposition}\label{prop:pushforward}
Assume that $\pi$ is proper and that $\shm$ is $\pi$-good. Then $\bR\pi_*\shm\in\rD^\rb_\coh(\DXSp)$. Moreover, if $\Char\shm\subset\Lambda\times S$ with $\Lambda\subset T^*X$, then for each $k\in\NN$, $\Char R^k\pi_*\shm\subset\Lambda\times S'$. In particular, if moreover $\shm$ is holonomic, then $\bR\pi_*\shm\in\rD^\rb_\hol(\DXSp)$.
\end{proposition}

\begin{comment}
\begin{proof}
Since the problem is local on $\XS'$, we can assume that $\shm$ has a good filtration. Set $F_pR^k\pi_*\shm:=\text{image}[R^k\pi_*F_p\shm\to R^k\pi_*\shm]$. Then, by considering the spectral sequence for the pushforward and usual coherence arguments, the spectral sequence degenerates at a finite step and $\gr^FR^k\pi_*\shm$ is a subquotient of $R^k\pi_*\gr^F\shm$. One concludes that it is $\gr^F\DXSp$-coherent, hence the $\DXSp$-coherence of $R^k\pi_*\shm$. The assertion on the characteristic variety is obtained with the same argument.
\end{proof}
\end{comment}

There is a duality functor $\bD:\rD^\rb(\DXS)\mto\rD^\rb(\DXS)$ defined by
\[
\bD\shm=\Rhom_{\DXS}(\shm,\DXS\otimes\Omega^{\otimes-1}_{\XS/S})[d_X],
\]
and we set $\bD'(\cbbullet)=\bD(\cbbullet)[-d_X]$.

\begin{proposition}[{\cite[Prop.\,5.10, Th.\,5.15]{Sch-Sch94}}]\label{P:rel}
Let $f:X\to Y$ be a proper morphism of complex manifolds. Then there exists a morphism $\Df_!\bD(\shm)\to \bD(\Df_!\shm)$ in $\rD^\rb(\DYS)^{\op}$, which is functorial with respect to $\shm\in\rD^\rb(\DXS)$. It is an isomorphism for $\shm$ in $\rD^\rb_{f\textup{-good}}(\DXS)$.
\end{proposition}

As a consequence, using the projection formula for sheaves and replacing $\shm$ by $\bD'(\shm)$ we recover the relative version of \cite[Th.\,4.33]{Kashiwara03}.

\begin{corollary}[Adjunction formula]\label{C:rel3}
Let $f:X\to Y$ be a proper morphism of complex manifolds. For $\shm\in\rD^\rb_{f\textup{-good}}(\DXS)$ and $\shn\in \rD^\rb(\DYS)$, there exists a canonical morphism $Rf_*\Rhom_{\DXS}(\shm,\Df^*\shn)[d_X]\to \Rhom_{\DYS}(\Df_!\shm, \shn)[d_Y]$ which is an isomorphism.
\end{corollary}

\subsection{Holonomic \texorpdfstring{$\DXS$}{DXS}-modules}

The notion of holonomic $\DXS$-module has been recalled in the introduction. We refer to \cite{MF-S12} for details on some of their properties. Recall (\cf Introduction) that, for such a $\DXS$-module, we set $\pDR\shm:=\DR\shm[d_X]$ and $\pSol\shm=\Sol\shm[d_X]$, and that $\pDR\shm\simeq\bD\pSol\shm$.

\begin{proposition}\label{prop:holLi*0}
Let $\shm$ be a holonomic $\DXS$-module and let $\shm_{(x_o,s_o)}$ be its germ at $(x_o,s_o)\in \XS$.
\begin{enumerate}
\item\label{prop:holLi*01}
Assume that $L i^*_{s_o}\shm=0$ for each $s_o\in S$. Then $\shm=0$.
\item\label{prop:holLi*02}
Let $\cI_{s_o}$ be an ideal of $\cO_{S,s_o}$ contained in the maximal ideal $\mathfrak{m}_{s_o}$. Assume that $\cI_{s_o}\shm_{(x_o,s_o)}=\shm_{(x_o,s_o)}$. Then $\shm_{(x_o,s_o)}=\nobreak0$.
\end{enumerate}
\end{proposition}

\begin{proof}
Assume $\shm\neq0$ and let $\Lambda\subset T^*X$ be a closed conic complex Lagrangian variety such that $\Char\shm\subset\Lambda\times S$. Let $(X_\alpha)_\alpha$ be a $\mu$-stratification of $X$ compatible with~$\Lambda$. We will argue by induction on $\max_\beta\dim X_\beta$, where $X_\beta$ runs among the strata included in the projection of $\Lambda$ in $X$.

Assume first that some stratum $X_\beta$ is open in $X$ and let $x_o\in X_\beta$, so that $\Char\shm\subset(T^*_XX)\times S$ in the neighborhood of $x_o$. It follows that $\shm$ is $\cO_{\XS}$-coherent in the neighborhood of $x_o\times S$, and Assumption \ref{prop:holLi*0}\eqref{prop:holLi*01} \resp \eqref{prop:holLi*02} implies, according to Nakayama, that $\shm_{|x_o\times S}=0$ \resp $\shm_{(x_o,s_o)}=0$.

We are thus reduced to the case where no stratum $X_\beta$ is open. Choose then a maximal stratum $X_\beta$. By applying Kashiwara's equivalence \ref{th:kashiwaraequivalence}, which commutes with the $\sho_S$-action, in the neighborhood of any point of $X_\beta$ we are reduced to the previous case. By induction on the dimension of the maximal strata, we conclude that $\Lambda$ can be chosen empty, hence $\shm=0$.
\end{proof}

\begin{corollary}
Let $\shm$ be an object of $\rD^\rb_\hol(\DXS)$. Assume $L i^*_{s_o}\shm=0$ for each $s_o\in S$. Then $\shm=0$.
\end{corollary}

\begin{proof}
We first prove that if $L i^*_\Sigma\shm=0$ for every codimension-one germ of submanifold $\Sigma$, then $\shm=0$. Let $\sigma$ be a local equation of $\Sigma$. Then the assumption is that the cone $C_\Sigma(\shm)$ of $\sigma:\shm\to\shm$ is isomorphic to zero. From the long exact sequence
\[
\cdots\to\shh^j\shm\To{\sigma}\shh^j\shm\to\shh^jC_\Sigma(\shm)=0\to\cdots
\]
we conclude as in Proposition \ref{prop:holLi*0} that $\shh^j\shm=0$. We argue now by induction on~$d_S$. In the case $d_S=1$, every point has codimension one, so there is nothing more to prove. In general, for every germ of hypersurface $\Sigma$ and every $s_o\in \Sigma$, we have $Li^*_{s_o}Li_\Sigma^*\shm\simeq Li^*_{s_o}\shm=0$, so $Li_\Sigma^*\shm=0$ by induction, and the first part of the proof gives the desired assertion.
\end{proof}

\begin{corollary}\label{cor:Hjnul}
Let $\shm$ be an object of $\rD^\rb_\hol(\DXS)$. Assume that $\shh^j L i_{s_o}^*\shm=0$ for all $j\neq0$ and all $s_o\in S$. Then $\shh^j\shm=0$ for all $j\neq0$ and $L i^*_{s_o}\shh^0\shm=\shh^0 L i^*_{s_o}\shm$ for all $s_o\in S$.
\end{corollary}

\begin{proof}
We prove it by induction on $d_S$. Let $s$ be part of a local coordinate system centered at some $s_o\in S$ and denote by $i_{S'}:S'=\{s=0\}\hto S$ the inclusion. Then $ L i_{S'}^*\shm$ is an object of $\rD^\rb_\hol(\DXSp)$ and by induction it is concentrated in degree zero. Considering the long exact sequence
\[
\cdots\shh^j\shm\To s\shh^j\shm\to\shh^j L i_{S'}^*\shm\to\cdots
\]
one obtains that $s:\shh^j\shm\to\shh^j\shm$ is onto for $j\neq0$. According Proposition \ref{prop:holLi*0}\eqref{prop:holLi*02}, we have $\shh^j\shm=0$ for $j\neq0$. The remaining statement is clear.
\end{proof}

\begin{corollary}\label{cor:holstrict}
Let $\shm$ be a strict holonomic $\DXS$-module. Then $\shh^j\bD\shm=0$ for $j\neq0$ and $\shh^0\bD\shm$ is a strict holonomic $\DXS$-module.
\end{corollary}

\begin{proof}
By the strictness property, $\shh^jL i_{s_o}^*\shm=0$ for $j\neq0$, hence $L i_{s_o}^*\shm=\shh^0 L i_{s_o}^*\shm$ is a holonomic $\cD_X$-module, so $\shh^j\bD L i_{s_o}^*\shm=0$ for $j\neq0$ and $\shh^0\bD L i_{s_o}^*\shm$ is $\cD_X$-holonomic. Recall also (\cf \cite[Prop.\,3.1]{MF-S12}) that $L i_{s_o}^*\bD\shm\simeq \bD L i_{s_o}^*\shm$ for any $s_o\in S$. Therefore, $\shh^jL i_{s_o}^*\bD\shm=0$ for $j\neq0$. According to \hbox{\cite[Prop.\,2.5]{Sch-Sch94}}, $\bD\shm$ has holonomic cohomology. As a consequence, according to Corollary~\ref{cor:Hjnul}, $\shh^j\bD\shm=0$ for $j\neq0$, $\shh^0\bD\shm$ is $\DXS$-holonomic, and $L i^*_{s_o}\shh^0\bD\shm$ has cohomology in degree zero at most for any $s_o\in S$, since
\[
L i^*_{s_o}\shh^0\bD\shm\simeq \shh^0 L i^*_{s_o}\bD\shm\simeq\shh^0\bD L i^*_{s_o}\shm
\]
and $L i^*_{s_o}\shm$ is a holonomic $\cD_X$-module. The conclusion follows from Lemma \ref{lem:flatness} below.
\end{proof}

\begin{lemma}\label{lem:flatness}
Let $\shm$ be a coherent $\DXS$-module. Then $\shm$ is strict if and only if $\shh^j L i^*_{s_o}\shm=0$ for each $s_o\in S$ and each $j\neq0$.
\end{lemma}

\begin{proof}
The ``only if'' part is clear. The ``if'' part is clear if $d_S=1$, since strictness is then equivalent to the absence of $\sho_S$-torsion. In general, assume $\shm$ is not strict. Recall (\cf \cite[Cor.\,A.0.2]{Bibi86I}) that there exists then a morphism $\pi:S_1\to S$ from a smooth curve $S_1$ to $S$ such that $\pi^*\shm$ has $\cO_{S_1}$-torsion. Let $s_1$ be a local coordinate on $S_1$ such that $\shk:=\ker[s_1:\pi^*\shm\to \pi^*\shm]\neq0$. Let $i$ denote the composition $\{s_1=0\}\hto S^1\hto S$. The exact sequence
\[
\cdots\to\shh^{-1}L i^*\shm\to\shh^0 L \pi^*\shm\To{s_1}\shh^0L \pi^*\shm\to\shh^0L i^*\shm\to0
\]
show that $\shh^{-1}L i^*\shm$ surjects onto $\shk$, hence is nonzero.
\end{proof}

\begin{corollary}\label{cor-perv}
Let $\shm$ be a strict holonomic $\DXS$-module and set $\cF=\pDR\shm$ or $\pSol\shm$. Then $\cF$ and $\bD\cF$ are $S$-perverse.
\end{corollary}

\begin{proof}
According to Corollary \ref{cor:holstrict}, this follows from Th.\,1.2, Prop.\,3.10, Th.\,3.11 and Prop.\,2.28 in \cite{MF-S12}.
\end{proof}

\begin{proof}[Proof of Proposition \ref{P:3.3}]\mbox{}\par
$(1)\Leftrightarrow (2)$ follows from Corollary \ref{cor:holstrict}.

$(1)\Rightarrow (3)$ follows from Corollary \ref{cor-perv}.

$(1)\Leftarrow (3)$: According to Lemma \ref {perv}, for any $s_o\in S$, $Li^*_{s_o}\pSol(\shm)$ is perverse, therefore, for any $s_o\in S$, $Li^*_{s_o}(\shm)$ is concentrated in degree $0$ and $\shh^0 Li^*_{s_o}(\shm)$ is holonomic. The result then follows by Corollary \ref{cor:Hjnul} and Lemma \ref{lem:flatness}.
\end{proof}

We now can make precise the behaviour of the functors $\pDR$ and $\pSol$. According to Lemma \ref{L:s}, given $\shm\!\in\!\rD^{\geq0} (\DXS)$ and $s_o\in S$, we have $Li^*_{s_o}(\shm)\!\in\! \rD^{\geq-d_S}(\shd_X)$. If moreover, $\shh^k(\shm)$ is $\pOS$-flat for every $k$, then $Li^*_{s_o}(\shm)\in \rD^{\geq 0}(\shd_X)$.

\begin{proposition}\label{P:10}
The functor $\pDR$ has the following behaviour when considering the standard t-structure on $\rD^\rb_\hol(\DXS)$ and the t-structure given on $\rD^\rb_\cc(\pOS)$:
\begin{enumerate}
\item
Let $\shm\in \rD^{\leq0}_\hol(\DXS)$. Then $\pDR\shm\in \pD^{\leq0}_\cc(\pOS)$.

\item
Let $\shm\in \rD^{\geq0}_\hol(\DXS)$. Then $\pDR\shm\in \pD^{\geq-d_S}_\cc(\pOS)$. Moreover, if for any $k$, $\shh^k(\shm)$ is $\pOS$-flat then $\pDR\shm\in \pD^{\geq 0}_\cc(\pOS)$.
\end{enumerate}
\end{proposition}

\begin{proof}\mbox{}
\begin{enumerate}
\item
The assumption entails that, for any $s_o\in S$, $Li^*_{s_o}(\shm)\in \rD^{\leq0}_\hol(\shd_{X})$. Therefore, $\pDR(Li^*_{s_o}(\shm))\simeq Li^*_{s_o}(\pDR\shm)\in \pD^{\leq0}_\cc(X)$. As a consequence,
$$Li^*_{s_o} (\pDR\shm|_{X_{\alpha}\times S})\simeq \pDR(Li^*_{s_o}(\shm))|_{X_{\alpha}}\in \rD^{\leq -\dim X_{\alpha}}(X_{\alpha}).$$ The statement then follows by Proposition \ref{P:1}.

\item
We have to prove that $R\Gamma_{X_{\alpha}\times S}(\pDR \shm)\in \rD^{\geq -\dim X_{\alpha}-d_S}(\XS)$. By Lemma \ref{RGamma}, we have, for any $s_o\in S$,
$$Li^*_{s_o} R\Gamma_{X_{\alpha}\times S}(\pDR\shm)\simeq R\Gamma_{X_{\alpha}\times \{s\}}(Li^*_{s_o}(\pDR\shm))$$
On the other hand $Li^*_{s_o}(\shm)\in \rD^{\geq-d_S}_\hol(\shd_X)$ so $\pDR Li^*_{s_o}(\shm)\in
\pD^{\geq-d_S}_\cc(X)$. Therefore
$$R\Gamma_{X_{\alpha}}(Li^*_{s_o}(F))\in \rD^{\geq-\dim X_{\alpha}-d_S}(X)$$
and the statement follows again by Proposition \ref{P:1}. The same argument implies the second part of the statement since, by Lemma \ref{L:s}, when $\shh^k(\shm)$ is strict, for each~$k$, $Li^*_{s_o}(\shm)\in \rD^{\geq 0}_\hol(\shd_X)$ for any~$s$.\qedhere
\end{enumerate}
\end{proof}

By the same arguments of Proposition \ref{P:10} , we obtain:

\begin{proposition}\label{P:12}
The functor $\pSol$ satisfies the following:
\begin{enumerate}
\item
Let $\shm\in \rD^{\leq0}_\hol(\DXS)$. Then $\pSol\shm\in \pD^{\geq0}_\cc(\pOS)$.

\item
Let $\shm\in \rD^{\geq0}_\hol(\DXS)$. Then $\pSol\shm\in \pD^{\leq d_S}_\cc(\pOS)$. Moreover, if for any~$k$, $\shh^k(\shm)$ is strict then $\pSol\shm\in \pD^{\leq 0}_\cc(\pOS)$.
\end{enumerate}
\end{proposition}

Simple examples show that the final statement in Proposition \ref{P:10} does not hold in general in the non strict case.

\begin{theorem}[Proper pushforward]\label{T:hol}
Let $f:X\to Y$ be a proper morphism of complex manifolds and let $\shm$ belong to $\rD^\rb_\hol(\DXS)$ and has $f$-good cohomology. Then
\begin{enumeratea}
\item\label{T:hol1}
the pushforward $\Df_*\shm\!:=\!Rf_*(\shd_{Y\leftarrow X/S}\otimes^L_{\DXS}\shm)$ belongs to $\rD^\rb_\hol(\DYS)$,
\item\label{T:hol2}
the de~Rham complex satisfies $\pDR\Df_*\shm\simeq Rf_*\pDR\shm$ functorially in $\shm$,
\item\label{T:hol3}
the solution complex satisfies $\pSol\Df_*\shm\simeq Rf_*\pSol\shm$ functorially in $\shm$.
\end{enumeratea}
\end{theorem}

\begin{proof}\mbox{}
\begin{enumeratea}
\item
The coherence and the holonomicity of the cohomology groups of $\Df_*\shm$ follow from \cite[Th.\,4.2 and Cor.\,4.3]{Sch-Sch94}.
\item
The proof for $\cD_X$-modules applies in a straightforward way and does not use holonomicity nor coherence of $\shm$ and neither properness of $f$ (\cf \eg \cite[Th.\,4.2.5]{H-T-T08}).
\item
We have
\[
\pSol\Df_*\shm\overset{(1)}\simeq\pDR\bD\Df_*\shm\overset{(2)}\simeq\pDR\Df_*\bD\shm\overset{(3)}\simeq Rf_*\pDR\bD\shm\overset{(4)}\simeq Rf_*\pSol\shm,
\]
where (1) and (4) follow from \cite[(5)]{MF-S12}, (2) is given by \cite[Th.\,5.15]{Sch-Sch94} and (3) follows from \eqref{T:hol2} above.\qedhere
\end{enumeratea}
\end{proof}

\section{Regular holonomic \texorpdfstring{$\DXS$}{DXS}-modules}

\Subsection{Regularity}\label{subsec:regularity}
As usual $\rD^\rb_\rhol(\shd_X)$ denotes the bounded derived category of complexes with regular holonomic cohomology sheaves (for an introduction to regular holonomic $\shd$\nobreakdash-modules in the absolute case, we refer to \cite{Kashiwara03,Mebkhout87,Mebkhout04,H-T-T08}, and to the references therein for details).

\begin{definition}[Regular holonomic $\DXS$-module]\label{D:HR}
Let $\shm$ be a holonomic $\DXS$-module. We say that $\shm$ is regular holonomic if, for any $s_o\in S$, $Li^*_{s_o}\shm$ belongs to $\rD^\rb_{\rhol}(\shd_X)$.
\end{definition}

We similarly denote by $\rD^\rb_\rhol(\DXS)$ the bounded derived category of complexes with regular holonomic cohomology sheaves. Note that, given an exact sequence in $\Mod_\hol(\DXS)$,
$$0 \to \shn \to\shm\to\shl \to 0$$ if two of its terms are regular holonomic, then the third one is also regular holonomic.

\begin{proposition}\label{Rhol}
Given a distinguished triangle in $\rD^\rb_{\hol}(\DXS)$,
$$\shn \to\shm\to\shl\To {+1}$$ if two of its terms belong to $\rD^\rb_\rhol(\DXS)$ the same holds for the third.
\end{proposition}

\begin{proposition}\label{P:15}
Assume that $\shm$ is an object of $\rD^\rb_{\rhol}(\DXS)$. Then so is $\bD\shm$.
\end{proposition}

\begin{proof}
Since the functors $\bD$ and $Li^*_{s_o}$ commute, the result follows by the definition and the fact that $\rD^\rb_{\rhol}(\shd_X)$ is stable by duality.
\end{proof}

\begin{comment}
\begin{proposition}\label {Nonchar2}
Let $\shm\in\rD^\rb_{\rhol}(\DXS)$ and let $Y$ be a closed submanifold of $X$. Let~$i_Y$ denote the embedding $Y\times S\to \XS$. Let us assume that $\Char(\shm)\cap T^*_Y\XS\subset T^*_X\XS$. Then $\Di^*_{Y\times S}\shm$ has regular holonomic cohomologies over $\DYS$.
\end{proposition}

\begin{proof}
Apply Proposition \ref{Nonchar} and the commutativity of $Li^*_{s_o}$ with $ f^*$ and $\otimes$, for any $s_o\in S$.
\end{proof}
\end{comment}

\begin{corollary}[of Theorem \ref{T:hol}]\label{T:rholim}
Let $f:X\to Y$ be a proper morphism of complex manifolds.
Let $\shm$ belong to $\rD^\rb_\rhol(\DXS)$. Then $Rf_*(\shd_{Y\from X/S}\otimes^L_{\DXS}\shm)$ belongs to $\rD^\rb_\rhol(\DYS)$.
\end{corollary}

\begin{proof}
The regularity follows from the commutativity of $Li^*_{s_o}$ with $R f_*$ and $\otimes$, for any $s_o\in S$.
\end{proof}

\subsection{Deligne extension of an $S$-locally constant sheaf}\label{subsec:Deligneextension}

Let $D$ be a normal crossing divisor in $X$ and let $j:X^*:=X\moins D\hto X$ denote the inclusion. Let $\cF$ be a coherent $S$-locally constant sheaf on $\XsS$ and let $(E,\nabla)=(\cO_{\XsS}\otimes_{\pOS}\nobreak\cF,\rd_X)$ be the associated coherent $\cO_{\XsS}$-module with flat relative connection (\cf Remark \ref{rem:RH}). In particular, $E$ is naturally endowed with the structure of a left $\cD_{\XsS/S}$-module and $j_*E$ with that of a $\DXS$-module. There exists a coherent $\sho_S$-module~$\cG$ such that, locally on $\XsS$, $\cF\simeq p^{-1}\cG$ (\cf Proposition \ref{prop:cG}). More precisely, let~$U$ be any contractible open set of $X^*$; then $\cF_{|\US}\simeq p_U^{-1}\cG$ (\cf Proposition \ref{prop:contractible}).

Let $\varpi:\wt X\to X$ denote the real blowing up of $X$ along the components of~$D$. Denote by $\wtj:X^*\hto\wt X$ the inclusion, so that $j=\varpi\circ\wtj$. Let $x_o\in D$, $\wt x_o\in\varpi^{-1}(x_o)$ and let $s_o\in S$. Choose local coordinates $(x_1,\dots,x_\ell,x'_{\ell+1},\dots,x'_n)$ at $x_o$ such that $D=\{x_1\cdots x_\ell=0\}$ and consider the associated polar coordinates $(\rhog,\thetag,\boldsymbol{x}'):=(\rho_1,\theta_1,\dots,\rho_\ell,\theta_\ell,x'_{\ell+1},\dots,x'_n)$ so that $\wt x_o$ has coordinates $\rhog^o=0$, $\thetag^o$, $\boldsymbol{x}^{\prime o}=0$. For $\epsilon>0$, we set
\[
\wt U_\epsilon:=\{\Vert\rhog\Vert<\epsilon,\Vert\boldsymbol{x}'\Vert<\epsilon,\Vert\thetag-\thetag^o\Vert<\epsilon\},\quad \wt U^*_\epsilon:=\wt U_\epsilon\moins\{\rho_1\cdots\rho_\ell=0\}.
\]
On the other hand, for $s_o\in S$, we denote by $V$ some open neighborhood of $s_o$ in $S$. Note that since $\wt U^*_\epsilon$ is contractible, we have $\cF_{|\wt U^*_\epsilon\times S}\simeq p^{-1}_{\wt U^*_\epsilon}\cG$ locally on $S$ (\cf Proposition \ref{prop:contractible}), and thus $(E,\nabla)_{|\wt U^*_\epsilon\times S}\simeq (\cO_{\wt U^*_\epsilon\times S}\otimes_{\pOS}p^{-1}\cG,\rd_X)$.

\begin{definition}[Moderate growth]\label{def:mod}\mbox{}
\begin{enumeratea}
\item
A germ of section $\wt v\in(\wtj_*E)_{(\wt x_o,s_o)}$ is said to have \emph{moderate growth} if for some (or any) system of generator of $\cG_{s_o}$, some $\epsilon>0$ and some $V$ so that $\wt v$ is defined on $\wt U_\epsilon\times V$, its coefficients on the chosen generators of~$1\otimes\cG_{s_o}$ (these are sections of $\cO(\wt U^*_\epsilon\times V)$ by means of the isomorphism above) are bounded by $C\rhog^{-N}$, for some $C,N>0$.

\item
A germ of section $v\in(j_*E)_{(x_o,s_o)}$ is said to have \emph{moderate growth} if for each~$\wt x_o$ in $\varpi^{-1}(x_o)$, the corresponding germ in $(\wtj_*E)_{(\wt x_o,s_o)}$ has moderate growth.
\end{enumeratea}
\end{definition}

\begin{theorem}\label{th:Deligneext}
The subsheaf $\wt E$ of $j_*E$ consisting of local sections having moderate growth is stable by $\nabla$ and is $\cO_{\XS}(*D)$-coherent.
\end{theorem}

\begin{proof}
The problem is local on $\XS$. We thus assume that $\XS$ is a small neighborhood of $(x_o,s_o)$ as above. In such a neighborhood, giving the local system is equivalent to giving $T_1,\dots,T_\ell \in \Aut(\cG)$ which pairwise commute, according to Proposition \ref{prop:monodromyrepr}. According to \cite[Cor.\,2.3.10 \& (3.45)]{Wang08}, in the neighborhood of $s_o$ there exists for each $k$ a logarithm of $T_k$, hence there exists $A_k\in\End(\cG)$ such that $T_k=\exp(-2\pi i A_k)$, and the formula (2.11) of \cite{Wang08} can be used to show that there exist $A_1,\dots,A_\ell\in\End(\cG)$ which pairwise commute. Set $\wh E_1:=\cO_{\XS}(*D)\otimes_{\pOS}\cG$, equipped with the connection~$\nabla$ such that $\nabla_{x_k\partial_{x_k}}:\wh E_1\to\wh E_1$ is given by $A_k$ if $k=1,\dots,\ell$, and zero otherwise. Set $E_1=\wh E_{1|\XsS}$. Then the monodromy representation of $\nabla$ on $E_1$ is given by $T_1,\dots,T_\ell$, from which one deduces an isomorphism $(E,\nabla)\simeq(E_1,\nabla)$, according to Proposition \ref{prop:monodromyrepr} and Remark \ref{rem:RH}. It is then enough to show that $\wt E_1=\wh E_1$, where the former is as in the statement of the theorem, since we clearly have $\wt E\simeq\wt E_1$.

Let us fix local generators $\bg:=(g_i)$ of $\cG$ and let us still denote by $A_k$ a matrix of the endomorphism $A_k$ with respect to $(g_i)$. Any local section $v$ of $E_1$ can be expressed as $v=(1\otimes\bg)\cdot f$ for some vector $f$ of local holomorphic functions. Let us set $x^A:=x_1^{A_1}\cdots x_\ell^{A_\ell}$. A family of $\nabla$-horizontal generators is then given by $(1\otimes\bg)\cdot x^{-A}$, showing that the generators $1\otimes g_i$ of $\wh E_1$ have moderate growth, hence are local sections of $\wt E_1$. Therefore, sections of $j_*E_1$ have moderate growth if and only if their coefficients over the generators $1\otimes g_i$ have moderate growth. Since these coefficients are sections of $\cO_{\XsS}$, they must be meromorphic. Hence $\wt E_1=\wh E_1$.
\end{proof}

\begin{remark}\label{rem:tildeexactfunctor}
The functors $j_*$ \resp $\wtj_*$ are exact functors from the category of coherent $\cO_{\XsS}$-modules with integrable relative connection to that of $\cO_{\XS}$- \resp $\wtj_*\cO_{\XsS}$-modules with integrable relative connection, since any point of $D$ \resp $\varpi^{-1}(D)$ has a fundamental system of neighborhoods whose intersection with $X\moins D$ \resp $\wt X\moins\nobreak\varpi^{-1}(D)$ is Stein. Similarly, the correspondence $E\mto\wt E$ is an exact functor from the category of coherent $\cO_{\XsS}$-modules with integrable relative connection to that of $\varpi^{-1}\cO_{\XS}$-modules with integrable relative connection. Indeed, given a morphism $\varphi:(E,\nabla)\to(E',\nabla)$, it is clear that the morphism $\wtj_*\varphi$ sends~$\wt E$ to~$\wt E'$. The only point to check is right exactness, that is, that the induced morphism~$\wt\varphi$ is onto as soon as $\varphi$ is onto. Keeping the notation of the proof of Theorem~\ref{th:Deligneext}, we can start with $E_1,E'_1$. The surjectivity of $\varphi$ is equivalent to that of the morphism between the corresponding relative local systems, and restricting to~$x_o$, to the induced morphism $G\to G'$. As a consequence, the morphism $\wh\varphi$ is onto, and therefore so is $\wt\varphi$, according to the identification proved in the theorem.
\end{remark}

\begin{corollary}\label{cor:Deligneext}
Assume moreover that $d_S=1$. Then $\wt E$ is $\DXS$-holonomic and regular with characteristic variety contained in $\Lambda\times S$, where $\Lambda$ is the union of the conormal spaces of the natural stratification of $(X,D)$. Moreover, if $F$ is $\pOS$-locally free, then $\wt E$ is strict.
\end{corollary}

\begin{proof}
We keep the notation of the proof of Theorem \ref{th:Deligneext}. Since the statement is local on $\XS$, we can work with $E_1$. We fix $s_o\in S$ and we take a local coordinate~$s$ centered at $s_o$. We denote by $(S,s_o)$ the germ of~$S$ at~$s_o$.

\subsubsection*{Step one: assume $G$ is $\cO_S$-locally free}
In this case, we are reduced to proving that~$\wh E_1$ is strict holonomic with characteristic variety contained in $\Lambda\times S$, since its restriction to any~$s_o$ is a $\cD_X$-module of Deligne type, hence is regular holonomic. Note that the strictness of $\wh E_1$ is obvious. We regard~$\cG_{s_o}$ as an $\cO_{S,s_o}[A_1,\dots,A_\ell]$-module which is $\cO_{S,s_o}$-free. Let $\rho:(S',s'_o)\to (S,s_o)$ be the finite ramification of order $N$. Then $\cG_{s_o}$ is identified with the invariant part of the pull-back $\cG'_{s'_o}$ of $\cG_{s_o}$ by the Galois group~$\ZZ^N$. Similarly, $\wh E_1$ is identified with the invariant part of $\rho_*\rho^*\wh E_1$ by the Galois group. The assertions of the corollary hold then for $\wh E_1$ if they hold for $\wh E'_1:=\rho^*\wh E_1$, according to Proposition \ref{prop:pushforward}.

\begin{lemma}
There exists a finite ramification $\rho:(S',s'_o)\to (S,s_o)$ such that the pull-back $\cG'_{s'_o}$ of $\cG_{s_o}$ has a finite filtration by $\cO_{S',s'_o}[A_1,\dots,A_\ell]$-submodules whose successive quotients have rank one over $\cO_{S',s'_o}$.
\end{lemma}

By the lemma and the previous remarks, we are reduced to the case where $G$ has rank one over $\cO_S$, so $A_k(s)$ is a holomorphic function $\lambda_k(s)$ and $\wh E_1=\cO_{\XS}(*D)$ endowed with the $\DXS$-action defined by $\partial_{x_k}\cdot 1=\lambda_k(s)/x_k$ if $k=1,\dots,\ell$, and $\partial_{x_k}\cdot 1=0$ otherwise. This is the external product over $\cO_S$ of the $\cD_{\CC\times S/S}$-modules $\cO_{\CC\times S}(*0)$ endowed with the connection $x_k\partial_{x_k}-\lambda_k(s)$ ($k=1,\dots,\ell$) or with the connection $\partial_{x_k}$ ($k>\ell$). It is therefore enough to to show the corollary in the case $d_X=1$. The latter case being obvious, we are reduced to proving the $\cD_{\CC\times S/S}$-coherence of $(\cO_{\CC\times S}(*0),x\partial_x-\lambda(s))$ and to showing that the characteristic variety is contained in $(T^*_\CC\CC)\times S\cup (T^*_0\CC)\times S$.

Up to isomorphism, we can assume that, if $\lambda(0)\in\NN$, then $\lambda(0)=0$. We then denote by $\mu$ its order of vanishing, \ie $\lambda(s)=s^\mu u(s)$ with $u(0)\neq0$. Then one has the relation
\[
\partial_x^m\cdot\frac1x=\frac{\prod_{j=1}^m(\lambda(s)-j)}{x^{m+1}},
\]
in which the numerator is thus invertible in $\cO_{S,s_o}$. It follows that $1/x$ is a $\cD_{\CC\times S/S}$-generator of this module. We thus have a surjective morphism
\[
\cD_{\CC\times S/S}/\cD_{\CC\times S/S}(x\partial_x-(\lambda(s)-1))\to (\cO_{\CC\times S}(*0),x\partial_x-\lambda(s))
\]
sending the class of $1$ to $1/x$. It obviously becomes an isomorphism after tensoring with $\cO_{\CC\times S}(*0)$. Since the left-hand side has no $\cO_S$-torsion (\cf\cite[Ex.\,3.12]{MF-S12}), we conclude that the above morphism is an isomorphism. Now, the assertions of the corollary are clear for $\cD_{\CC\times S/S}/\cD_{\CC\times S/S}(x\partial_x-(\lambda(s)-1))$.

\subsubsection*{Step two}
We now relax the assumption of local freeness on $G$. If $\cO_{S,s_o}^N\to G_{s_o}$ is onto, then the kernel is torsion free, hence $\cO_{S,s_o}$-free, and we have an exact sequence
\[
0\to G''\to G'\to G\to0
\]
where $G',G''$ are $\cO_{S,s_o}$-free of finite rank. By flatness of $\cO_{\XS}(*D)$ over $\cO_S$, we have an exact sequence
\[
0\to\wh E''_1\to\wh E'_1\to\wh E_1\to0,
\]
from which we deduce, according to Step one, that $\wh E_1$ is holonomic with characteristic variety contained in $\Lambda\times S$. Moreover, the cohomology of $Li_{s_o}^*\wh E_1$ appears as the kernel and cokernel of the morphism $i_{s_o}^*\wh E''_1\to i_{s_o}^*\wh E'_1$, hence is also regular holonomic.
\end{proof}

\begin{proof}[Proof of the lemma]
Let us first work with the $\cO_{S,s_o}(*0)$-vector space $G_{s_o}(*0):=\cO_{S,s_o}(*0)\otimes_{\cO_{S,s_o}}G_{s_o}$. There exists a finite ramification $$\rho:(S',s'_o)\to (S,s_o)$$ such that each equation $\det(t\id-A_k(s'))=0$ has all its solutions in the field $\cO_{S',s'_o}(*0)$. These solutions, being algebraic over $\cO_{S',s'_o}$, belong to this ring. We can then assume from the beginning that all eigenvalues belong to $\cO_{S,s_o}$. Then $G_{s_o}(*0)$ decomposes as an $\cO_{S,s_o}(*0)[A_1,\dots,A_\ell]$-module with respect to the multi-eigenvalues $\lambda(s)=(\lambda_1(s),\dots,\lambda_\ell(s))$ as $G_{s_o}(*0)=\bigoplus_\lambda G_{s_o}(*0)_\lambda$, and $A_k(s)-\lambda_k(s)$ is nilpotent on $G_{s_o}(*0)_\lambda$. If we choose a total order on the set of~$\lambda$'s, we can define a filtration $$G_{s_o}(*0)_{\leq\lambda}:=\bigoplus_{\lambda'\leq\lambda} G_{s_o}(*0)_{\lambda'}.$$ It induces a filtration $G_{s_o,\leq\lambda}:=G_{s_o}(*0)_{\leq\lambda}\cap G_{s_o}$, where the intersection is taken in $G_{s_o}(*0)$, and every successive quotient $G_{s_o,\leq\lambda}/G_{s_o,<\lambda}$ is an $\cO_{S,s_o}[A_1,\dots,A_\ell]$-submodule of $G_{s_o}(*0)_\lambda$, hence is $\cO_{S,s_o}$-locally free and $A_k(s)-\lambda_k(s)$ is nilpotent on it for every $k=1,\dots,\ell$.

We can therefore assume from the beginning that every $A_k(s)$ is nilpotent on $G_{s_o}$. We now argue by induction on $\ell$. Consider the kernel filtration $G_{s_o}(*0)_j:=\ker A_1^j$. This is a filtration by $\cO_{S,s_o}(*0)[A_1,\dots,A_\ell]$-submodules and $A_1$ acts by zero on the quotient module $G_{s_o}(*0)_j/G_{s_o}(*0)_{j-1}$ for every $j$. As above, we can induce this filtration on $ G_{s_o}$ by setting $G_{s_o,j}:=G_{s_o}(*0)_j\cap G_{s_o}$ and $G_{s_o,j}/G_{s_o,j-1}$ is contained in $G_{s_o}(*0)_j/G_{s_o}(*0)_{j-1}$, hence has no $\cO_{S,s_o}$-torsion, \ie is $\cO_{S,s_o}$-free. By induction on $\ell$, we find a filtration whose successive quotients are free $\cO_{S,s_o}$-modules and on which every $A_k$ acts by zero, so that every rank-one $\cO_{S,s_o}$-submodule is also trivially an $\cO_{S,s_o}[A_1,\dots,A_\ell]$-submodule, and the lemma is proved.
\end{proof}

The following definition is a relative version of \cite[Def.\,2.3.1]{K-K81}.

\begin{definition}\label{D-t}
A coherent $\DXS$-module $\shl$ is said of D-type with singularities along $D$ if it satisfies the following conditions:
\begin{enumerate}
\item
$\Char\shl\subset (\pi^{-1}(D)\times S)\cup (T^*_{X} \XS)$,
\item
$\shl$ is regular holonomic and strict,
\item
$\shl\simeq\shl(*D)$.
\end{enumerate}
\end{definition}

\begin{proposition}\label{T:D-T}
The category of holonomic systems $\shl$ of D-type along $D$ is equivalent to the category of locally free $p_{X\setminus D}^{-1}\sho_S$-modules of finite type $F$ on $(X\setminus D)\times S$ under the correspondences $\shl\mto F=\shh^0 \DR\shl|_{(X\setminus D)\times S}$ and $F\mto \shl=\tilde{E}$.
\end{proposition}

Before entering the proof of Proposition \ref{T:D-T} we need the following description of relative moderate growth.

\begin{lemma}\label{fibers}
Assume that $F$ is a locally free $p_{X\setminus D}^{-1}\sho_S$-modules of finite type. Then a section $v$ of \hbox{$j_*(F\otimes_{\pOS}\sho_{(X\setminus D)\times S})$} is a section of $\tilde{E}$ if and only if, for each $s_o\in S$, $v(\cdot,s_o)$ has moderate growth as a section of the $\C$-local system $\tilde{E}|_{\{s=s_o\}}$ on $X\setminus D$. In particular, $i_{s_o}^*\tilde{E}=\tilde{i_{s_o}^*E}$.
\end{lemma}

\begin{proof}[\proofname\ (provided by Daniel Barlet)]
\item[]
\subsubsection*{Case $(1)$}
Let us assume that $F$ is $S$-constant. We may assume that the rank of $F$ is $1$. The statement being local, we may take local coordinates in a neighborhood of $(x_o,s_o)\in D\times S$ and assume that we are given a holomorphic function $v(x,s)$ in $(U\setminus D)\times V$, where $U$ is an open ball centered in $x_o$ and $V$ is an open ball centered in~$s_o$, such that, for any $s\in V$, $v(x,s)$ is meromorphic with poles along $D$. To~prove that~$v$ is meromorphic with poles along $D\times V$ it is sufficient, by Hartogs, to assume~$D$ non singular, hence defined by a coordinate $t=0$. Writing the Laurent's expansion of~$v$ $$v(x,s)=\sum_{i\in\Z} v_i(x,s)t^i,$$ where the $v_i$ are holomorphic in $U\times V$, we introduce, for $m\geq 0$, the increasing sequence of closed analytic sets
$$X_m=\{(x,s)\mid v_k(x,s)=0, \forall k\leq -m\}.$$
By the assumption, $U\times V=\bigcup_m X_m$ hence there must exist $m_0$ such that $X_{m_0}=U\times V$, which proves the claim.

\subsubsection*{Case $(2)$}
Let us assume that $D=\{x_1\cdots x_\ell=0\}$.
For a general $S$-local system $F$ locally free of rank $d$, let $\cG$ and $$(T_i(s), A_i(s))_{i=1,\dots,\ell},\quad s\in V$$ be given by Proposition \ref{prop:monodromyrepr} and Theorem \ref{th:Deligneext}, such that $T_i(s)=\exp(-2i\pi A_i(s))$, $i=1,\dots, \ell$. Let $(v_1,\dots, v_d)$ be a section of $F\otimes_{p_{X\setminus D}^{-1}\sho_S} \sho_{(X\setminus D)\times S}$ defined in $U'\times V$, for an open convex subset $U'$ of $U\setminus D$ (where we keep the notation of Case (1)). Then, according to Case (1),
\[
x_1^{A_1(s)}\cdots x_\ell^{A_\ell(s)}\begin{pmatrix}v_1(x,s)\\\vdots\\v_d(x,s)\end{pmatrix}
=\begin{pmatrix}u_1(x,s)\\\vdots\\ u_d(x,s)\end{pmatrix},
\]
where each $u_i$ is a meromorphic function with poles along $D$. Since the action of the matrix $x_1^{A_1(s)}\cdots x_\ell^{A_\ell(s)}$ does not affect the growth along $D$ the statement follows.
\end{proof}

\begin{proof}[\proofname\ of Proposition \ref{T:D-T}]
According to Theorem \ref{th:Deligneext}, if $F$ is a $\pOS$-local system locally free on \hbox{$(X\setminus D)\times S$}, $\tilde{E}$ is a coherent $\DXS$-module of D-type. Conversely, suppose that~$\shl$ is a coherent $\DXS$-module of D-type along $D$. The strictness assumption entails that $F:=\ho_{\DXS}(\cO_{\XS}, \shl)|_{(X\setminus D)\times S}$ is locally free of finite type.
Let $\psi:\shl\to j_*(F\otimes_{p|_{(X\setminus D)\times S}^{-1}\sho_S}\sho_{(X\setminus D)\times S})$ be the natural morphism of $\DXS$-modules. Since $\psi$ is an isomorphism on $(X\setminus D)\times S$ and $\shl=\shl(*D)$, it is injective. By definition, for each section $u$ of $\psi(\shl)$ and each fixed $s_o\in S$, $u(\cdot,s_o)$ is a section of a regular holonomic $\cD_X$-module, hence has moderate growth in the sense of \cite[p.\,862]{K-K81}. According to Lemma \ref{fibers}, $u$ is a local section of $\wt E$, hence $\psi(\shl)\subset\wt E$. The quotient sheaf $\wt E/\psi(\shl)$ also satisfies $\wt E/\psi(\shl)(*D)=\wt E/\psi(\shl)$, and is zero on $(X\setminus D)\times S$. Therefore $\psi(\shl)=\tilde{E}$.
\end{proof}

\section{Relative tempered cohomology functors}\label{sec:3}
We shall keep the notations of Section \ref{S:1} but for the main purpose of this section we may have to allow $X$ and $S$ to be real analytic manifolds since we often make use of subanalytic techniques which are naturally associated to real analytic structures. Indeed, when $X$ and $S$ are complex, it is often convenient to use the ``realification'' tool which enables one to go from the real to the complex analytic setting tensoring by the Dolbeault complex $\sho_{\overline{X}\times \overline{S}}$, where $\overline{X}$ and $\overline{S}$ denote the respective complex conjugate manifold.
We shall specify each case whenever there is a risk of ambiguity.

\subsection{Relative subanalytic site}\label{subsec:relsubanalytic}
We recall below the main constructions and results contained in \cite{MF-P14} and obtain complementary results to be used in the sequel. We refer to \cite{K-S01} as a foundational paper and to \cite{K-Sch06} for a detailed exposition on the general theory of sheaves on sites.

Let $X$ and $S$ be real analytic manifolds. On $\XS$ it is natural to consider the family $\sht$ consisting of finite unions of open relatively compact subsets and the family~$\sht'$ of finite unions of open relatively compact sets of the form $U\times V$ making $\XS$ both a $\sht$- and a $\sht'$-space in the sense of \cite{E-P16} and \cite{K-S96}. The associated sites $({\XS})_{\sht}$ and $({\XS})_{\sht'}$ are nothing more than, respectively, $(\XS)_\sa$ and the product of sites $X_\sa\times S_\sa$.

We shall denote by $\rho$, without reference to $\XS$ unless otherwise specified, the natural functor of sites $\rho:\XS \to (\XS)_\sa$ associated to the inclusion $\Op((X\times\nobreak S)_\sa) \subset \Op(\XS)$. Accordingly, we shall consider the associated functors $\rho_{*}, \rho^{-1}, \rho!$ introduced in \cite{K-Sch06} and studied in \cite{Prelli08}.

We shall also denote by $\rho':\XS \to (\XS)_{\sht'}$ the natural functor of sites. Following \cite{E-P16} we have functors $\rho'_*$ and $\rho'_!$ from $\Mod(\CC_{\XS})$ to $\Mod(\CC_{X_\sa\times S_\sa})$.

Note that $W$ is a $\sht'$-open subset or, equivalently, $W\in \Op(X_\sa\times S_\sa)$, if $W$ is a locally finite union of relatively compact subanalytic open subsets $W$ of the form $U \times V$, $U \in \Op(X_\sa)$, $V \in \Op(S_\sa)$. We denote by $\eta:(X \times S)_\sa \to X_\sa \times S_\sa$ the natural functor of sites associated to the inclusion $\Op(X_\sa \times S_\sa) \hto \Op((X \times S)_\sa)$.

\begin{remark}
\label{R}
As well-known consequences of the properties of $\sht$-spaces (\cf\cite[Ch.\,6,\,\S6.4,\,Prop.\,6.6.3]{K-S01}, see also \cite{Prelli08}) we recall:
\begin{itemize}
\item
$\rho'^{-1}$ and $\rho'_!$ are exact and commute with tensor products.
\item
If $f:X\to Y$ is a morphism, $\rho'^{-1}$ commutes with $f^{-1}$ and $\rho'_*$ commutes with $f_*$.
\item
$\rho'^{-1} \circ \rho'_* = \rho'^{-1}\circ \rho'_!=\Id$.
\item
Adjunctions:
\begin{align*}
\rho'_*\shhom(\rho^{\prime -1}(\cbbullet),\cbbullet) &\simeq \shhom(\cbbullet,\rho'_*(\cbbullet)),\\
\rho'^{-1}\shhom(\rho'_!(\cbbullet),\cbbullet)&\simeq \shhom(\cbbullet,\rho'^{-1}(\cbbullet)).
\end{align*}
\item
$\rho'_*$ commutes with $\shhom$ and $\Rhom$.
\item
Let $f$ be a real analytic map $X\to Y$. Still denoting by~$f$ the morphism $f\times \Id_S: \XS\to Y\times S$ or the associated morphism of sites, $X_{\sa}\times S_{\sa}\to Y_{\sa}\times S_{\sa}$,
according to \cite[17.5]{K-Sch06}, we have
\begin{itemize}
\item
a left exact functor of relative direct image
\[
f_*:\Mod(\CC_{X_{\sa}\times S_{\sa}})\to \Mod(\CC_{Y_{\sa}\times S_{\sa}}),
\]

\item
an exact functor of relative inverse image
\[
f^{-1}: \Mod(\CC_{Y_{\sa}\times S_{\sa}})\to \Mod(\CC_{X_{\sa}\times S_{\sa}}),
\]
and $(f^{-1},f_*)$ is a pair of adjoint functors.
\end{itemize}
\item
$\rho'^{-1}$ commutes with $f^{-1}$ and $\rho'_*$ commutes with $f_*$.
\end{itemize}

For example, the fourth item follows from adjunction and from the second item:
\[
\shhom(\rho'_*(\cbbullet),\rho'_*(\cbbullet)) \simeq \rho'_*\shhom(\rho'^{-1} \circ \rho'_*(\cbbullet),\cbbullet)\simeq
\rho'_*\shhom(\cbbullet,\cbbullet).
\]
For the commutation with $\Rhom$ one uses injective resolutions plus the property that~$\rho_*'$ transform injective objects into quasi-injective objects which are $\shhom(\rho'_*(F),\cbbullet)$-acyclic for any $F$.
\end{remark}

If $\shr$ is a sheaf of rings on $X_{\sa}\times S_{\sa}$, these properties remain true in $\Mod(\shr)$. According to \cite[Th.\,18.1.6 \& Prop.\,18.5.4]{K-Sch06}, $\Mod(\mathcal{R})$ is a Grothendieck category so it admits enough injectives and enough flat objects. Hence the derived functors appearing in the sequel are well-defined.

\subsection{Complements on S-constructible sheaves}\label{subsec:complements}
Along the proofs of the following results and the main constructions throughout this paper, as explained in the introduction, we need to consider families of open subanalytic sets generating the open coverings of $X_{\sa}$ formed by Stein, hence $\sho_S$-acyclic open subsets. This requires the assumption $d_S=1$.

\begin{assumption}
Throughout this section \ref{subsec:complements} we shall assume that $S$ is a complex manifold with complex dimension $d_S=1$ and we still denote by $S$ the underlying real analytic manifold.
\end{assumption}
The following result shows that $\Mod_\rc(\pOS)$ is a $\rho_*$ as well as a $\rho'_*$-acyclic category.

\begin{proposition}\label{P:236}
Let $F\in \Mod_\rc(\pOS)$. Then $\shh^k\rho_*(F)=\shh^j\rho'_*(F)=0$ for $k>0$. In particular $\rho'_*$ is exact on $\Mod_\rc(\pOS)$.
\end{proposition}

\begin{proof}
Let $U$ and $V$ be open subanalytic relatively compact sets respectively in $X$ and in $S$. Since dimension of $S$ is $1$, we may assume that $V$ is Stein. Similarly to the proof of \cite[Lem.\,2.1.1]{Prelli08}, it is sufficient to prove that, for each $k\neq 0$, there exists a finite covering $\{U_j\times V_j\}_{j\in I}, U_j\times V_j\in \sht'$, of
$U\times V$, such that $H^k(U_j\times V_j, F)=0$.

Let $X=\bigcup_{\alpha}X_{\alpha}$ be a Whitney stratification adapted to $F$. By \cite[Prop.\,8.2.5]{K-S90}, there exist a simplicial complex $\bK=(K, \Delta)$ and a homeomorphism $i: |\bK|\simeq X$ such that, for any simplex $\sigma\in\Delta$, there exists $\alpha$ such that $i(|\sigma|)\subset X_{\alpha}$ and $i(|\sigma|)$ is a subanalytic manifold of $X$. Moreover, we may assume that $U$ is a finite union of the images by
$i$ of open subsets $U(\sigma)$ of $|\bK|$, with $U(\sigma)=\bigcup_{\tau\in\Delta, \tau\supset \sigma}|\tau|$.
We shall see that we may take for $U_j\times V_j$ the open sets $i(U(\sigma))\times V$.
Therefore, still denoting by~$i$ the homeomorphism $|\bK|\times S\to \XS$, it~is enough to prove that for any $\sigma\in\Delta$ and any $x\in |\sigma|$, we have:
\begin{enumeratei}
\item
$H^0(U(\sigma)\times V; i^*F)\simeq H^0(V, F_{\{x\}\times S})$,
\item
$H^k(U(\sigma)\times V; i^*F)=H^k(V, F_{\{x\}\times S})=0$, for $j\neq 0$.
\end{enumeratei}

The proof of (i) and (ii) proceeds mimicking the proof of \cite[Prop.\,8.1.4]{K-S90}, using Proposition \ref{P:const}, the fact that the $F|_{\{x\}\times S}$ is $\sho_S$-coherent and that $V$ is Stein.
\end{proof}

\begin{remark}
By construction the isomorphisms (i) commute with the restrictions to open subsets $V'\subset V$ in $S$.
\end{remark}

Recall that, given an abelian category $\shc$, $\rK^{\rb}(\shc)$ denotes the category of complexes in $\shc$ having bounded cohomology, the morphisms being defined up to homotopy. For a locally closed set of $\XS$, $\CC_Z$ denotes both the constant sheaf on $Z$ and its extension by zero as a sheaf on $\XS$ (\cf\cite[Prop.\,2.5.4]{K-S90}).

\begin{proposition}\label{P:C1}
Let $F\in \Mod_\rc(\pOS)$. Then $F$ is quasi-isomorphic to a complex
$$0\to\oplus_{i_{\alpha}\in I_{\alpha}} \pOS \otimes \CC_{U_{\alpha, i_{\alpha}}\times V_{\alpha,i_{\alpha}}}\to\cdots\to\oplus_{i_{\beta}\in I_{\beta}} \pOS \otimes \CC_{U_{\beta, i_{\beta}}\times V_{\beta,i_{\beta}}} \to 0,$$
where $\{U_{j,i_j}\}_{j,i_j}$ are locally finite families of relatively compact open subanalytic subsets of $X$ and $\{V_{j,i_j}\}_{j,i_j}$ are locally finite families of relatively compact open subanalytic subsets of $S$.
\end{proposition}

\begin{proof}
We shall adapt the outline of the proof of \cite[Prop.\,A.2]{D-G-S11}. Let $X=\bigcup_{\alpha}X_{\alpha}$ be a Whitney stratification adapted to $F$. We keep the notation of Proposition \ref{P:236} and its proof when using \cite[Prop.\,8.2.5]{K-S90}.

For each integer $i$, let $\Delta_i\subset\Delta$ denote the subset of simplices of dimension $\leq i$ and set $\bK_i=(K, \Delta_i)$. We shall prove by induction on $i$ that there exists a morphism $\phi_i:G_i\to F$ in $\rK^{\rb}(\pOS)$ such that:
\begin{enumeratea}
\item\label{enum:Ka}
The $G^k_i$ are finite direct sums of $\pOS\otimes \CC_{h(U_{\sigma})\times V_{i,\sigma}}$ for some $\sigma\in\Delta_i$ and subanalytic open set $V_{i,\sigma}$ of $S$,
\item\label{enum:Kb}
The family $(h(U(|\sigma|))\times V_{i,\sigma})_{\sigma\in\Delta_i}$ is a locally finite covering of $h(|\bK_i|)\times S$,
\item\label{enum:Kc}
One has
$$\phi_i|_{|\bK_0|\times S}: G_i|_{|\bK_0|\times S}\to F|_{|\bK_0|\times S}$$ is a quasi-isomorphism.
\end{enumeratea}

\subsubsection*{Case $i=0$}

Let $x\in h(|\bK_0|)$, \ie $x=h(\sigma)$ for some $\sigma\in\bK_0$ and let $s_o\in S$. We have that $F|_{\{x\}\times S}\simeq G_{0,\sigma}$ for some $G_{0,\sigma}\in \rD^{\rb}_{\coh}(\sho_S)$; we then choose a subanalytic open set $V_{\sigma}\subset S$ such that $s_o\in V_{\sigma}$ and that $G_{0,\sigma}|_{V_{\sigma}}$ admits a bounded locally free $\sho_S|_{V_{\sigma}}$ resolution $R_{0, \sigma}\to G_{0,\sigma}|_{V_{\sigma}}$. Since $\dim S=1$, we may assume that $V_{\sigma}$ is Stein.

Clearly, the family $(h(U(\sigma))\times V_{\sigma})_{\sigma\in\Delta_0}$ is a locally finite covering of $h(|\bK_0|)\times S$. By (i) of Proposition \ref{P:236} we have isomorphisms of $\CC$-vector spaces
$$
\Gamma(h(U(\sigma))\times V'; F)\simeq \Gamma(V'; F|_{{\{x\}\times V_{\sigma}}}).
$$
In view of the freeness of $R_{\sigma,0}$, of the fact that $V'$ is Stein and of isomorphisms (i) and (ii) of Proposition \ref{P:236}, we conclude quasi-isomorphisms in $\rK^{\rb}(\CC)$
\begin{enumerate}
\item[(iii)]
$\Gamma (V'; R_{\sigma,0})\to \Gamma(h(U(\sigma))\times V'; F)$,
\end{enumerate}
which commute with restrictions. On the other hand we have isomorphisms
\begin{enumerate}
\item[(iv)]
$\phi_{V'}:\Gamma(h(U(\sigma))\times V'; F)\simeq \Gamma(V'; p_*\ho_{\rK^\rb(\CC_{\XS})}( \CC_{h(U(\sigma))\times V_{\sigma}}, F))$,
\end{enumerate}
which also commute with restrictions to open subsets $V''$ of $V'$.

Combining (iii) and (iv) we get a quasi-isomorphism in $\rK(\Mod(\sho_{V_{\sigma}}))$
$$
R_{\sigma,0}\to p_*\ho_{\rK^\rb(\CC_{\XS})|_{V_{\sigma}}}( \CC_{h(U(\sigma))\times V_{\sigma}}, F)|_{V_{\sigma}}.
$$
By adjunction, we get a morphism in $\rK^\rb(\Mod(p^{-1}\sho_{V_{\sigma}}))$:
$$p|_{V_{\sigma}}^{-1}R_{\sigma,0}\to \ho_{\rK^\rb(\CC_{\XS})}( \CC_{h(U(\sigma))\times V_{\sigma}}, F)|_{X\times V_{\sigma}}$$
which, by the functorial properties of $\ho$ and $\otimes$, induces a morphism
\begin{equation}
\phi_{\sigma,0}: p^{-1}R_{\sigma,0}\otimes \CC_{h(U(\sigma))\times V_{\sigma}}\to F.
\end{equation}
By construction $\phi_0:=\oplus_{\sigma\in\Lambda_0}\phi_{\sigma,0}$ gives the desired morphism.

\subsubsection*{General case}
Let us assume that $\phi_i$ is constructed and let us consider the distinguished triangle in $\rK^\rb(\Mod(\pOS))$:
$$
H_i\To{v_i} G_i\To{\phi_i} F\To{+1},
$$
where $H_i|_{h(|\bK_i|)\times S}$ is quasi-isomorphic to $0$. Therefore
$$
\bigoplus_{\sigma\in \Delta_{i+1}\setminus \Delta_i}
H_i|_{h(|\sigma|)\times S}\to H_i|_{h(|\bK_{i+1}|)\times S}$$
is a quasi-isomorphism.

Likewise the case $i=0$, let us choose, for each $\sigma\in\Delta_{i+1}\setminus\Delta_i$ and each $s_o\in S$, an open subanalytic relatively compact open subset $V_{\sigma,i+1}$ in $S$ containing $s$, a complex $R_{\sigma,i+1}$ of free $\sho_{V_{\sigma,i+1}}$-modules quasi-isomorphic to $F|_{\{x\}\times V_{\sigma,i+1}}$, for arbitrary $x\in |\sigma|$, and a morphism
$$
R_{\sigma,i+1}\to p_*\ho(\CC_{h(U(\sigma))\times V_{\sigma,i+1}}, H_i)|_{V_{\sigma,i+1}}.
$$
The family obtained as union of $(h(U_{\sigma})\times V_{\sigma,i})_{\sigma\in \Delta_i}$ and $(h(U_{\sigma})\times V_{\sigma,i+1})_{\sigma\in \Delta_{i+1}\setminus \Delta_i}$ clearly satisfies \eqref{enum:Kb} with respect to $h(|\bK_{i+1}|)\times S$.

As above we deduce a morphism
$$
\phi'_{i+1}:G'_{i+1}:=\bigoplus_{\sigma\in\Delta_{i+1}\setminus\Delta_i} p^{-1}R_{\sigma,i+1}\otimes \CC_{h(U(\sigma))\times V_{\sigma,i+1}}\to H_i
$$
such that, for $(x,s)\in h(|\bK_{i+1}|\setminus |\bK_{i}|)\times S$, the ${\phi'_{i+1}}_{(x,s)}$ are quasi-isomorphisms. For $(x,s)\in h(|\bK_i|)\times S$, the condition on $H_i$ entails that ${\phi'_{i+1}}_{(x,s)}$ is trivially a quasi-isomorphism. Therefore, $\phi'_{i+1}|_{h(|\bK_{i+1}|)\times S}$ is a quasi-isomorphism.

Let $G_{i+1}$ and $H_{i+1}$ be defined by the distinguished triangles
$$G'_{i+1}\To{\phi'_{i+1}}H_i\to H_{i+1}\To{+1}\quad\text{and}\quad
G'_{i+1}\To{v_i\circ\phi'_{i+1}} G_i\to G_{i+1}\To{+1}.$$
By construction and the induction hypothesis, $G_{i+1}$ satisfies \eqref{enum:Ka}. The octahedral axiom applied to the preceding triangles induces a morphism $\phi_{i+1}: G_{i+1}\to F$ and hence a distinguished triangle
$$H_{i+1}\to G_{i+1}\To{\phi_{i+1}}F\To{+1}.$$
Since by its construction ${H_{i+1}}|_{h(|\bK_{i+1}|)\times S}$ is quasi-isomorphic to zero, $\phi_{i+1}$ satisfies~\eqref{enum:Kc} as desired.
\end{proof}
\begin{remark}
\label{lct}
Recall (see \cite[Prop.\,3.9]{Sch-Sch94}) that each $U_{j,i_j}$ can be chosen so that $\bD'\CC_{U_{j,i_j}}\simeq\CC_{\ov U_{j,i_j}}$.
\end{remark}

Let $q:\XS\to X$ denote the projection on the first factor.

\begin{corollary}\label{L:210}
Let $F \in D^b_\rc(\pOS)$, $F'\in D^b_\rc(\C_X)$ and let $s_o\in S$. Then
\begin{enumeratea}
\item\label{L:210a}
$\rho'_*F \otimes \rho'_*q^{-1}F' \simeq \rho'_*(F \otimes q^{-1}F')$.

\item\label{L:210b}
The natural morphism
\[
\rho'_*p^{-1}(\sho_S/\mathfrak{m}_{s_o})\otimes^L_{\rho'_*p^{-1}\sho_S}\rho'_*F\to \rho'_*(p^{-1}(\sho_S/\mathfrak{m}_{s_o})\otimes^L_{p^{-1}\sho_S}F)
\]
is an isomorphism.
\end{enumeratea}
\end{corollary}

\begin{proof}\mbox{}
\begin{enumeratea}
\item
According to Proposition \ref{P:C1}, which provides a $p^{-1}(\sho_S)$-flat resolution of $F$, we may assume that $F=\C_{U\times V}\otimes p^{-1}\sho_S$. Similarly, as proved in \cite{Kashiwara84}, we may assume $F'=\C_{U'}$ for some open relatively compact subset $U'\subset X$. Therefore $F\otimes q^{-1}F'=\C_{(U\cap U')\times S}\otimes p^{-1}(\sho_S\otimes \C_V)$. So, on one hand, according to \cite[Lem.\,3.6(2)]{MF-P14} we have $$\rho'_*(F\otimes q^{-1}F')=\rho'_*\C_{(U\cap U')\times S}\otimes \rho'_*p^{-1}(\sho_S\otimes \C_V)$$
On the other hand we have, for the same reason,
\begin{align*}
\rho'_*(F)\otimes \rho'_*(q^{-1}F')&=\rho'_*(\C_{U\times S}\otimes p^{-1}(\sho_S\otimes \C_V))\otimes \rho'_*(\C_{U'\times S})\\
&=\rho'_*(\C_{U\times S})\otimes \rho'_*p^{-1}(\sho_S\otimes \C_V))\otimes \rho'_*(\C_{U'\times S})
\end{align*}
and the result follows from the equality $\rho'_*(\C_{(U\cap U')\times S}=\rho'_*(\C_{U\times S})\otimes \rho'_*(\C_{U'\times S})$
(\cf \hbox{\cite[Th.\,2.2.6(2)]{E-P16}}).

\item
According to Proposition \ref{P:236}, $\rho'_*$ is exact on $D^b_\rc(\pOS)$ and, as above, we may assume that $F=\C_{U\times V}\otimes p^{-1}\sho_S$. Up to shrinking $V$ (possible by the construction of the family $\{V_{j,i_j}\}_{j,i_j}$ mentioned in Proposition \ref{P:C1}), we can also assume that there is a holomorphic coordinate $s$ vanishing at $s_o$ defined on $V$. It remains to observe that the left term in \eqref{L:210b} is realized by the complex
$\rho'_*F\To{s-s_o}\rho'_*F$ and the right term by $\rho'_*(F\To{s-s_o}F)$. They are thus isomorphic by the exactness of~$\rho'_*$.\qedhere
\end{enumeratea}
\end{proof}

\subsection{Relative subanalytic sheaves}
Subanalytic sheaves are defined on the subanalytic site of a real analytic manifold, so we start by assuming that $X$ and $S$ are real analytic manifolds, and we denote by $\shd_{\XS}$ the sheaf of linear differential operators with real analytic coefficients.

In the absolute case ($S=\mathrm{pt}$), the functors of tempered cohomology from $\rD^\rb_{\rc}(\C_X)$ to $\rD^\rb(\shd_X)$, respectively denoted by $\tho(\cdot,\Db_X)$ ($\Db_X$ is the sheaf of distribution on the underlying $C^\infty$ manifold, with its $\cD_X$-module structure) and $\tho(\cdot,\sho_X)$ (in the complex case), were introduced by M. Kashiwara in \cite{Kashiwara84} and later, in \cite{K-S01}, the authors showed that they can be recovered using the language of subanalytic sheaves as $\rho^{-1}\rh(\cdot, \Db_X^t)$ \resp $\rho^{-1}\rh(\cdot,\sho_X^t)$, where $\Db_X^t$ is the subanalytic sheaf of tempered distributions, \resp $\sho_X^t$ is the subanalytic complex of tempered holomorphic functions on~$X_{\sa}$.
Recall that, if $U$ is an open relatively compact subanalytic subset of $X$, for any open subset $\Omega\subset X$, $\Gamma(\Omega, \tho(\C_U,\Db_X))$ is the space of distributions on $\Omega\cap U$ which extend to $\Omega$ and, if moreover $X$ is complex and $U$ is Stein, \hbox{$\Gamma(\Omega, \tho(\C_U, \sho_X))$} is the space of holomorphic functions on $\Omega\cap U$ which have a moderate growth with respect to the distance to $\partial U$ and so they extend as distributions to $\Omega$.
Here we will adapt these notions to the relative case.

Let $G$ be a sheaf on $(\XS)_{\sht'}$. One defines the (separated) presheaf $\eta^\gets G$ on $(\XS)_\sa$ by setting, for $W \in \Op((\XS)_\sa)$,
$$
\eta^\gets G(W)= \lind {W \subset W'}G(W')
$$
with $W' \in \Op((\XS)_{\sht'})$ (\cf Section \ref{subsec:relsubanalytic} for $\sht'$). Let $\imin\eta G$ be the associated sheaf.

Let $F$ be a subanalytic sheaf on $(\XS)_\sa$. We shall denote by $F^{S,\sharp}$ the sheaf on $X_\sa \times S_\sa$ associated to the presheaf
\[
\begin{array}{rll}
\Op(X_\sa \times S_\sa) & \displaystyle\to \Mod(\C)& \\[2pt]
U \times V & \displaystyle\mto \Gamma(X \times V;\imin\rho\Gamma_{U \times S}F) &\simeq \Hom(\CC_U \boxtimes \rho_!\CC_V,F) \\[2pt]
&& \simeq \lpro {\substack{W \Subset V\\ W\in\Op^c(S_\sa)}}\Gamma(U \times W;F).
\end{array}
\]
With the notations above, for a morphism $f:X\to Y$ of analytic manifolds, we have
\begin{align*}
f^{-1}(F^{S.\sharp})&\simeq ((f\times \id_S)^{-1}F)^{S,\sharp},\\
f_*(F^{S.\sharp})&\simeq ((f\times \id_S)_*F)^{S,\sharp},
\end{align*}
for any $F\in\Mod(\mathbb{K}_{(\XS)_\sa})$. We set
\begin{equation}\label{E:100}
F^{S}:=\imin\eta F^{S,\sharp}
\end{equation}
and call it the \emph{relative sheaf associated to $F$}. It is a sheaf on $(X \times S)_\sa$ and $(\cbbullet)^S$ defines a left exact functor on $\Mod(\CC_{(X \times S)_\sa})$. We will denote by $(\cbbullet)^{RS,\sharp}$ and $(\cbbullet)^{RS} \simeq \imin\eta \circ (\cbbullet)^{RS,\sharp}$ the associated right derived functors.

Recall that by \cite[Lem.\,3.4]{MF-P14}, we have an isomorphism $\id \simeq R\eta_*\eta^{-1}$.

\begin{definition}\label{Dbt}
We define $\Db_{\XS}^{t,S}$ as the relative sheaf associated to $\Db^{t}_{\XS}$. It is naturally endowed with a structure of $\rho_!\shd_{\XS}$-module as well as a structure of $\rho'_*p^{-1}\sho_S$-module which commutes with the structure of $\rho_!\DXS$-module.
\end{definition}

By \cite[ Props.\,5.1(2), 5.2(i) and 5.3(2)]{MF-P14}, $\Db^{t,S}_{\XS}$ has the following properties:
\begin{align*}
\tag{i}
&\left\{\begin{aligned}
\Gamma(U \times V;\Db^{t,S}_{\XS})&=\Gamma(X \times V;\imin\rho\Gamma_{U \times S}\Db_{\XS}^t)\\
&\simeq \Gamma(X \times V;\tho(\CC_{U \times S},\Db_{\XS})),
\end{aligned}\right.\\
\tag{ii}
&\left\{\begin{aligned}
\!\imin\rho\rh(G \boxtimes H,\Db^{t,S}_{\XS}) &\simeq \imin\rho\rh(G \boxtimes \rho_!H,\Db_{\XS}^{t,S})\\
&\simeq \rh(\CC_X \boxtimes H,\tho(G \boxtimes \CC_S,\Db_{\XS})),\\
&\hspace*{2cm}\text{for any $G\in \rD^\rb_\rc(X), \, H\in \rD^\rb_\rc(S)$.}
\end{aligned}\right.
\end{align*}

\noindent(iii)
$\Db^{t,S}_{\XS}$ is $\Gamma(U \times V;\cdot)$-acyclic for each $U \in \Op(X_\sa)$, $V \in \Op(S_\sa)$. In particular, $\Db^{t}_{\XS}$ is $(\cbbullet)^{S,\sharp}$-acyclic and hence $(\cbbullet)^S$-acyclic.

As a consequence, we have an isomorphism in $\rD^\rb(\shd_{\XS})$
\begin{equation}\label{E:23}
\rho^{-1}\Db^{t,S}_{\XS}\simeq {\rho'}^ {-1}\Db^{t, S,\sharp}_{\XS}\simeq \Db_{\XS}.
\end{equation}
Having in mind the underlying real analytic structures, the preceding statements make sense when either $X$ or $S$ or both are complex.
\begin{assumption} {Throughout the rest of this section we shall assume that $X$ is a complex manifold of dimension $d_X$. Throughout the rest of this work we assume that~$S$ is complex and $d_S=1$.}
\end{assumption}

We denote as usual by $\overline{X}\times \overline{S}$ the complex conjugate manifold and regard \hbox{$(\XS)\times(\overline{X}\times\overline{S})$} as a complexification of the real analytic manifold $(\XS)_{\R}$ underlying $\XS$. We shall write for short $\Db^{t,S}_{\XS}$ instead of $\Db^{t,S}_{(\XS)_{\R}}$. By \hbox{\cite[Lem.\,5.4 and Lem.\,5.5]{MF-P14}}, there is a natural action of $\eta^{-1}\rho'_*\pOS$, $\rho_!\shd_{\XS}$ and $\rho_!\shd_{\overline{X}\times \overline{S}}$ on $\Db^{t,S}_{\XS}$ and the same argument holds for $\rho'$ instead of $\rho$ in the two last actions with $\Db^{t,S,\sharp}_{\XS}$ instead of $\Db^{t,S}_{\XS}$. We then define
$\cO_{\XS}^{t,S}$ as the derived relative complex associated to ${\sho}^t_{\XS}$, that~is, it is defined via the isomorphism in $\rD^\rb(\rho_!\shd_{\XS})$
$$\cO_{\XS}^{t,S} \simeq (\rh_{\rho_!\shd_{\overline{X}\times \overline{S}}}(\rho_!\sho_{\overline{X}\times \overline{S}},{\Db}^t_{\XS}))^{RS}.$$
${\sho}^t_{\XS}$ is naturally an object of $\rD^\rb(\eta^{-1}\rho'_*\pOS$). More precisely, setting
$\cO_{\XS}^{t,S,\sharp}:=(\cO_{\XS}^{t})^{RS,\sharp}$, then $\cO_{\XS}^{t,S}\simeq \eta^{-1} \cO_{\XS}^{t,S,\sharp}$.

According to the $(\cbbullet)^{S.\sharp}$ and the $(\cbbullet)^{S}$-acyclicity of ${\Db}^t_{\XS}$ (\cf \cite[Prop.\,5.2(i)] {MF-P14}) we get isomorphisms in $\rD^\rb(\rho_!\shd_{\XS})$
$$\cO_{\XS}^{t,S} \simeq \rh_{\rho_!\shd_{\overline{X}\times \overline{S}}}(\rho_!\sho_{\overline{X}\times \overline{S}},\Db^{t,S}_{\XS})$$
and
$$\cO_{\XS}^{t,S,\sharp}\simeq \rh_{\rho'_!\shd_{\overline{X}\times \overline{S}}}(\rho'_!\sho_{\overline{X}\times \overline{S}},\Db^{t,S,\sharp}_{\XS}).$$
Moreover, by \cite[Prop.\,4.1\,\&\,5.7] {MF-P14}, for $G=\CC_U$ and $H =\CC_V$ we have isomorphisms in $\rD^\rb(\rho_!\shd_{\XS})$
\begin{equation}\label{E:ot}
\begin{split}
\rho'^{-1}\rh(\CC_{U\times V},\sho^{t,S,\sharp}_{\XS})
&\simeq \rho^{-1}\rh(\CC_{U\times V}, \cO_{\XS}^{t,S})\\
&\simeq \rho^{-1}\rh(\CC_U \boxtimes \rho_!\CC_V,\sho^{t}_{\XS})\\
&\simeq \rh(\CC_{X\times V}, \tho(\CC_{U \times S},\sho_{\XS})),\\&\hspace*{1cm}\text{for any $U\in\Op(X_\sa)$ and $V\in\Op(S_\sa)$.}
\end{split}
\end{equation}
Since $\sho_{\XS}\simeq \rho'^{-1}(\sho^{t,S,\sharp}_{\XS})$, by adjunction we get a morphism in $\rD^\rb( \rho'_!\shd_{\XS})$
\begin{equation}\label{E:t}
\sho^{t,S,\sharp}_{\XS}\to R\rho'_*\sho_{\XS}
\end{equation}

\begin{lemma}\label{L.201}
The morphism \eqref{E:t} induces an isomorphism in $\rD^\rb(\rho'_*\pOS):$
\begin{multline*}
\rh_{\rho'_!\DXS}(\rho'_!\cO_{\XS}, \cO_{\XS}^{t.S,\sharp})\\
\isom\rh_{\rho'_!\DXS}(\rho'_!\cO_{\XS}, R\rho'_*\sho_{\XS})\simeq \rho'_*\pOS.
\end{multline*}
\end{lemma}

\begin{proof}
We start by proving the first isomorphism. Since the family of open subanalytic sets of the form $U\times V$ generate the open coverings of $X_{\sa}\times S_{\sa}$, it is sufficient to prove that, for any open subanalytic relatively compact sets $U$ in X and~$V$ in $S$, the morphism \eqref{E:t} induces an isomorphism, functorial in $U\times V$
\begin{multline*}
R\Gamma(U\times V,\rh_{\rho'_!\DXS}(\rho'_!\cO_{\XS}, \cO_{\XS}^{t,S,\sharp}))\\ \to R\Gamma(U\times V;\rh_{\rho'_!\DXS}(\rho'_!\sho_{\XS},R\rho'_* \sho_{\XS}))
\end{multline*}
We have a chain of isomorphisms
\begin{align*}
R\Gamma(&U\times V,\rh_{\rho'_!\DXS}(\rho'_!\cO_{\XS}, \cO_{\XS}^{t,S,\sharp}))\\
&\simeq \Rh(\rho'_*\CC_{U\times V},\rh_{\rho'_!\DXS}(\rho'_!\cO_{\XS}, \cO_{\XS}^{t,S,\sharp}))\\
&\simeq \Rh_{\rho'_!\DXS}(\rho'_!\cO_{\XS}, \rh(\CC_{U\times V}, \cO_{\XS}^{t,S,\sharp}))\\
&\simeq \Rh_{\DXS}(\cO_{\XS}, \rho'^{-1}\rh(\rho'_*\CC_{U\times V},\cO_{\XS}^{t,S,\sharp}))\quad\text{(by adjunction)}\\
&\simeq \Rh_{\DXS}(\cO_{\XS}, \rh(\CC_{X\times V}, \tho(\CC_{U \times S},\sho_{\XS})))\\
&\simeq \Rh(\CC_{X\times V}, \rh_{\DXS}(\cO_{\XS}, \tho(\CC_{U \times S},\sho_{\XS}))),\quad\text{(by \eqref{E:ot})}.
\end{align*}
Similarly we have the chain of isomorphisms:
\begin{align*}
R\Gamma(U\times V&;\rh_{\rho'_!\DXS}(\rho'_!\sho_{\XS},R\rho'_* \sho_{\XS}))\\
&\simeq \Rh(\rho'_*\C_{U\times V},\rh_{\rho'_!\DXS}(\rho'_!\sho_{\XS},R\rho'_*\sho_{\XS}))\\
&\simeq \Rh_{\rho'_!\DXS}(\rho'_!\sho_{\XS},\rh(\rho'_*\C_{U\times V},R\rho'_*\sho_{\XS}))\\
&\simeq \Rh_{\DXS}(\sho_{\XS},\rho'^{-1}\rh(\rho'_*\C_{U\times V},R\rho'_*\sho_{\XS}))\\
&\simeq \Rh_{\DXS}(\sho_{\XS},\rh(\C_{X\times V},\Rhom(\C_{U \times S},\sho_{\XS})))\\
&\simeq\Rh(\C_{X\times V},\rh_{\DXS}(\sho_{\XS},\rh(\C_{U \times S},\sho_{\XS}))).
\end{align*}

The isomorphisms of each chain are compatible with \eqref{E:t} because they come from natural equivalences of functors. We have thus reduced the proof to showing that the morphism
\begin{multline*}
\tag{i}
\shf_1:=\rh_{\DXS}(\cO_{\XS}, \tho(\CC_{U\times S},\cO_{\XS}))\\
\to \shf_2:=\rh_{\DXS}(\cO_{\XS}, \rh(\CC_{U\times S},\cO_{\XS}))
\end{multline*}
is an isomorphism in $\rD^ \rb(\pOS)$, functorial in $U$. The functoriality on $U$ is obvious. According to Proposition \ref{P:C1} and Remark \ref{lct}, we may assume that $U$ is relatively compact contractible and $\bD'\CC_U\simeq\nobreak\CC_{\ov U}$. Moreover, since the statement is a local question, we may also consider that $X=\C^n$ with the coordinates $x=(x_1,\dots, x_n)$ and then $\sho_{\XS}$ is realized by
\[
{\DXS}/\DXS\partial_{x_1}+\DXS\partial_{x_{2}}+\cdots+\DXS\partial_{x_n}.
\]

On the one hand, remarking that $\shd_{\XS}\otimes _{\DXS}\sho_{\XS}$ is nothing but the transfer module $\shd_{\XS\to S}$ associated to $p:\XS\to S$, we deduce a functorial chain of isomorphisms
\begin{align*}
\shf_1
&\simeq\rh_{\shd_{\XS}}(\shd_{\XS}\otimes _{\DXS}\cO_{\XS}, \tho(\CC_{U\times S},\cO_{\XS}))\\
&\simeq\rho^{-1}\rh_{\rho_!\shd_{\XS}}(\rho_!\shd_{\XS\to S}, \rh(\CC_{U\times S},\cO^t_{\XS}))\\
&\overset{(3)}\simeq \rho^{-1}\rh(\CC_{U\times S}, \rh_{\rho_!\shd_{\XS}}(\rho_!\shd_{\XS\to S}, \cO^t_{\XS}))\\
&\overset{(4)}\simeq \rho^{-1}\rh(\CC_{U\times S}, \pOS^t)
\end{align*}
where $\rho_S$ denotes the morphism of sites $S\to S_\sa$ thus satisfying $\rho_S p=p\rho$, $(3)$~follows by the associative property relating the derived functors of $\otimes$ and $\ho$ and~$(4)$~follows from \cite[Cor.\,A.3.7]{Prelli13}.

Similarly, in the non-tempered case we have a chain of isomorphisms :
\begin{align*}
\shf_2
&\simeq\rh_{\shd_{\XS}}(\shd_{\XS}\otimes _{\DXS}\cO_{\XS}, \rh(\CC_{U\times S},\cO_{\XS}))\\
&\simeq\rh_{\shd_{\XS}}(\shd_{\XS\to S}, \rh(\CC_{U\times S},\cO_{\XS}))\\
&\simeq \rh(\CC_{U\times S}, \rh_{\shd_{\XS}}(\shd_{\XS\to S}, \cO_{\XS}))\\
&\simeq \rh(\CC_{U\times S}, \pOS)
\end{align*}
The isomorphisms of each chain are compatible with (i) because they come from natural equivalences of functors. Hence it remains to prove that
the natural morphism
\[
\tag{ii}\rho^{-1}\rh(\CC_{U\times S}, \pOS^t)\to\rh(\CC_{U\times S}, \pOS)
\]
is an isomorphism.
We have a commutative diagram of natural morphisms
\[
\xymatrix{
\rho^{-1}\rh(\CC_{U\times S}, \pOS^t)\ar[r]^-{(\text{ii})}\ar[d]_\wr&\rh(\CC_{U\times S}, \pOS)\ar[d]^\wr\\
\rho^{-1}(\C_{\overline{U}\times S}\otimes \pOS^t)\ar[r]\ar[d]_\wr&\C_{\overline{U}\times S}\otimes \pOS\\
\C_{\overline{U}\times S}\otimes \rho^{-1}\pOS^t \ar[ur]^-{(\text{iii})}&
}
\]
where the vertical morphisms are obtained from \hbox{\cite[Th.\,3.4.4]{K-S90}} (in the framework of sheaves on $(\XS)_\sa$ in \cite[Prop.\,5.3.9]{Prelli13} for the left arrow) together with the assumption on~$U$. Then (iii) is an isomorphism because $\rho^{-1}\pOS^t\simeq p^{-1}\rho_S^{-1}\sho^t_S\to \pOS$ is equal to $p^{-1}$ of the natural isomorphism $\rho_S^{-1}\sho^t_S\simeq\sho_S$, concluding therefore the proof of (i).

For the second isomorphism in Lemma \ref{L.201}, we note that, for any $U, V$ as above,
\begin{align*}
\Rh(\C_{X\times V}, \rh_{\DXS}(&\sho_{\XS},\rh(\C_{U \times S},\sho_{\XS})))\\
&\simeq \Rh(\C_{U\times V}, \rh_{\DXS}(\sho_{\XS},\sho_{\XS}))\\
&\simeq R\Gamma(U\times V; \rh_{\DXS}(\sho_{\XS},\sho_{\XS}))).
\end{align*}
The last expression is isomorphic to
$$
R\Gamma(U\times V; \pOS)=R\Gamma(U\times V; \rho'_*\pOS),
$$
as desired.
\end{proof}

\subsection{The functors $\TH^S$ and $\RH^S$}\label{subsec:RHTH}

Recall that $\rho'^{-1}\rho'_!=\Id$, so if $X$ is either a real or a complex analytic manifold, we have $\rho'^{-1}\rho'_!\DXS=\DXS$. We define the following functors.
\begin{itemize}
\item
If $X$ is a real analytic manifold, let us consider the subsheaf $\DXSR$ of the sheaf $\shd_{\XS_{\R}}$ of linear differential operators with real analytic coefficients on \hbox{$\XS_\RR$} which commute with $\pOS$. Hence the operators in $\DXSR$ are those not containing holomorphic derivations with respect to $S$. Then $\TH^S:\rD^\rb_\rc(\pOS)\mto\rD^\rb(\DXSR)$ is given by the assignment
$$F\mto \TH^S(F):=
\rho'^{-1}\rh_{\rho'_*\pOS}(\rho'_*F, \Db^{t,S,\sharp}_{\XS}),$$
\item
If $X$ is a complex manifold of complex dimension $d_X$, $\RH^S:\rD^\rb_\rc(\pOS)\mto\rD^\rb(\DXS)$ is given by the assignment
$$F\mto \RH^S(F):=
\rho'^{-1}\rh_{\rho'_*\pOS}(\rho'_*F, \sho^{t,S,\sharp}_{\XS})[d_X].$$
\end{itemize}

Clearly $\rho^{-1}=\rho'^{-1}\eta_*$ but $\rho_*\neq \eta^{-1}\rho'_*$ hence, if in the preceding definitions we replace $\rho'$ by $\rho$, we obtain a different notion.

When $X$ is complex, considering the underlying real analytic structure on $X$, the functor $\TH^S(\cbbullet)$ is also defined on $\rD^\rb_{\rc}(\pOS)$ with values in $\rD^\rb(\shd_{X_\RR\times S_\RR/S})$. Note that $\shd_{X_\RR\times S_\RR/S}$ contains $\shd_{\overline{X}\times \overline{S}}$ as a subsheaf. Note also that, by the adjunction formula for $\rho'$, we have an isomorphism in $\rD^\rb(\DXS)$\
\begin{equation}\label{E:3n}
\RH^S(F)\simeq \rh_{\shd_{\overline{X}\times \overline{S}}}(\sho_{\overline{X}\times \overline{S}}, \TH^S(F))[d_X].
\end{equation}

We have a functorial isomorphism in $\rD^\rb(\DXS)$:
$$\Rhom_{\pOS}(F,\sho_{\XS})\simeq\rho'^{-1}\Rhom_{\rho'_*{\pOS}}(\rho'_*F, R\rho'_*\sho_{\XS})$$
Combining this isomorphism with \eqref{E:t} we obtain a functorial morphism
\begin{equation}
\RH^S(F)[-d_X]\to\Rhom_{\pOS}(F,\sho_{\XS}),
\end{equation}
and therefore, for any object $\shm$ of $\rD^\rb(\DXS)$, a bi-functorial morphism
\begin{multline}\label{E:Theta2}
\Rhom_{\DXS}(\shm, \RH^S(F)[-d_X])\\
\to \Rhom_{\DXS}(\shm, \Rhom_{\pOS}(F,\sho_{\XS})).
\end{multline}

\begin{lemma}\label{L:comp}
Let $F\in \rD^\rb_\rc(\pOS)$. Then the natural morphism
\begin{multline*}
\rh_{\DXS}(\sho_{\XS}, \RH^S(F)[-d_X])\\
\to \rh_{\DXS}(\sho_{\XS}, \rh_{\pOS}(F, \sho_{\XS}))
\end{multline*}
is an isomorphism. In particular $\pDR(\RH^S(F))\simeq \bD F$.
\end{lemma}

\begin{proof}
We have a chain of functorial isomorphisms in $\rD^\rb(\pOS)$\begin{align*}
\rh_{\DXS}(&\sho_{\XS}, \RH^S(F)[-d_X])\\
&\simeq \rho'^{-1}\rh_{\rho'_!\DXS}(\rho'_!\sho_{\XS}, \rh_{\rho'_*\pOS}(R\rho'_*F, \sho^{t,S,\sharp}_{\XS}))\\
&\simeq \rho'^{-1}\rh_{\rho'_*\pOS}(R\rho'_*F, \rh_{\rho'_!\DXS}(\rho'_!\sho_{\XS}, \sho^{t,S,\sharp}_{\XS})).
\end{align*}
Similarly, the other side, we have a chain of functorial isomorphisms
\begin{align*}
\rh_{\DXS}&(\sho_{\XS}, \rh_{\pOS}(F, \sho_{\XS}))\\
&\simeq \rho'^{-1}\rh_{\rho'_!\DXS}(\rho'_!\sho_{\XS}, \rh_{\rho'_*\pOS}(R\rho'_*F, R\rho'_*\sho_{\XS}))\\
&\simeq \rho'^{-1}\rh_{\rho'_*\pOS}(R\rho'_*F, \rh_{\rho'_!\DXS}(\rho'_!\sho_{\XS}, R\rho'_*\sho_{\XS})).
\end{align*}
The first part of the statement then follows by Lemma \ref{L.201}. Let us prove the last assertion. We have, functorially in $F$, a chain of isomorphisms
\begin{multline*}
\rh_{\DXS}(\sho_{\XS}, \rh_{\pOS}(F, \sho_{\XS}))\\
\simeq \rh_{\pOS}(F, \rh_{\DXS}(\sho_{\XS},\sho_{\XS}))\\
\simeq \rh_{\pOS}(F, p^{-1}\sho_{S}))\simeq \bD'(F).\qedhere
\end{multline*}
\end{proof}

\begin{example}\label{Ex:RH}
Let $X$ be a complex manifold. For a given hypersurface $Y$ of $X$ (possibly with singularities) let us denote by $j$ the open inclusion $X\setminus Y\hto X$ as well as the associated map $(X\setminus Y)\times S\hto \XS$.\begin{enumerate}
\item\label{Ex:RH1}
Assume $F\simeq \pOS^\ell|_{(X\setminus Y)\times S}$ for some $\ell \in\mathbb{N}$. Then $\tho(\CC_{(X\setminus Y)\times S}, \sho_{\XS})$ is a regular holonomic $\shd_{\XS}$-module endowed with a natural structure of $\DXS$-module, as proved in \cite{Kashiwara84}, hence it is regular holonomic as a $\DXS$-module. More precisely, given locally an equation $f=0$ defining $Y$, $\tho(\CC_{(X\setminus Y)\times S}, \sho_{\XS})$ is the localized of $\sho_{\XS}$ with respect to $f$. By Corollary~\ref{L:210}\eqref{L:210a} we have in $\rD^\rb(\DXS)$
\begin{align*}
\RH^S(j!F)[-d_X]&\simeq\rho'^{-1}\rh_{\rho'_*\pOS}(\rho'_*\pOS^\ell\otimes \rho'_*\CC_{(X\setminus Y)\times S}, \sho^{t,S,\sharp}_{\XS})\\
&\simeq \rho'^{-1}\rh_{\rho'_*\pOS}(\rho'_*\pOS^\ell,\rh(\rho'_*\CC_{(X\setminus Y)\times S}, \sho^{t,S,\sharp}_{\XS}))\\
&\simeq \rho^{-1}\rh(\rho_*\CC_{(X\setminus Y)\times S}, \sho^{t,S}_{\XS}))^\ell,
\end{align*}
and by \eqref{E:ot} we get $$ \RH^S(j!F)[-d_X] \simeq\tho(\CC_{(X\setminus Y)\times S}, \sho_{\XS})^\ell.$$

\item\label{Ex:RH2}
Let us assume that $F\simeq p^{-1} \cG$ with $\cG$ coherent over $\sho_S$, that is, $F$ is an $S$-constant local system on $(X\setminus Y)\times S$.

According to Case (1), by considering a local free resolution $\sho_S^{\cbbullet}$ of $\cG$ on a sufficiently small open subset $V$ of $S$, we obtain that $\RH^S(j_!F)[-d_X]|_{p^{-1}V}$ is quasi-isomorphic to a complex for which the terms are finite direct sums of $\tho(\CC_{(X\setminus Y)\times S}, \sho_{\XS})$ and the differentials are given by the right multiplication by matrices with entries in $\pOS$, hence $\DXS$-linear morphisms. Therefore the cohomology groups are regular holonomic $\DXS$-modules.

\item\label{Ex:RH3}
We assume $X=\C$.
For any $F\in \Mod_\cc(\pOS)$ such that $(\C^*, \{0\})$ is an adapted stratification we have $F\otimes \CC_{\{0\}\times S}\simeq p^{-1}\cG\otimes \CC_{\{0\}\times S}$ for some coherent $\sho_S$\nobreakdash-module $\cG$. Then, by Corollary \ref{L:210}\eqref{L:210a},
\begin{align*}
\RH^S(F)&\simeq \rh_{\pOS}(p^{-1}\cG, \tho(\CC_{\{0\}\times S}, \sho_{\XS}))[1]\\
&\simeq \rh_{\pOS}(p^{-1}\cG, B_{\{0\}\times S|\XS}).
\end{align*}
\end{enumerate}
\end{example}

\subsection{Extension in the case of an open subanalytic set}\label{subsec:Ext}
\begin{assumption} In this section $X$ is a complex manifold and $U$ denotes an open subanalytic subset of $X$.
\end{assumption}
Let $U_{X_\sa}$ be the subanalytic site induced by $X_\sa$ on $U$ (\cf \cite[Rem.\,1.1.1]{Prelli08}), and let $j:U_{X_\sa}\times S_\sa\to X_\sa\times S_\sa$ denote the open embedding of subanalytic sites. Recall that the authorized open coverings in $U_{X_{\sa}}$ are those obtained as intersections of coverings of $X_{\sa}$ with~$U$. We keep the notation $j$ for the morphism $j\times \Id_S: U\times S\to \XS$ and $\rho'$ for the morphism of sites $U\times S\to U_{X_{\sa}}\times S_{\sa}$. One easily checks from the notion of morphism of sites (see \cite[Chap.\,16]{K-Sch06} for details) that the items in Remark \ref{R} still hold in this framework with $f=j$.

\begin{lemma}\label{L:Groth}
Let $F$ be a $\pOS$-coherent $S$-locally constant sheaf on $U\times S$. Then, for any $\rho'_*p^{-1}_{U\times S}\sho_S$-module $\shl$, the natural morphism
\begin{starequation}\label{EGroth}
\rho'_*\bD'F\overset{L}{\otimes}_{\rho'_*\pOS}\shl \to\Rhom_{\rho'_*\pOS}(\rho'_*F,\shl)
\end{starequation}%
is an isomorphism.
\end{lemma}

\begin{proof}
The construction of \eqref{EGroth} is similar to the case of usual sheaves:
\begin{enumerate}
\item
For any $\shf,\shg,\shh\in\Mod(\rho'_*\pOS)$, we have a natural morphism
\[
\shhom_{\rho'_*\pOS}(\shf,\shg)\otimes_{\rho'_*\pOS}\shh\to \shhom_{\rho'_*\pOS}(\shf,\shg\otimes_{\rho'_*\pOS}\shh).
\]
We deduce a natural morphism
\begin{equation}\label{eq:ast}
\Rhom_{\rho'_*\pOS}(\shf,\shg)\overset{L}{\otimes}_{\rho'_*\pOS}\shh\to \Rhom_{\rho'_*\pOS}(\shf,\shg\overset{L}{\otimes}_{\rho'_*\pOS}\shh)
\end{equation}
by considering a flat resolution of $\shh$ and an injective resolution of $\shg$.

\item
We have $\bD'F:=\Rhom_{\pOS}(F, \pOS)$.
Recall that $\rho'_*$ commutes with $\shhom$ and $\Rhom$ hence $$\rho'_*\bD'F\simeq \Rhom_{\rho'_*\pOS}(\rho'_*F, \rho'_*\pOS).$$
\end{enumerate}
Therefore, the desired morphism \eqref{EGroth} is obtained from \eqref{eq:ast} with
\[
\shf=\rho'_*F,\, \shg=\rho'_*\pOS,\, \shh=\shl.
\]
To prove that it is an isomorphism it is sufficient to consider a locally finite covering of $U\times S$ by open subsets of the form $U'\times V$, $U'\in\Op(U_{X_{\sa}})$, $V\in \Op(S_{\sa})$, such that $F|_{U'\times V}$ admits a free $\pOS$-resolution $F^\cbbullet$ of finite length. It~follows that $\rho'_*F^\cbbullet$ is a $\shhom_{\rho'_*\pOS}(\cdot, \shg)$-injective resolution of $\rho'_*F|_{U'\times V}$, for any $\shg\in\nobreak\Mod(\rho'_*\pOS|_{U'\times V}$). Hence we are reduced to the case $F=\pOS$, which is clear.
\end{proof}

\begin{corollary}\label{C:RHDe}
Let $F$ be a $\pOS$-coherent $S$-locally constant sheaf on $U\times S$. Then we have a natural isomorphism in $\Mod_{\coh}(\shd_{U\times S/S})$:
\[
\RH^S(\bD'F)[-n]\simeq F\otimes_{\pOS}\sho_{U\times S}
\]
\end{corollary}

\begin{proof}
We have $$\RH^S(\bD'F)[-n]=\rho'^{-1}\Rhom_{\rho'_*\pOS}(\rho'_*\bD'F,\sho_{U\times S}^{t,S,\sharp})$$
Then we apply \eqref{EGroth} with $\shl=\sho_{U\times S}^{t,S,\sharp}$, recalling that $\rho'^{-1}$ commutes with tensor products and that $\rho'^{-1}\sho_{U\times S}^{t,S,\sharp}\simeq \sho_{U\times S}$.
\end{proof}

\begin{lemma}\label{L:1/7}
With the preceding notations, let $F\in\rD^\rb_\rc(p^{-1}_{U}\sho_S)$. Then there are natural isomorphisms in $\rD^\rb(\DXS)$
\begin{align}\tag{\protect\ref{L:1/7}$\,*$}\label{eq:L:1/7*}
\TH^S(j_!F)&\simeq\rho'^{-1}Rj_*R\shhom_{\rho'_*p_{U}^{-1}\sho_S}(\rho'_*F, j^{-1}\Db^{t, S,\sharp}_{\XS}),
\\
\tag{\protect\ref{L:1/7}$\,**$}\label{eq:L:1/7**}
\RH^S(j_!F)[-d_X]&\simeq
\rho'^{-1}Rj_*R\shhom_{\rho'_*p_{U}^{-1}\sho_S}(\rho'_*F, j^{-1}\sho^{t, S,\sharp}_{\XS}).
\end{align}
Moreover, if $U=X\setminus D$, where $D$ is a normal crossing divisor, and $F$ is $\pOS$-locally free of finite rank, $$\RH^S(j_!F)[-d_X]$$ is concentrated in degree zero.
\end{lemma}

\begin{proof}
The proofs of \eqref{eq:L:1/7*} and \eqref{eq:L:1/7**} are similar so we only prove \eqref{eq:L:1/7**}. Let us start by noting that, from \cite[(2.4.4),\,Prop.\,2.4.4]{K-S01}, one deduces an isomorphism of functors on $\rD^\rb(\C_{X_{\sa}\times S_{\sa}})$
$$\Rhom(\rho'_*j_!\C_{U\times S}, (\cbbullet))\to Rj_*\Rhom(\rho'_*C_{U\times S}, j^{-1}(\cbbullet))\simeq Rj_*j^{-1}(\cbbullet)$$ using the following facts:
\begin{itemize}
\item{$\rho'_*j_!\C_{U\times S}=(j_!\C_{U\times S})_{X_{\sa}\times S_{\sa}}$ since $X_{\sa}\times S_{\sa}$ is a $\sht'$-space as explained in the beginning of this section.}
\item{One derives isomorphism $(2.4.4)$ of \loccit\ using injective resolutions since $j^{-1}$ transforms injective objects into injective objects.}
\end{itemize}
We have:
\begin{align*}
\RH^S(j_!F)&\simeq \rho'^{-1} \rh_{\rho'_*p_{X}^{-1}\sho_S}(\rho'_*j_!F, \sho^{t,S,\sharp}_{\XS})\\
&\simeq \rho'^{-1} \rh_{\rho'_*\pOS}(\rho'_*(Rj_*F\otimes j_!\CC_{U\times S}) , \sho^{t,S,\sharp}_{\XS})\\
&\overset{(1)}\simeq \rho'^{-1} \rh_{\rho'_*p_{X}^{-1}\sho_S}(\rho'_*Rj_*F\otimes\rho'_* j_!\C_{U\times S}, \sho^{t,S,\sharp}_{\XS}))\\
&\simeq \rho'^{-1} \rh_{\rho'_*p_{X}^{-1}\sho_S}(\rho'_*Rj_*F, \Rhom(\rho'_*j_!\C_{U\times S}, \sho^{t,S,\sharp}_{\XS}))\\
&\overset{(2)}\simeq \rho'^{-1} \rh_{\rho'_*p_{X}^{-1}\sho_S}(Rj_*\rho'_*F, Rj_*j^{-1}\sho^{t,S,\sharp}_{\XS}))\\
&\overset{(3)}\simeq \rho'^{-1} Rj_*\rh_{\rho'_*p_{U}^{-1}\sho_S}(\rho'_*F, j^{-1}\sho^{t,S,\sharp}_{\XS}).
\end{align*}
The isomorphism (1) follows from Corollary \ref {L:210}\eqref{L:210a}, (2) follows from the commutation of~$\rho'_*$ and $j_*$ and the preceding remark, and (3) follows by adjunction.

Let us now prove the second part of the statement. We use a similar argument as in \cite[Lem.\,7.4]{Kashiwara84}. We assume that $F$ is $S$-locally free of rank $\ell$ and let $d_X=n$. Note that, according to Example \ref{Ex:RH}\eqref{Ex:RH1}, the result is true if $F$ is constant. Note also that the commutation of $\rho'^{-1}$ with $j^{-1}$ together with Corollary \ref{C:RHDe} imply that $\RH^S(j_!F)[-d_X]|_{(X\setminus Y)\times S}$ is concentrated in degree zero.

By \cite[Prop.\,5.2(i)]{MF-P14} and the exactness of $j^{-1}$, $j^{-1}\Db^{t,S,\sharp}_{\XS}$ is concentrated in degree zero, hence, by the assumption on $F$
$$R\shhom_{\rho'_*p_{U}^{-1}\sho_S}(\rho'_*F, j^{-1}\Db^{t, S,\sharp}_{\XS})$$ is also concentrated in degree zero.

Let us assume that $D=\{x_1\cdots x_d=0\}$ where $(x_1,\dots,x_n)$ are local coordinates in a neighborhood $\nb_X(x_o)$ of $x_o\in D$. Let also $s$ denote a local coordinate in $W\subset S$. Let us identify $\nb_D(x_o)$ with an open subset of $\C^n$ and $W$ to an open subset in $\C$. Let~$U'$ denote an open ball in $\nb_D(x_o)$ and $V$ denote an open ball in $W$. Let $\cG$ and $$(T_i(s), A_i(s))_{i=1,\dots,d},\quad s\in V$$ be given by Proposition \ref{prop:monodromyrepr} and Theorem \ref{th:Deligneext} with respect to $F$, such that $T_i(s)=\exp(-2i\pi A_i(s)), i=1,\dots, d$.

Let $\phi$ be a section of $\shhom_{\rho'_*p_{U}^{-1}\sho_S}(\rho'_*F, j^{-1}\Db^{t, S,\sharp}_{\XS})$ defined on $(U'\setminus D)\times V$. Note that, on any open set of the form $\gamma\times V$ with $\gamma$ open subanalytic simply connected in $U'\setminus D$, $$\exp(A_1(s)\log x_1)\cdots\exp(A_{d}(s)\log x_d)$$ is a matrix with holomorphic entries which are tempered in $X\times V$ (\cf \cite[Ex.\,5.1]{MF-P14}).
Then, the $\rho'_*p^ {-1}\sho_S$ linearity of $\phi$ implies that $$\exp(A_1(s)\log x_1)\cdots\exp(A_{d}(s)\log x_d)\phi$$ is a well-defined section of $\shhom_{\rho'_*p_{U}^{-1}\sho_S}(\rho'_*\cG, j^{-1}\Db^{t, S,\sharp}_{\XS})$ on $(U'\setminus D)\times V$. Since the open sets of the form $(U'\setminus D)\!\times\!V$ form a basis of the topology in $U_{X_{\sa}}\!\times \nobreak\!\nobreak S_{\sa}$, this means that the multiplication by $\prod_i\exp(A_i(s)\log x_i)$ defines an isomorphism
\begin{equation}\label{eq:ast2}
\shhom_{\rho'_*p_{U}^{-1}\sho_S}(\rho'_*F, j^{-1}\Db^{t, S,\sharp}_{\XS})\simeq \shhom_{\rho'_*p_{U}^{-1}\sho_S}(\rho'_*p^{-1}\cG, j^{-1}\Db^{t, S,\sharp}_{\XS})
\end{equation}
where $p^{-1}$$\cG$ is $S$-constant and free. Therefore the righthand side of \eqref{eq:ast2} is isomorphic to $({j^{-1}\Db_{\XS}^{t,S,\sharp}})^\ell$. By the $\Gamma(U'\times V',\cbbullet)$-acyclicity of $\Db_{\XS}^{t,S,\sharp}$ for arbitrary relatively compact subanalytic sets $U'$ and $V'$ in $X$ and $S$ respectively (\cf \cite[Prop.\,5.2(i)]{MF-P14}), we have $Rj_*j^{-1}\Db_{\XS}^{t,S,\sharp}\simeq j_*j^{-1}\Db_{\XS}^{t,S,\sharp}$, that is, it is concentrated in degree zero, hence $\TH^{S}(j_!F)$ is concentrated in degree zero. Moreover, by the construction, the isomorphism \eqref{eq:ast2} preserves the actions of $j^{-1}\rho'_!\sho_{\XS}$-modules and of $j^{-1}\rho'_!\shd_{\overline{X}\times \overline{S}}$-modules. Thus, by (\ref{E:3n}), \eqref{eq:ast2} induces an isomorphism of $\rho'^{-1}j_*j^{-1}\rho'_!\sho_{\XS}$-modules, hence of $\rho'^{-1}\rho'_!\sho_{\XS}\simeq\sho_{\XS}$-modules $$\RH^S(j_!F)[-n]\simeq \RH^S(j_!\bD'p^{-1}\cG)[-n].$$ So we may assume from the beginning that $F$ is constant and the result follows.
\end{proof}

\subsection{Functorial properties}\label{subsec:functorialprop}
In order to study the functorial properties of $\TH^S$ and $\RH^S$ useful for the sequel, we come back for a moment to the real framework in the first factor but of course remain complex in the $S$ factor.
Let $f:Y\to X$ be a morphism of real (or complex) analytic manifolds. We shall still denote by $f$ the associated morphism $f\times \id_S: Y\times S\to \XS$. We shall study the associated derived functor $\Df_*:\rD^\rb(\DYS)\mto\rD^\rb(\DXS)$. We begin with the relative version of \cite[Th.\,4.1]{Kashiwara84}:

\begin{theorem}\label{L:A3}
Let $f:Y\to X$ be a morphism of real analytic manifolds, let $F\in \rD^{\rb}_\rc({p_Y}^{-1}\sho_S)$, and assume that $f$ is proper on $\supp F$. Then we have a canonical isomorphism in $\rD(\DXS)$:
$$\Df_*\TH^S(F)\simeq \TH^S(Rf_*F).
$$
\end{theorem}

\begin{proof}
We can replace $F$ with a complex as in Proposition \ref{P:C1} and we argue by induction on its length, so it is sufficient to assume that $F$ is of the form $\CC_{U\times V}\otimes p_Y^{-1}\sho_S$, where $U$ (\resp $V$) is a relatively compact subanalytic open subset of~$Y$ (\resp $S$). In that case, one has $\TH^S(F)\simeq T\ho(\CC_{U\times V}, \Db_{Y\times S})$. On the other hand
$$Rf_*(\CC_{U\times V}\otimes p^{-1}_Y\sho_S)\simeq Rf_*(\CC_{U\times V})\otimes \pOS$$ hence
$$\TH^S(Rf_*(\CC_{U\times V}\otimes p^{-1}_Y\sho_S))\simeq T\ho(Rf_*(\CC_{U})\otimes \CC_{X\times V}, \Db_{\XS}).$$ Therefore the statement follows by the absolute case in \cite[Th.\,4.1]{Kashiwara84}.
\end{proof}

Recalling \eqref{E:3n} and adapting \cite[Lem.\,7.2]{Kashiwara84} one obtains:

\begin{theorem}\label{L:A4}
Let $f:Y\to X$ be a morphism of complex analytic manifolds, let $F\in \rD^{\rb}_\rc({p_Y}^{-1}\sho_S)$, and assume that $f$ is proper on $\supp F$. Then we have a canonical isomorphism in $\rD(\shd_{\XS/ S})$:
$$\Df_*\RH^S(F)\simeq \RH^S(Rf_*F).$$
\end{theorem}

\begin{comment}
Using Proposition \ref{P:C1}, Theorem \ref{L:A4} together with the line of proof of \cite[Th.\,5.8(ii) \& Prop.\,5.9]{K-S96}, we can conclude:

\begin{proposition}\label{P-ind}
Let $f:Y\to X$ be a morphism of complex analytic manifolds and let $F\in \rD^{\rb}_\rc({p_Y}^{-1}\sho_S)$. Then we have a natural morphism in $\rD(\shd_{\XS/ S})$
$$\Df^*\RH^S(F)[-d_X]\to \RH^S(f^{-1}F)[-d_Y],$$
which is an isomorphism if $f$ is a closed embedding.
\end{proposition}
\end{comment}

\begin{proposition}\label{P:Lis}
Let $X$ be a complex manifold. For any $F\in \rD^{\rb}_\rc(\pOS)$ and any $s_o\in S$, there is a natural morphism
$$Li^*_{s_o}\mathrm{RH^S}(F)[-d_X]\to \mathrm \tho(Li^*_{s_o} F,\sho_X)$$ which is an isomorphism, where we identify $X$ with $X\times\{s_o\}$ and $X_{\sa}$ with $X_{\sa}\times \{s_o\}$.
\end{proposition}

\begin{proof}
Let us construct the morphism. We have\enlargethispage{\baselineskip}%
\begin{align*}
Li^*_{s_o}\mathrm{RH^S}&(F)[-d_X]=p^{-1}(\sho_S/\mathfrak{m}_{s_o})\overset{L}{\otimes}_{\pOS}\RH^S(F)[-d_X]\\
&\simeq \rho'^{-1}\Big(\rho'_*(p^{-1}(\sho_S/\mathfrak{m}_{s_o}))\overset{L}{\otimes}_{\rho'_*{\pOS}}\Rhom_{\rho'_*{\pOS}}(\rho'_*F, \sho^{t,S,\sharp}_{\XS})\Big)\\
&\overset{(*)}\simeq \rho'^{-1}\Rhom_{\rho'_*(p^{-1}(\sho_S/\mathfrak{m}_{s_o}))}\Big((\rho'_*(p^{-1}(\sho_S/\mathfrak{m}_{s_o})\overset{L}{\otimes}_{\pOS}F),\\[-5pt]
&\hspace*{6cm}\rho'_*(p^{-1}(\sho_S/\mathfrak{m}_{s_o}))\overset{L}{\otimes}_{\rho'_*{\pOS}}\sho^{t,S,\sharp}_{\XS}\Big)\\
&\simeq\rho'^{-1}\Rhom_{\rho'_*(p^{-1}(\sho_S/\mathfrak{m}_{s_o}))}\Big(\rho'_*Li^*_{s_o}F, \rho'_*(p^{-1}(\sho_S/\mathfrak{m}_{s_o}))\overset{L}{\otimes}_{\rho'_*{\pOS}}\sho^{t,S,\sharp}_{\XS}\Big),
\end{align*}
where $(*)$ uses Corollary \ref{L:210}\eqref{L:210b}. So it remains to show
$$\rho'_*(p^{-1}(\sho_S/\mathfrak{m}_{s_o}))\overset{L}{\otimes}_{\rho'_*{\pOS}}\sho^{t,S,\sharp}_{\XS}|_{X_{\sa}\times\{s\}}\simeq \sho^t_X$$
Since $Li^*_{s_o}$ commutes with $\Rhom_{\shd_{\overline{X}\times \overline{S}}}(\sho_{\overline{X}\times\overline{S}},\cbbullet)$ it is sufficient to show
$$\rho'_*(p^{-1}(\sho_S/\mathfrak{m}_{s_o}))\overset{L}{\otimes}_{\rho'_*{\pOS}}\Db^{t,S,\sharp}_{\XS}|_{X_{\sa}\times\{s\}}\simeq \Db^t_X$$
Taking any local coordinate $s$ on $S$ centered at $s_o$, this amounts to showing
$$\{\Gamma(U\times V;\Db^{t,S,\sharp}_{\XS})\To{s}\Gamma(U\times V;\Db^{t,S,\sharp}_{\XS})\}\simeq\Gamma(U; \Db^t_X)$$ for any relatively compact open subsets $U\subset X$, $V\subset S$. We note that
\begin{itemize}
\item
$\Gamma(U\times V; \Db^{t,S,\sharp}_{\XS})=\Gamma(X\times V, \tho(U\times S, \Db_{\XS}))$,
\item
$Li^*_{s_o}\tho(U\times S, \Db_{\XS})|_{\{s_o\}}\simeq \tho(U,\Db_X)$ (\cf \cite[Th.\,4.5 (4.8)]{K-S96}).
\end{itemize}
Therefore
\[
\Gamma(X\times V; \tho(U\times S, \Db_{\XS}))\To{s}\Gamma(X\times V; \tho(U\times S, \Db_{\XS}))
\]
is quasi-isomorphic to
\[
\Gamma(X; \tho(U, \Db_{X}))=\Gamma(U; \Db^t_X),
\]
which gives the desired result.
\end{proof}

\section{Proof of the main results}

In order to apply the results of Sections \ref{subsec:complements}--\ref{subsec:functorialprop}, \emph{we continue to assume that $d_S=1$ and that $X$ is a complex manifold of complex dimension $d_X$}.

\begin{remark}[The locally constant case]
In view of Corollary \ref{C:RHDe} and Remark \ref{rem:RH}, Theorem \ref{T:11} is true if $F$ is an $S$-locally constant coherent sheaf. Similarly, the isomorphism of Theorem \ref{C:14} holds for $\shm=F\otimes_{\pOS}\cO_{\XS}$. Moreover, we recover Deligne's Riemann-Hilbert correspondence by means of $\RH^S$ as an equivalence between the category of $S$-local systems on $\XS$ and the category of coherent $\sho_{\XS}$-modules endowed with a relative flat connexion.
\end{remark}

\subsection{Proof of Theorem \ref{T:11}}
We first consider the setting of Section \ref{subsec:Deligneextension}, that is, we consider in $X$ a normal crossing divisor $D$, $j: (X\setminus D)\times S\to \XS$ denotes the open inclusion and $F$ is an $S$-locally constant coherent sheaf of $p^{-1}_{X\setminus D}\sho_S$-modules,

We set $E_F:=\sho_{(X\setminus D)\times S}\otimes_{p^{-1}_{X\setminus D}\sho_S}F$. The $\cO_{\XS}$-module $j_*E_F$ carries a natural structure of $\DXS$-module. Recall that in Theorem \ref{th:Deligneext} we denoted by $\tilde{E_F}$ the subsheaf of $j_*E_F$ consisting of local sections having moderate growth. Since $d_S=1$, according to Corollary \ref{cor:Deligneext}, $\tilde{E_F}$ is regular holonomic and has a characteristic variety contained in
\hbox{$(\pi^{-1}(D)\times S)\cup(T^*_{X} \XS)$}, where $\pi$ is the projection from $T^*X$ to $X$. Moreover $\tilde{E_F}\simeq \tilde{E_F}(*D)$, hence $R\Gamma_{[D\times S]}\tilde{E_F}=0$.

\begin{lemma}\label{RHV}
Assume that $F$ is $\pOS$-locally free of finite rank. Then we have, functorially in $F$,
$$\tilde{E_F}\simeq\RH^S(\pSol\tilde{E_F}).$$
\end{lemma}

\begin{proof}
We first rewrite the right-hand side as $\RH^S(j_!\bD'F)[-d_X]$. Indeed, since $\bD'F=\Sol E_F$ on $X\moins D$, we have a natural morphism $j_!\bD'F\to\Sol\wt{E_F}$. We will prove that it is an isomorphism. Since both complexes are $S$-$\C$-constructible, we are left with proving the same property after applying $Li_{s_o}^*$ for any $s_o\in S$, according to \cite[Prop.\,2.2]{MF-S12}. On the other hand, $Li_{s_o}^*\Sol\wt{E_F}=\Sol Li_{s_o}^*\wt{E_F}$ after \cite[Prop.\,2.1]{MF-S12}, and $\Sol Li_{s_o}^*\wt{E_F}=\Sol i_{s_o}^*\wt{E_F}$ since $\wt{E_F}$ is strict (Corollary \ref{cor:Deligneext}). Moreover, $i_{s_o}^*\wt{E_F}=\wt{i_{s_o}^*E_F}$ (Lemma \ref{fibers}). Similarly, one checks that $Li_{s_o}^*j_!\bD'F=j_!\bD'i_{s_o}^*F$. It is well-known that $j_!\bD'i_{s_o}^*F\to\Sol\wt {i_{s_o}^*E_F}$ is an isomorphism, hence the desired assertion.

According to Lemma \ref{L:1/7}, the complex $\RH^S(j_!\bD'F)[-d_X]$ is concentrated in degree zero, and according to Lemma \ref{L:Groth} (with $\shh= j^{-1}\sho^{t, S,\sharp}_{\XS})$, \eqref{eq:L:1/7**} and the $\pOS$-flatness of $F$ we have
\begin{align*}
\RH^S(j_!\bD'F)[-d_X]&\simeq
\rho'^{-1}Rj_*R\shhom_{\rho'_*\pOS}(\rho'_*\bD'F, j^{-1}\sho^{t, S,\sharp}_{\XS})\\
&\simeq
\rho'^{-1}Rj_*(\rho'_*F\overset{L}\otimes_{\rho'_*\pOS} j^{-1}\sho^{t, S,\sharp}_{\XS})\\
&\simeq
\rho'^{-1}Rj_*(\rho'_*F\otimes_{\rho'_*\pOS} j^{-1}\sho^{t, S,\sharp}_{\XS}).
\end{align*}
We shall prove that $\rho'^{-1}Rj_*(\rho'_*F{\otimes}_{\rho'_*p_{X\setminus D}^{-1}\sho_S} j^{-1}\sho^{t, S,\sharp}_{\XS})$ coincides with $\tilde{E_F}$. Firstly, applying the commutation of $\rho'^{-1}$ with $j^{-1}$ together Corollary \ref{C:RHDe} entails that $\rho'^{-1}Rj_*(\rho'_*F{\otimes}_{\rho'_*p_{X\setminus D}^{-1}\sho_S} j^{-1}\sho^{t, S,\sharp}_{\XS})$ and $\tilde{E_F}$ coincide on $(X\setminus D)\times S$. Therefore it is enough to prove that, for each $(y,s_o)\in D\times S$, for any $W\in \Op(X)$ running in a basis of neighborhoods of $y$ and for any $V\in\Op(S_{\sa})$ running in a basis of neighborhoods of $s_o$, we have
\begin{multline*}
\tag{a}
\varprojlim_{\substack{{V'\in \Op^c(S_{\sa}), s_o\in V',V'\Subset V}\\{U\in \Op^c(X_{\sa}),\,y\in U, \,U\Subset W}}}\hspace*{-3mm}
\Gamma((U\setminus D)\times V'; \rho'_*F\otimes_{\rho'_*p_{X\setminus D}^{-1}\sho_S} j^{-1}\sho^{t, S,\sharp}_{\XS})\\[-10pt]
= \Gamma((W\setminus D)\times V; \tilde{E_F}).
\end{multline*}

Recall that, by definition of $\otimes$, the subanalytic sheaf $\rho'_*F\otimes_{\rho'_*p_{X\setminus D}^{-1}\sho_S} j^{-1}\sho^{t, S,\sharp}_{\XS}$ is the sheaf associated to the presheaf defined by the formula:
$$\omega\times \omega'\mto \Gamma(\omega\times \omega'; F)\otimes \Gamma(\omega\times \omega'; j^{-1}\sho_{\XS}^{t,S,\sharp}), \quad \omega\in\Op((X\setminus D)_{X_{\sa}}),\,\omega'\in\Op(S_{\sa}).$$
(b) Therefore a section $h$ in
\[
\Gamma((U\setminus D)\times V'; \rho'_*F\otimes_{\rho'_*p_{X\setminus D}^{-1}\sho_S} j^{-1}\sho^{t, S,\sharp}_{\XS})
\]
is uniquely determined by the data of an open covering of $U\setminus D$ by simply connected Stein open subanalytic sets $(U_{\beta})_{\beta\in B}$ and of a family $(h_{\beta})_{\beta\in B}$ of vectors of~$\ell$ holomorphic functions, $h_{\beta}=({h_{i,\beta}})_{i=1,\dots,\ell}$, such that, for each $i=1,\dots,\ell$, $h_{i,\beta}$ is a holomorphic function defined in $U_{\beta}\times V'$ tempered in $X\times V'$ and such that the $h_{\beta}$ have the monodromy of $F$.

Taking local coordinates $(x_1,\dots, x_n)$ in $X$, $s$ in $S$, such that $D$ is given by an equation $x_i\cdots x_d=0$ in a neighborhood of $y=0\in D$, we may assume that
\begin{align*}
W&=B_{\delta}(0)=\{(x_1,\dots,x_n), |x_j|<\delta, j=1,\dots, d \},\quad V=B_{\delta}(s_o),\\
U&=B_{\epsilon}(0)\text{ and } V'=B_{\epsilon}(s_o),
\end{align*}
for some $\delta>0$ sufficiently small and arbitrary $\epsilon$ satisfying $0<\epsilon<\delta$.

\begin{enumerate}
\item
Let $f\in \Gamma((W\setminus D)\times V; \tilde{E_F})$.
We can decompose $W\setminus D$ as a union of
a finite family $(W_{\alpha,\epsilon_{\alpha}})_{\alpha\in A, \epsilon_{\alpha}>0}$ of open convex subsets such that, for each $(\alpha, \epsilon_{\alpha})$ and for each $U$, $W_{\alpha,\epsilon_{\alpha}}\cap U:=U^*_{\alpha,\epsilon_{\alpha}}$ is a convex open subset (hence Stein) in the conditions of Definition \ref{def:mod}.

For each $\alpha\in A$ we can choose an isomorphism
\[
\psi_{\alpha,\epsilon_{\alpha}}: F|_{W_{\alpha,\epsilon_{\alpha}}\times V}\simeq \pOS ^\ell|_{W_{\alpha,\epsilon_{\alpha}}\times V}
\]
which induces an isomorphism
\[
\phi_{\alpha,\epsilon_{\alpha}}:\tilde{E_F}|_{W_{\alpha,\epsilon_{\alpha}}\times V}\simeq \sho_{\XS}^\ell|_{W_{\alpha,\epsilon_{\alpha}}\times V}.
\]
Then, setting $\phi_{\alpha,\epsilon_{\alpha}}(f):=f_{\alpha,\epsilon_{\alpha}}$, the family $(f_{\alpha, \epsilon_{\alpha}})_{\alpha\in A}$ has the monodromy of $F$.

Let ${f_{i,\alpha,\epsilon_{\alpha}}}$ denote the $i$ component of $f_{\alpha,\epsilon_{\alpha}}$,$i=1,\dots,\ell$.
By construction, each $f_{i,\alpha,\epsilon_{\alpha}}$ is holomorphic (hence tempered) at any $(x,s)$ such that $x\in\partial U^*_{\alpha,\epsilon_{\alpha}}\setminus D$ and $s\in V$. Hence, by \cite[Prop.\,5.8]{MF-P14}, $f_{i,\alpha,\epsilon_{\alpha}}$ satisfies the estimation of Definition \ref{def:mod} if and only if, for each $i=1,\dots, \ell$, $\alpha \in A$, $f_{i,\alpha,\epsilon_{\alpha}}|_{U^*_{\alpha,\epsilon_{\alpha}}\times V'}$ is tempered at $X\times V'$.

Therefore, the family $(f_{\alpha,\epsilon_{\alpha}})_{\alpha\in A}$ defines an element of
\[
\Gamma((U\setminus D)\times V'; \rho'_*F\otimes_{\rho'_*p_{X\setminus D}^{-1}\sho_S} j^{-1}\sho^{t, S,\sharp}_{\XS}).
\]
With $\epsilon \to \delta$ we obtain $f$ as a section on $(W\setminus D)\times V$ of
\[
\rho'^{-1}Rj_*(\rho'_*F{\otimes}_{\rho'_*p_{X\setminus D}^{-1}\sho_S} j^{-1}\sho^{t, S,\sharp}_{\XS}).
\]

\item
By the characterization of the elements of
\[
\Gamma((U\setminus D)\times V'; \rho'_*F\otimes_{\rho'_*p_{X\setminus D}^{-1}\sho_S} j^{-1}\sho^{t, S,\sharp}_{\XS})
\]
given in (b), the converse is similar.\qedhere
\end{enumerate}
\end{proof}

\begin{lemma}\label{L:211}
Let $F$ be an $S$-locally constant coherent $\pOS$-module on $(X\setminus D)\times S$. Then $\RH^S(j_!F)\in \rD^\rb_\rhol(\DXS)$.
\end{lemma}

\begin{proof}
One considers the exact sequence of $S$-local systems
$$0\to F_\tors\to F\to F_\lf\to 0$$
where $F_\tors$ denotes the S-local system of $\pOS$-torsion sections of $F$ and $F_\lf$ denotes the quotient $F/F_\tors$. According to Lemma \ref{RHV} and Theorem \ref{th:Deligneext}, the result holds for~$F_\lf$. By the functoriality of $\RH^S$, it will hold for $F$ provided it holds for $F_\tors$.

So we now assume that $F$ is a torsion module. By definition, $F$ is $p^{-1}_{X\setminus D}\sho_S$-coherent so the support of $F$ is contained in $p^{-1}_{X\setminus D} S_0$, where $S_0$ is a discrete subset of $S$. Let us consider $s_o\in S_0$, $x_o\in D$ and let us prove that $\RH^S(j_!F)$ is regular holonomic in a neighborhood of $(x_o,s_o)$. If $s$ is a local coordinate vanishing at $s_o$, we can choose a power $N$ such that $s^NF=0$. Arguing by induction on $N$, one easily reduces to the case $N=1$. In that case, $F$ is isomorphic to $F'\boxtimes \sho_S/s\sho_S $ for some $\C$-local system~$F'$ on $X\setminus D$. We have
\begin{align*}
\RH^S(j_!F)&=\rho'^{-1}\Rhom_{\rho'_*{p^{-1}_{X}\sho_S}}(\rho'_*j_!F, \sho^{t,S,\sharp}_{\XS})\\
&\simeq \rho'^{-1}\Rhom(\rho'_*(j_!(F'\boxtimes \sho_S/s\sho_S), \sho^{t,S,\sharp}_{\XS}).
\end{align*}
On the other hand
\begin{align*}
j_!(F'\boxtimes \sho_S/s\sho_S)&\simeq j_!((F'\boxtimes\CC_S)\otimes p^{-1}_{X\setminus D}(\sho_S/s\sho_S)\\
&\simeq j_!(F'\boxtimes\CC_S)\otimes p^{-1}_{X}(\sho_S/s\sho_S)\\
&\simeq (j_!F'\boxtimes\CC_S)\otimes p^{-1}_{X}(\sho_S/s\sho_S),
\end{align*}
hence, according to \cite[Prop.\,4.7(1)]{MF-P14}, we get
$$\RH^S(j_!F)\simeq \Rhom_{p^{-1}_{X}\sho_S}(p^{-1}_{X}(\sho_S/s\sho_S),\tho(j_!F'\boxtimes \CC_{S},\sho_{\XS})).$$
Since $\tho(j_!F'\boxtimes \CC_{S},\sho_{\XS})$ is in $\Mod_\rhol(\DXS)$, the result follows.
\end{proof}

\begin{proof}[End of the proof of Theorem \ref{T:11}]
We now consider the general situation. We have the tools to follow the outline of Kashiwara's proof in the absolute case \cite[\S7.3]{Kashiwara84}.
We may assume that $F$ is an $S$\nobreakdash-$\CC$\nobreakdash-cons\-truc\-tible sheaf. Then we argue by induction on the dimension of a closed analytic set $Z$ such that $Z\times S\subset \supp F$. Let $Z_0$ be a closed analytic subset of $Z$ such that
\begin{enumerate}
\item{$F|_{(Z\setminus Z_0)\times S}$ is locally constant coherent over $\pOS$.}
\item{$Z\setminus Z_0$ is non singular.}
\end{enumerate}

By the induction hypothesis, $\mathrm{RH^S}(F_{Z_0\times S})$ belongs to $\rD^\rb_{\rhol}(\DXS)$. So we may assume that $F_{Z_0\times S}=0$. The question is local on $Z$. Consider then a projective morphism $\pi:X'\to X$ such that $X'$ is non singular and $\pi(X')=Z$, $Z'_0:=\pi^{-1}(Z_0)$ is a normal crossing divisor in $X'$ and $\pi: X'\setminus Z'_0\to Z\setminus Z_0$ is an isomorphism. Let $F'=\pi^{-1}(F)$. Then we obtain that $F'|_{Z'_0\times S}=0$, $F'_{(X'\setminus Z'_0)\times S}$ is locally constant coherent, and $R\pi_*F'=F$. Now the first part of Theorem \ref{T:11} follows straightforwardly from Proposition \ref{L:A4} and Lemma \ref {L:211}.

For the second part, we note that, since $\RH^S(F)$ is holonomic by the first part, Lemma \ref{L:comp}, together with \cite[Cor.\,3.9]{MF-S12}, implies $\bD\pSol(\RH^S(F))\simeq\bD F$, and we conclude by bi-duality \cite[Prop.\,2.23]{MF-S12}. As a consequence, $\pSol$ is essentially surjective from $\rD^\rb_{\rhol}(\DXS)$ to $\rD^\rb_\cc(\pOS)$.
\end{proof}

\begin{proof}[Proof of Corollary \ref{P:fin1}]
Set $\shm=\RH^S(F)$. By Proposition \ref{P:3.3}, it is enough to check that, if $F$ and $\bD F$ are perverse, then so are $\pSol\shm$ and $\pDR\shm$. But $\pSol\shm\simeq F$ by Theorem \ref{T:11}, and Lemma \ref{L:comp} gives $\pDR\shm\simeq\bD F$.
\end{proof}

\subsection{A reminder on mixed twistor $\shd$-modules}\label{subsec:MTMreminder}
Let $\sha$ be a subset of $\RR\times\CC$. The theory of $\sha$-mixed twistor $\shd$-modules on a complex manifold $X$ has been developed in \cite{Mochizuki11}, after the pure regular case considered in \cite{Bibi01c,Mochizuki07}, leading to the category $\MTM(X)_\sha$. Pure objects in $\MTM(X)_\sha$ are triples $(\shm',\shm'',C)$ (satisfying various conditions) consisting of coherent modules $\shm',\shm''$ over the sheaf of rings $\shr_{X\times\CC}$ of operators in $z\partial_{x_i}$ with coefficients in $\sho_{X\times\CC}$ ($z$ is the coordinate on the factor~$\CC$, corresponding to the ``twistor line''), and $C$ is a sesquilinear pairing that does not need to be made precise here. For our purpose, we restrict the setting to $S=\CC^*$, and we identify in a natural way $\shr_{X\times\CC^*}$ with $\DXS$ since $z$ is invertible. The holonomy property imposed to define $\MTM(X)_\sha$ implies the holonomy property of $\shm',\shm''$ on $\XS$ as defined in \hbox{\cite[\S3.4]{MF-S12}}. Moreover, $\shm',\shm''$ are strict. Mixed objects are endowed with a finite weight filtration $W_\sbullet$ whose graded pieces are pure objects. We say that \emph{$\shm$ underlies an $\sha$-mixed twistor $\shd$-module} if it is equal to the restriction to $\XS$ of the component~$\shm''$ of an object of $\MTM(X)_\sha$ (in~particular, we do neither care about the pairing $C$ nor the weight filtration). We denote by $\Mod_{\rhol,\MTM_\sha}(\DXS)$ the category whose objects are regular holonomic $\DXS$-modules underlying an object of $\MTM(X)_\sha$, and whose morphisms are induced from those of $\MTM(X):=\MTM(X)_{\RR\times\CC}$ (since $\MTM(X)_\sha$ is a full subcategory of $\MTM(X)$). It is an abelian category, as follows from the abelianity of $\MTM(X)_\sha$ and Proposition \ref{Rhol}. In Theorem \ref{C:14} we consider objects in this category with $\sha=\RR\times\{0\}$, but morphisms are in $\Mod(\DXS)$. One can define similarly $\Mod_{\hol,\MTM_\sha}(\DXS)$.

As an example, let us recall the main theorem of \cite{Mochizuki07} which implies that, when~$X$ is smooth projective, any irreducible regular holonomic $\shd_X$-module (or any finite direct sum of such) gives rise, by a suitable twistor deformation, to an object of $\Mod_{\rhol,\MTM}(\DXS)$. When all local monodromies of the corresponding de~Rham complex have eigenvalues of absolute value equal to one, we obtain an object of $\Mod_{\rhol,\MTM_\sha}(\DXS)$ with $\sha=\RR\times\{0\}$, as in Theorem \ref{C:14}.

The projective pushforward and duality functors are defined in $\MTM_\sha$ (\cf\cite[\S7.2.2\,\&\,\S13.3]{Mochizuki11} for the most general setting) and their restriction to the component $\shm''_{|\XS}$ are the corresponding functors for $\DXS$-modules. Also, Kashiwara's equivalence holds in $\MTM(X)_\sha$ (\cf \cite[Prop.\,7.2.8]{Mochizuki11}) and, together with Theorem \ref{th:kashiwaraequivalence}, we obtain that it holds in $\Mod_{\rhol,\MTM_\sha}(\DXS)$.

If $Y$ is a hypersurface in $X$, there is a localization functor in the category $\MTM(X)$, denoted by $[*Y]$ (\cf\cite[\S11.2.2]{Mochizuki11}). It preserves $\MTM(X)_\sha$ for any $\sha$, and when $\sha=\RR\times\{0\}$, it induces the functor $\shm(*Y):=\cO_{\XS}(*(\YS))\otimes_{\cO_{\XS}}\shm$. This explains our choice of $\sha$ in Theorem \ref{C:14}, \emph{that we fix from now on}. We have a natural morphism $\shm\to\shm(*Y)$ induced from the natural morphism in $\MTM(X)_\sha$, whose kernel and cokernel are supported on $Y$.

\begin{lemma}\label{lem:locrholMTM}
The $\DXS$-module $\shm(*Y)$ is an object of $\Mod_{\rhol,\MTM_\sha}(\DXS)$ if~$\shm$ is so.
\end{lemma}

\begin{proof}
It is a matter of proving regularity. Since $\shm$ is strict, so is $\shm(*Y)$, by flatness of $\cO_{\XS}(*(\YS))$ over $\cO_{\XS}$. As a consequence, the functor~$Li_{s_o}^*$ acts on them as~$i_{s_o}^*$, and we have $i_{s_o}^*(\shm(*Y))=(i_{s_o}^*\shm)(*Y):=\cO_X(*Y)\otimes_{\cO_X}(i_{s_o}^*\shm)$ for all $s_o$. We know that holonomic regularity of $\shd_X$-modules is preserved by localization along a hypersurface, so $i_{s_o}^*(\shm(*Y))$ is regular holonomic.\end{proof}

\begin{lemma}\label{lem:pullbackrholMTM}
Let $\pi:X'\to X$ be a morphism between complex manifolds and let~$\shm$ be an object of $\Mod_{\hol,\MTM_\sha}(\DXS)$ (\resp $\Mod_{\rhol,\MTM_\sha}(\DXS)$). Then the cohomology sheaves of the pullback complex $\Dpi^*\shm$ are holonomic (\resp regular holonomic) and strict.
\end{lemma}

\begin{proof}
We can decompose $\pi$ as the composition of a projection $X'\times X\to X$ with a closed inclusion of a smooth submanifold. The case of a projection is easily treated since it corresponds to the external product with $\sho_{\XS}$ over $\pOS$ and can be lifted at the level of $\MTM_\sha$ (\cf\cite[\S11.4.2]{Mochizuki11}). The regularity property is easily seen to be preserved.

Since the question is local, the case of a closed inclusion can be obtained by considering successive inclusions of closed smooth hypersurfaces. If $i:Y\hto X$ is smooth, Kashiwara's equivalence identifies $\shh^{-1}\Di^*\shm$ and $\shh^{0}\Di^*\shm$ with the kernel and cokernel of $\shm\to\shm(*Y)$, and the formers are therefore in $\Mod_{\hol,\MTM_\sha}(\DXS)$, \resp $\Mod_{\hol,\MTM_\sha}(\DXS)$, according to Lemma \ref{lem:locrholMTM}.
\end{proof}

Together with strictness, these will be the main properties used the proof below.

\begin{remark}\label{rem:proviso}
With $\sha$ and $S_0$ as in Remark \ref{rem:intro} of the introduction, the statements above hold true provided that we replace $S=\CC^*$ with $S\moins S_0$, and the proof of Theorem \ref{C:14} given below extends with this proviso.
\end{remark}

\subsection{Proof of Theorem \ref{C:14}}\label{subsec:C14}
Let $\shm,\shn$ be holonomic $\DXS$-modules and let $F$ be an object of $\rD^\rb_\cc(\pOS)$. Recall (\cf \cite[(2.6.7)]{K-S90}) that we have a natural isomorphism in $\rD(\pOS)$:
\[
\Rhom_{\cD_{\XS/S}}(\shm,\Rhom_{\pOS}(F,\cO_{\XS}))\simeq\Rhom_{\pOS}(F,\Sol\shm),
\]
which is bi-functorial with respect to $\shm,F$. By composing with \eqref{E:Theta2}, we obtain a bi-functorial morphism
\[
\Rhom_{\DXS}(\shm, \RH^S(F)[-d_X])\to\Rhom_{\pOS}(F,\Sol\shm).
\]
Choosing $F=\Sol\shn$ finally produces a bi-functorial morphism
\begin{equation}\label{eq:comparaison}
\Rhom_{\DXS}(\shm, \RH^S(\pSol\shn))\to\Rhom_{\pOS}(\Sol\shn,\Sol\shm).
\end{equation}
If \eqref{eq:comparaison} is an isomorphism, then by taking global sections we find a bi-functorial isomorphism
\[
\Hom_{\DXS}(\shm, \RH^S(\pSol\shn))\to\Hom_{\pOS}(\Sol\shn,\Sol\shm)
\]
and the isomorphism $(*)$ stated in Theorem \ref{C:14} is obtained as that corresponding to $\id_{\Sol\shm}$ when $\shn=\shm$, while $(**)$ follows by applying $(*)$ to $\shn$. We consider the following three statements.

\begin{enumeratea}
\item\label{enum:constructiblea}\itshape
If $\shm,\shn\in\Mod_{\rhol,\MTM_\sha}(\DXS)$, then the complex
\[
\Rhom_{\DXS}(\shm, \RH^S(\pSol\shn))
\]
is $S$-$\CC$-constructible.

\item\label{enum:constructibleb}
If $\shm,\shn\in\Mod_{\rhol,\MTM_\sha}(\DXS)$, then \eqref{eq:comparaison} is an isomorphism.

\item\label{enum:constructiblec}
If $\shm\in\Mod_{\rhol,\MTM_\sha}(\DXS)$, then so does $\RH^S(\pSol\shm)$.
\end{enumeratea}

We also denote by $\eqref{enum:constructiblea}_n$ (\resp $\eqref{enum:constructibleb}_n$, $\eqref{enum:constructiblec}_n$) the statement \eqref{enum:constructiblea} (\resp $\eqref{enum:constructibleb}$, $\eqref{enum:constructiblec}$) for $\shm,\shn$ with support in $X$ of dimension $\leq n$. The first part of Theorem \ref{C:14} follows from~$\eqref{enum:constructibleb}_n$ for any $n\geq0$ by setting $\shn=\shm$. We will prove $\eqref{enum:constructiblec}_n$ for any $n\geq0$, which will be enough, according to the lemma below.

\begin{lemma}\label{lem:abc}
For any $n\geq0$, the statements $\eqref{enum:constructiblea}_n$, $\eqref{enum:constructibleb}_n$ and $\eqref{enum:constructiblec}_n$ are equivalent.
\end{lemma}

Notice already that $\eqref{enum:constructibleb}_n\implique\eqref{enum:constructiblec}_n$ is obtained exactly as \eqref{eq:comparaison}${}\implique{}$Theorem \ref{C:14}.

\begin{proof}[Proof of $\eqref{enum:constructiblea}_n\implique\eqref{enum:constructibleb}_n$]
Assume that $\eqref{enum:constructiblea}_n$ holds true. We will prove that so does $\eqref{enum:constructibleb}_n$ by applying the functor $Li_{s_o}^*$ for any $s_o\in S$, to reduce to \cite[Cor.\,8.6]{Kashiwara84}.

Assume that $\shm,\shn\in\Mod_{\rhol,\MTM_\sha}(\DXS)$ have support of dimension~$\leq\nobreak n$. Accor\-ding to $\eqref{enum:constructiblea}_n$ and to \cite[Prop.\,2.2]{MF-S12}, \eqref{eq:comparaison} is an isomorphism as soon as $Li_{s_o}^*\eqref{eq:comparaison}$ is an isomorphism for any $s_o\in S$, since we already know that the complex $\Rhom_{\pOS}(\Sol\shn,\Sol\shm)$ is $S$-$\CC$-constructible, as $\Sol\shm,\Sol\shn$ are so (\cf \cite[Th.\,3.7]{MF-S12}).

On the other hand, arguing as for \cite[Prop.\,2.1]{MF-S12} by using \cite[(A.10)]{Kashiwara03} (together with \cite[(2.6.7)]{K-S90}), $Li_{s_o}^*\eqref{eq:comparaison}$ is the morphism
\[
\Rhom_{\cD_X}(Li_{s_o}^*\shm, Li_{s_o}^*\RH^S(\pSol\shn))\to\Rhom_{\CC}(Li_{s_o}^*\Sol\shn,Li_{s_o}^*\Sol\shm),
\]
and still by \cite[Prop.\,2.1]{MF-S12}, we can replace the right-hand side with
\[
\Rhom_{\CC}(\Sol Li_{s_o}^*\shn,\Sol Li_{s_o}^*\shm).
\]
On the other hand, by Proposition \ref{P:Lis}, we can replace $Li_{s_o}^*\RH^S(\pSol\shn)$ with $\tho(Li_{s_o}^*\Sol\shn,\cO_X)$. In such a way, $Li_{s_o}^*\eqref{eq:comparaison}$ is an isomorphism, according to \cite[Cor.\,8.6]{Kashiwara84}, and this ends the proof of $\eqref{enum:constructibleb}_n$.
\end{proof}

\begin{proof}[Proof of $\eqref{enum:constructiblec}_n\implique\eqref{enum:constructiblea}_n$]
This is an immediate consequence of Proposition \ref{P:Theta5} below.
\end{proof}

\begin{proposition}\label{P:Theta5}
Let $\shm, \, \shn \in \Mod_{\hol,\MTM_\sha}(\DXS)$. Then $\Rhom_{\DXS}(\shm, \shn)$ is $S$-$\C$-constructible.
\end{proposition}

\begin{proof}
We shall apply the following relative versions of Lemmas 1.8, 1.9 and Proposition 4.7 of \cite{Kashiwara78}. Since they are trivial adaptations of the original ones, we omit their proof.

\begin{lemma}\label{L:K1}
Let $\shm\in\rD^\rb(\DXS)^{\op}$ and $\shn\in\rD^\rb(\DXS)$. Then $$\Rhom_{\DXS}(\Omega^{d_X}_{\XS/S},\shm\overset{L}{\otimes}_{\sho_{\XS}}\shn)\simeq
\shm\overset{L}{\otimes}_{\DXS}\shn[-d_X].$$
\end{lemma}

\begin{lemma}\label{L:K2}
Let $\shm\in\rD^\rb_{\coh}(\DXS)$ and let $\shn\in\rD^\rb(\DXS)$. Then we have an isomorphism
$$\Rhom_{\DXS}(\shm,\shn)\simeq\Rhom_{\DXS}(\Omega_{\XS/S}^n, \bD'\shm\overset{L}\otimes_{\sho_{\XS}}\shn)[d_X].$$
\end{lemma}

\begin{lemma}\label{L:K3}
Let $\shm$ and $\shn$ be two strict $\DXS$-modules. Then
$$\shm\overset{L}{\otimes}_{\sho_{\XS}}\shn\simeq \Di^*_X(\shm\boxtimes \shn),$$
where $X$ is identified to the diagonal of $X\times X$ by the inclusion $i_X$ and $\boxtimes$ denotes the external product over $\pOS$.
\end{lemma}

Let us now return to the proof of Proposition \ref{P:Theta5}. The assumption on $\shm$ and~$\shn$ entails that $\bD'\shm\boxtimes\shn\in\Mod_{\hol,\MTM_\sha}(\DXS)$, according to \cite[Th.\,13.3.1\,\& Prop.\,11.4.6]{Mochizuki11}. So the cohomology sheaves of \hbox{$\Di^*_X(\bD'\shm\boxtimes \shn)$} are holonomic, according to Lemma \ref{lem:pullbackrholMTM}. In view of Lemmas \ref{L:K1}, \ref{L:K2} and \ref{L:K3}, the result follows from \cite[Th\,3.7]{MF-S12}, which entails that the de~Rham complex of an object of $\rD^\rb_\hol(\DXS)$ is $S$-$\C$-constructible.
\end{proof}

\subsection{End of the proof of Theorem \ref{C:14}}
We first notice the following lemma.

\begin{lemma}\label{lem:cDtype}
Let $X'$ be a complex manifold and let $Y'$ be a divisor with normal crossing in $X'$. Let $\shm'$ be a $\DXpS$-module of D-type along $Y'$ (Definition \ref{D-t}) underlying an object of $\MTM(X')$. Then~$\eqref{enum:constructiblec}$ holds for $\shm'$.
\end{lemma}

\begin{proof}
By Proposition \ref{T:D-T}, we can write $\shm'=\tilde{E_F}$ for some locally free $p_{X'\moins Y'}^{-1}\sho_S$-module $F$. Then Lemma \ref{RHV} gives $\shm'\simeq\RH^S(\pSol\shm')$. This implies that $\RH^S(\pSol\shm')$ underlies an object of $\MTM(X')$.
\end{proof}

\subsubsection*{Proof of $\eqref{enum:constructiblec}_n$ by induction on $n$}
We note that $\eqref{enum:constructiblea}_0$ reduces to the case $X=\mathrm{pt}$ by Kashiwara's equivalence recalled in Section \ref{subsec:MTMreminder}, and is nothing but Lemma \ref{lem:cDtype} with $X=\mathrm{pt}$. Hence $\eqref{enum:constructiblec}_0$ holds and we can assume $n\geq1$.
By induction, we assume any of $\eqref{enum:constructiblea}_{n-1}$, $\eqref{enum:constructibleb}_{n-1}$ and $\eqref{enum:constructiblec}_{n-1}$ is true, according to Lemma~\ref{lem:abc}.

\begin{proof}[Reduction to the localized case]
As explained in Section \ref{subsec:MTMreminder}, the main reason for restricting to the category $\Mod_{\rhol,\MTM_\sha}(\DXS)$ is that, if $Y$ is a hypersurface in $X$, we have a localization morphism $\shm\to\shm(*Y)$ in this category (Lemma \ref{lem:locrholMTM}). The localization enables us to argue by induction on the dimension of the support of $\shm$.

We will prove $\eqref{enum:constructiblea}_n$ with $\shn=\shm$, and this will imply $\eqref{enum:constructiblec}_n$ for $\shm$, according to Lemma \ref{lem:abc}. We are thus proving a constructibility property, which is a local one (\cf Section \ref{subsec:Sconstructible}), so the question is local on $X$. Let $Z\subset X$ denote the support of~$\shm$. We assume that $\dim Z=n$.

Together with the induction hypothesis, we also assume that there exists a hypersurface~$Y$ in $X$ intersecting $Z$ in dimension $\leq n-1$ such that $\eqref{enum:constructiblec}_n$ holds true for $\shm(*Y)\in\Mod_{\rhol,\MTM_\sha}(\DXS)$. By the abelianity of the latter category (\cf Section~\ref{subsec:MTMreminder}), the kernel $\shk$, image $\cI$ and cokernel $\shc$ of the natural morphism $\shm\to\shm(*Y)$ are objects of $\Mod_{\rhol,\MTM_\sha}(\DXS)$, and $\shk,\shc$ have support of dimension $\leq n-1$. We can thus apply $\eqref{enum:constructiblec}_{n-1}$ to them by the induction hypothesis.

From the distinguished triangle
\[
\RH^S\pSol\cI\to\RH^S\pSol\shm[*Y]\to\RH^S\pSol\shc\To{+1}
\]
and according to Proposition \ref{P:Theta5} applied to the last two terms, we find that $\Rhom_{\DXS}(\shm, \RH^S(\pSol\cI))$ is $S$-$\C$-constructible. Then, from the distinguished triangle
\[
\RH^S\pSol\shm\to\RH^S\pSol\cI\to\RH^S\pSol\shk[1]\To{+1}
\]
and Proposition \ref{P:Theta5} applied similarly, we conclude that
\[
\Rhom_{\DXS}(\shm, \RH^S(\pSol\shm))
\]
is $S$-$\C$-constructible, that is, $\eqref{enum:constructiblea}_n$ holds for $\shn=\shm$. Then $\eqref{enum:constructibleb}_n$ for $\shn=\shm$ also holds, and then $\eqref{enum:constructiblec}_n$ for $\shm$ too.
\end{proof}

\begin{proof}[Proof of $\eqref{enum:constructiblec}_n$ in the localized case]
Recall that we work locally on $X$. Given $\shm\in\Mod_{\rhol,\MTM_\sha}(\DXS)$, we wish to find a hypersurface~$Y$ intersecting $\supp\shm=Z$ in dimension $\leq n-1$ such that $\eqref{enum:constructiblec}_n$ holds true for $\shn:=\shm(*Y)$. We choose~$Y$ such that, moreover, $Z^*:=Z\moins Z\cap Y$ is a smooth manifold of dimension $n$ and~$\shm_{|X\moins Y}$ is the pushforward (by the inclusion $Z^*\hto X\moins Y$) of an object underlying a smooth twistor $\shd$-module on $Z^*$, \ie an admissible variation of twistor structure (\cf\hbox{\cite[\S9]{Mochizuki11}}). Lemma \ref{lem:locrholMTM} entails that $\shn\in\Mod_{\rhol,\MTM_\sha}(\DXS)$.

We can then choose a projective morphism $\pi:X'\to X$ such that $X'$ is a complex manifold, $\pi^{-1}(Y)=Y'$ is a normal crossing divisor in $X'$, and~$\pi$ induces an isomorphism \hbox{$X'\moins Y'\isom Z\moins Z\cap Y$} (hence $d_{X'}\!=\!d_Z$). We set \hbox{$\shm'\!=\!(\Dpi^*\shn)[d_Z-d_X]$}. It~satisfies $\shm'=\shh^0(\shm')=\shh^0(\shm')(*Y')$. By Lemma \ref{lem:pullbackrholMTM}, it is an object of $\Mod_\rhol(\DXpS)$ and it is strict. We conclude that $\shm'$ is of D-type (Definition~\ref{D-t}) and, by Lemma \ref{lem:cDtype}, $\eqref{enum:constructiblec}$ holds for~$\shm'$.

On the other hand, the adjunction isomorphism of Corollary \ref{C:rel3} induces an isomorphism
\[
\Hom_{\DXpS}(\shm',\shm')\isom \Hom_{\DXS}(\Dpi_*\shm', \shn),
\]
and $\id_{\shm'}$ provides the adjunction morphism $\Dpi_*\shm'\to\shm$. Note that $\Dpi_*\shm'=(\Dpi_*\shm')(*Y)$, so finally $\Dpi_*\shm'\simeq\shn$. We end the proof that $\eqref{enum:constructiblec}_n$ holds for~$\shn$ by considering the isomorphisms
\[
\shn\simeq\Dpi_*\shm'\overset{\ref{lem:cDtype}}\simeq\Dpi_*\RH^S(\pSol\shm')\overset{\ref{L:A4}}\simeq\RH^S(R\pi_*\pSol\shm')\overset{\ref{T:hol}\eqref{T:hol3}}\simeq\RH^S(\pSol\shn).\qedhere
\]
\end{proof}

\subsubsection*{Proof of the functoriality in Theorem \ref{C:14}}
Let $\varphi:\shm\to\shn$ be a morphism in the category $\Mod(\DXS)$ and assume that $\shm,\shn$ are objects of $\Mod_{\rhol,\MTM_\sha}(\DXS)$, so that \eqref{eq:comparaison} is an isomorphism, according to the above proof. By the bi-functoriality of~\eqref{eq:comparaison}, we have a commutative diagram
\[
\xymatrix@R=.5cm{
\Hom_{\DXS}(\shm, \RH^S(\pSol\shm))\ar[r]^-\sim\ar[d]_{(\RH^S\pSol\varphi)\circ{}}&\Hom_{\pOS}(\Sol\shm,\Sol\shm)\ar[d]^{\circ(\Sol\varphi)}\\
\Hom_{\DXS}(\shm, \RH^S(\pSol\shn))\ar[r]^-\sim&\Hom_{\pOS}(\Sol\shn,\Sol\shm)\\
\Hom_{\DXS}(\shn, \RH^S(\pSol\shn))\ar[r]^-\sim\ar[u]^{{}\circ\varphi}&\Hom_{\pOS}(\Sol\shn,\Sol\shn)\ar[u]_{(\Sol\varphi)\circ{}}
}
\]
If $\eta_\shm:\shm\isom\RH^S(\pSol\shm)$ corresponds to $\Id_{\Sol\shm}$ via the horizontal isomorphism, and similarly for $\eta_\shn$, then the diagram shows that $(\RH^S\pSol\varphi)\circ\eta_\shm=\eta_\shn\circ\varphi$, since both correspond to $\Sol\varphi$ via the middle horizontal isomorphism.\qed

\refstepcounter{section}
\renewcommand{\thesection}{A}
\section*{Appendix. Locally constant sheaves of \texorpdfstring{$\pOS$}{pOS}-modules}

In this appendix, $S$ denotes a complex analytic space which is not necessarily reduced, $S_\red$ denotes the associated reduced space and $\sho_S$ denotes its structure sheaf (a sheaf of rings on $S_\red$). When there is no risk of confusion, we will use the notation~$S$ instead of $S_\red$ as the underlying space. We state the results we need without proofs, which are straightforward.

\begin{comment}
\subsection{$S$-constant sheaves on $X$}
\end{comment}

An \emph{$S$-constant sheaf of $\pOS$-modules} on $\XS$ is a sheaf of the form $p^{-1}\cG$ for some sheaf $\cG$ of $\sho_S$-modules. We say that $\cF$ is $\pOS$-coherent if $\cG$ is $\sho_S$-coherent. Similarly, there is the notion of $S$-constant sheaf of $\CC$-vector spaces.

\begin{comment}
Recall (\cf \cite[p.\,110]{Godement64}) that giving a sheaf $\cF$ on a topological space $Z$ is equivalent to giving a sheaf space $\mu:\wt\cF\to Z$, where $\mu$ is a local homeomorphism. The sheaf space $\wt{p^{-1}\cG}$ is nothing but the product $\id_X\times\mu:X\times\wt \cG\to \XS$. This interpretation is useful for the next proposition.
\end{comment}

\begin{proposition}[$S$-constant sheaves] \label{sor}
Let $X$ be a topological space. An~$S$\nobreakdash-constant sheaf $\cF$ of $\pOS$-modules on $\XS$ satisfies the following properties.
\begin{enumerate}
\item\label{sor1}
If $f:Y\to X$ is a continuous map, then $f^{-1}\cF$ is $S$-constant on $Y$.
\item\label{sor2}
If $U\subset X$ is a \emph{connected} open set in $X$ and $x\in U$, the natural morphism $p_{U,*}\cF\to i_x^{-1}\cF$ is an isomorphism of $\sho_S$-modules. Conversely, if $X$ is a connected topological space and $\cF$ is a sheaf of $\pOS$-modules on $\XS$ such that the natural morphism $p_*\cF\to i_x^{-1}\cF$ is an isomorphism of $\sho_S$-modules for each $x\in X$, then $\cF$ is $S$-constant. In particular, if $X$ is a connected topological space and $\cG$ is a sheaf of $\sho_S$-modules, the sheaf $p_*p^{-1}\cG$ is naturally identified to $\cG$.
\item\label{sor3}
If $\cG,\cG'$ are $\sho_S$-modules, there are canonical isomorphisms:
$$
p^{-1}(\cG\oplus\cG')\simeq p^{-1}\cG\oplus p^{-1}\cG' ,\qquad p^{-1}(\cG\otimes_{\sho_S}\cG')\simeq p^{-1}\cG\otimes_{\pOS} p^{-1}\cG',
$$
and if moreover $\cG$ is $\sho_S$-coherent or $X$ is \emph{locally connected}
$$
p^{-1}\shhom_{\sho_S}(\cG,\cG')\simeq\shhom_{\pOS}(p^{-1}\cG,p^{-1}\cG').
$$
\item\label{sor4}
The functor $p^{-1}$, from the category of $\sho_S$-modules to that of $\pOS$-modules is exact. Moreover, if $X$ is connected, this functor is fully faithful.
\item\label{sor5}
If $X$ is a connected topological space, the kernel, the image and the cokernel of a morphism between $S$-constant sheaves of $\pOS$-modules are $S$-constant sheaves of $\pOS$-modules.
\end{enumerate}
\end{proposition}

We say that a sheaf $\cF$ on $\XS$ of $\pOS$-modules is \emph{$S$-locally constant} if each point $(x,s)\in \XS$ has a neighborhood on which $\cF$ is $S$-constant. We then say that $\cF$ is $\pOS$-coherent if it is locally (on $\XS$) isomorphic to the pull-back by~$p$ of a $\sho_S$-coherent sheaf.

\begin{proposition}\label{prop:cG}
If $X$ is connected and locally connected, and if $\cF$ is $S$-locally constant on $\XS$, then there exists a sheaf $\cG$ of $\sho_S$-modules such that, locally on $\XS$, we have $\cF\simeq p^{-1}\cG$. We can choose for $\cG$ any of the sheaves $i_x^{-1}\cF$ for $x\in X$. Moreover, $\cF$ is $\pOS$-coherent if and only if $\cG$ is $\sho_S$-coherent.
\end{proposition}

In other words, the isomorphisms are locally defined, but the sheaf $\cG$ exists globally on $S$. However, this sheaf is not unique.

\begin{comment}
\begin{proof}
For $x_o\in X$, let us set $\cG^{(x_o)}=i_{x_o}^{-1}\cF$. If $U$ is a connected neighborhood of~$x_o$ and $V$ is a neighborhood of $s_o$ such that $\cF_{|U\times V}\simeq p_U^{-1}\cG_V$ for some sheaf~$\cG_V$, we have $\cG_V\simeq p_{U_*}\cF_{|U\times V}\simeq i_{x_o}^{-1}\cF_{|U\times V}=\cG^{(x_o)}{}_{|V}$, according to the previous proposition. Then the set of $x$ for which $\cG^{(x)}$ is locally isomorphic to $\cG^{(x_o)}$ is open and closed in $X$, hence equal to $X$ since nonempty.
\end{proof}
\end{comment}

\begin{proposition}
Assume that $X$ is locally connected. Let $\cF$ be a sheaf of $\pOS$-modules on $\XS$. Then $\cF$ is an $S$-locally constant sheaf of $\pOS$-modules if and only if it is $S$-locally constant as a sheaf of $\CC$-vector spaces.
\end{proposition}

\begin{comment}
\begin{proof}
This is a local question, and one can assume that $\cF\simeq p^{-1}\cG$, where $\cG$ is a sheaf $\CC$-vector spaces. Let us consider the adjunction morphism $p^{-1}p_*\cF\to\nobreak\cF$. Since $X$ is locally connected, we can assume in the local setting that $X$ is connected. Then $\cG\to p_*p^{-1}\cG$ is an isomorphism of sheaves of $\CC$-vector spaces, and thus so is $p^{-1}p_*\cF\to\cF$. This morphism being a morphism in the category of $\pOS$-modules, it is also an isomorphism in this category.
\end{proof}
\end{comment}

\begin{proposition}\label{prop:stab}
If $\cF,\cF'$ are $S$-locally constant on $\XS$, and $\varphi:\cF\to\cF'$ is $\pOS$-linear, then $\cF\oplus\cF'$, $\cF\otimes_{\pOS}\cF'$, $\shhom_{\pOS}(\cF,\cF')$, $\ker\varphi$, $\im\varphi$ and $\coker\varphi$ are also $S$-locally constant. If $\cF$ and $\cF'$ are moreover $\pOS$-coherent, so are these sheaves.
\end{proposition}

\begin{corollary}
The category of $S$-locally constant sheaves of $\pOS$-modules (\resp and $\pOS$-coherent) is a full abelian subcategory of the category of sheaves of $\pOS$-modules.
\end{corollary}

\begin{corollary}
Let $0\to\cF'\to\cF\to\cF''\to0$ be an exact sequence of sheaves of $\pOS$-modules. If $\cF,\cF'$ (\resp $\cF,\cF''$) are $S$-locally constant (\resp and coherent), then so are $\cF'$ (\resp $\cF''$).
\end{corollary}

\begin{comment}
\subsection{The monodromy representation}
\end{comment}

\begin{proposition}\label{P:const}
Set $I=[0,1]$. Let $\cF$ be an $S$-locally constant sheaf of $\pOS$-modules on $\XS$, with $X=I$ or $X=I\times I$. Then $\cF$ is $S$-constant.
\end{proposition}

\begin{comment}
\begin{proof}
Let us start with $X=I$ and let us set $\cG=\cF_{|\{0\}\times S}$. Fix $s_o\in S$. Since~$I$ is compact, their exists a neighborhood $V$ of $s_o$, an integer $N\geq2$ and for each $k=0,\dots,N-2$, an isomorphism $\varphi_k:\cF_{|X_k\times V}\isom (p^{-1}\cG)_{X_k\times V}$, with $X_k=[k/N,(k+2)/N]$. On $(X_k\cap X_{k+1})\times V$ one has $\varphi_k=p^{-1}\psi_{k,k+1}\circ\varphi_{k+1}$, where $\psi_{k,k+1}$ is an automorphism of $\cG_{|V}$. Then one can replace $\varphi_0$ with $\varphi'_0$ defined on an interval of length $3/N$. By induction on the length of the first interval one defines $\varphi:\cF_{|X\times V}\to (p^{-1}\cG)_{X\times V}$.

We argue similarly for $X=I\times I$. The previous reasoning gives isomorphisms on $X_k\times I\times V$, and then on $I\times I\times V$.
\end{proof}
\end{comment}

Let $\gamma:I\to X$ be a continuous map, with $\gamma(0)=x_o$, $\gamma(1)=x_1$. If $\cF$ is $S$-locally constant on $\XS$, then so is $\gamma^{-1}\cF$ on $I\times S$, hence it is $S$-constant, and it defines an isomorphism $T_\gamma:i_{x_o}^{-1}\cF\isom i_{x_1}^{-1}\cF$ of $\sho_S$-modules.

\begin{proposition}\label{prop:monodromy}
If $\gamma$ and $\gamma'$ are homotopic with fixed endpoints, then $T_\gamma=T_{\gamma'}$. If $\gamma$ and $\gamma'$ can be composed, we have $T_{\gamma\cdot\gamma'}=T_\gamma\circ T_{\gamma'}$.
\end{proposition}

\begin{comment}
\begin{proof}
Let $H:I\times I\to X$ be an homotopy between $\gamma$ and $\gamma'$. Then $H^{-1}\cF$ is $S$\nobreakdash-constant and the first statement follows. The second statement is standard.
\end{proof}
\end{comment}

Let us now assume that $X$ is connected and locally path-connected, and let us fix a base point $x_o\in X$. We consider the category $\Rep_{\sho_S}(\pi_1(X,x_o))$ whose objects are representations $\rho:\pi_1(X,x_o)\to\Aut_{\sho_S}(\cG)$ for some sheaf $\cG$ of $\sho_S$-modules and the morphisms $\rho\to\rho'$ are $\sho_S$-linear morphisms $\varphi:\cG\to\cG'$ which satisfy $\rho'(\gamma)\circ\varphi=\varphi\circ\rho(\gamma)$ for any $\gamma\in\pi_1(X,x_o)$.

Given an $S$-locally constant sheaf $\cF$ on $\XS$, Proposition \ref{prop:monodromy} shows that $\gamma\mto\nobreak T_\gamma$ defines a representation $\rho:\pi_1(X,x_o)\to\Aut_{\sho_S}(i_{x_o}^{-1}\cF)$, called the monodromy representation attached to $\cF$. A morphism of $S$-locally constant sheaves obviously gives rise to a morphism of their monodromy representation. We thus get a functor from the category of $S$-constant local systems of $\pOS$-modules to $\Rep_{\sho_S}(\pi_1(X,x_o))$.

\begin{proposition}\label{prop:monodromyrepr}
The monodromy representation functor is an equivalence of categories.
\end{proposition}

\begin{comment}
\begin{proof}
Let us indicate the construction of a quasi-inverse functor. Let us choose a universal covering $(\wt X,\wt x_o)$ of $(X,x_o)$. The fundamental group $\pi_1(X,x_o)$ acts on~$\wt X$ and, through $\rho$, on $\wt X\times\wt\cG$, where $\wt\cG$ is the sheaf space of $\cG$. Then the quotient $(\wt X\times\wt\cG)/\pi_1$ is a sheaf space over $\XS$, which corresponds to an $S$-locally constant sheaf $\cF$.
\end{proof}
\end{comment}

\begin{remark}[Riemann-Hilbert]\label{rem:RH}
By the Riemann-Hilbert correspondence for coherent $S$-local systems proved in \cite[Th.\,2.23 p.\,14]{Deligne70}, the functor $\cF\mto\cO_{\XS}\otimes_{\pOS}\cF$ induces an equivalence between the category of coherent $S$-locally constant sheaves of $\pOS$-modules and the category of coherent $\cO_{\XS}$-modules $F$ equipped with an integrable relative connection $\nabla:F\to\Omega^1_{\XS/S}\otimes_{\cO_{\XS}}F$.
\end{remark}

\begin{comment}
\subsection{Flat families of local systems}
\end{comment}

\begin{proposition}\label{prop:flatLS}
Let $\cF$ be a coherent $S$-locally constant local system on $\XS$. Then the following properties are equivalent:
\begin{enumerate}
\item\label{prop:flatLS1}
there exists an $\sho_S$-locally free sheaf of finite rank $\cG$ such that locally $\cF\simeq p^{-1}\cG$;
\item\label{prop:flatLS2}
any coherent $\sho_S$-module $\cG$ such that locally $\cF\simeq p^{-1}\cG$ is $\sho_S$-locally free of finite rank.
\end{enumerate}

If $S$ is a complex manifold with its reduced structure, \eqref{prop:flatLS1} and \eqref{prop:flatLS2} are also equivalent~to
\begin{enumerate}\setcounter{enumi}{2}
\item\label{prop:flatLS3}
the dual $\bD\cF:=\bR\shhom_{\pOS}(\cF,\pOS)$ is an $S$-locally constant sheaf.
\end{enumerate}

If $X$ is connected and locally path-connected, and if $\pi_1(X,x_o)$ has finite presentation, so that $\Hom(\pi_1(X,x_o),\GL_r(\CC))$ is naturally an affine complex algebraic variety, then \eqref{prop:flatLS1} and \eqref{prop:flatLS2} are also equivalent to
\begin{enumerate}\setcounter{enumi}{3}
\item\label{prop:flatLS4}
for any open subset $V$ of $S$ on which some~$\cG$ as in \eqref{prop:flatLS2} is free or rank $r$, giving $\cF_{|X\times V}$ is equivalent to giving a holomorphic map $V\to\Hom(\pi_1(X,x_o),\GL_r(\CC))$.
\end{enumerate}
\end{proposition}

\begin{proposition}\label{prop:contractible}
Let $Y$ be a contractible topological space and let $\cF$ be an $S$-locally constant sheaf on $Y\times S$. Then $R^kp_{Y,*}\cF=0$ for each $k\geq1$ and $\cF$ is constant.
\end{proposition}

\backmatter
\providecommand{\eprint}[1]{\href{http://arxiv.org/abs/#1}{\texttt{arXiv\string:\allowbreak#1}}}\providecommand{\doi}[1]{\href{http://dx.doi.org/#1}{\texttt{doi\string:\allowbreak#1}}}
\providecommand{\og}{``}
\providecommand{\fg}{''}
\providecommand{\smfandname}{\&}


\begin{thebibliography}{10}

\bibitem{D-G-S11}
{\scshape A.~D'Agnolo, S.~Guillermou {\normalfont \smfandname} P.~Schapira} --
  {\og Regular holonomic {$\mathcal D[\![\hslash]\!]$}-modules\fg}, \emph{Publ.
  RIMS, Kyoto Univ.} \textbf{47} (2011), no.~1, p.~221--255.

\bibitem{Deligne70}
{\scshape P.~Deligne} -- \emph{Équations diff{\'e}rentielles {\`a} points
  singuliers r{\'e}guliers}, Lect. Notes in Math., vol. 163, Springer-Verlag,
  1970.

\bibitem{E-P16}
{\scshape {\relax M.J}.~Edmundo {\normalfont \smfandname} L.~Prelli} -- {\og
  Sheaves on {$\mathcal{T}$}-topologies\fg}, \emph{J.~Math. Soc. Japan}
  \textbf{68} (2016), no.~1, p.~347--381.

\bibitem{H-T-T08}
{\scshape R.~Hotta, K.~Takeuchi {\normalfont \smfandname} T.~Tanisaki} --
  \emph{{$D$-Modules, perverse sheaves, and representation theory}}, Progress
  in Math., vol. 236, Birkh{\"a}user, Boston, Basel, Berlin, 2008, in Japanese:
  1995.

\bibitem{Kashiwara78}
{\scshape M.~Kashiwara} -- {\og {On the holonomic systems of differential
  equations II}\fg}, \emph{Invent. Math.} \textbf{49} (1978), p.~121--135.

\bibitem{Kashiwara84}
\bysame , {\og The {Riemann-Hilbert} problem for holonomic systems\fg},
  \emph{Publ. RIMS, Kyoto Univ.} \textbf{20} (1984), p.~319--365.

\bibitem{Kashiwara03}
\bysame , \emph{{$D$}-modules and microlocal calculus}, Translations of
  Mathematical Monographs, vol. 217, American Mathematical Society, Providence,
  R.I., 2003.

\bibitem{K-K81}
{\scshape M.~Kashiwara {\normalfont \smfandname} T.~Kawai} -- {\og {On the
  holonomic systems of differential equations (systems with regular
  singularities) III}\fg}, \emph{Publ. RIMS, Kyoto Univ.} \textbf{17} (1981),
  p.~813--979.

\bibitem{K-S90}
{\scshape M.~Kashiwara {\normalfont \smfandname} P.~Schapira} -- \emph{{Sheaves
  on Manifolds}}, Grundlehren Math. Wiss., vol. 292, Springer-Verlag, 1990.

\bibitem{K-S96}
\bysame , \emph{Moderate and formal cohomology associated with constructible
  sheaves}, M{\'e}m. Soc. Math. France (N.S.), vol.~64, Soci{\'e}t{\'e}
  Math{\'e}matique de France, Paris, 1996.

\bibitem{K-S01}
\bysame , \emph{Ind-sheaves}, Ast{\'e}risque, vol. 271, Soci{\'e}t{\'e}
  Math{\'e}matique de France, Paris, 2001.

\bibitem{K-Sch06}
\bysame , \emph{{Categories and sheaves}}, Grundlehren Math. Wiss., vol. 332,
  Springer-Verlag, 2006.

\bibitem{Mebkhout87}
{\scshape Z.~Mebkhout} -- \emph{{Le formalisme des six op{\'e}rations de
  Grothendieck pour les {$\mathcal{D}$}-modules coh{\'e}rents}}, Travaux en
  cours, vol.~35, Hermann, Paris, 1989.

\bibitem{Mebkhout04}
\bysame , {\og {Le th{\'e}or{\`e}me de positivit{\'e}, le th{\'e}or{\`e}me de
  comparaison et le th{\'e}or{\`e}me d'existence de Riemann}\fg}, in
  \emph{{Él{\'e}ments de la th{\'e}orie des syst{\`e}mes diff{\'e}rentiels
  g{\'e}o\-m{\'e}\-triques}}, S{\'e}minaires \& Congr{\`e}s, vol.~8,
  Soci{\'e}t{\'e} Math{\'e}matique de France, Paris, 2004, p.~165--310.

\bibitem{Mochizuki07}
{\scshape T.~Mochizuki} -- \emph{{Asymptotic behaviour of tame harmonic bundles
  and an application to pure twistor $D$-modules}}, vol. 185, Mem. Amer. Math.
  Soc., no. 869-870, American Mathematical Society, Providence, R.I., 2007,
  \eprint{math.DG/0312230} \& \eprint{math.DG/0402122}.

\bibitem{Mochizuki11}
\bysame , \emph{{Mixed twistor D-Modules}}, Lect. Notes in Math., vol. 2125,
  Springer-Verlag, 2015.

\bibitem{MF-P14}
{\scshape T.~Monteiro~Fernandes {\normalfont \smfandname} L.~Prelli} -- {\og
  Relative subanalytic sheaves\fg}, \emph{Fund. Math.} \textbf{226} (2014),
  no.~1, p.~79--100.

\bibitem{MF-S12}
{\scshape T.~Monteiro~Fernandes {\normalfont \smfandname} C.~Sabbah} -- {\og
  {On the de Rham complex of mixed twistor $\mathcal{D}$-modules}\fg},
  \emph{Internat. Math. Res. Notices} (2013), no.~21, p.~4961--4984,
  \doi{10.1093/imrn/rns197}, \eprint{1204.1331}.

\bibitem{Prelli08}
{\scshape L.~Prelli} -- {\og Sheaves on subanalytic sites\fg}, \emph{Rend. Sem.
  Mat. Univ. Padova} \textbf{120} (2008), p.~167--216.

\bibitem{Prelli13}
\bysame , \emph{Microlocalization of subanalytic sheaves}, M{\'e}m. Soc. Math.
  France (N.S.), vol. 135, Soci{\'e}t{\'e} Math{\'e}matique de France, Paris,
  2013.

\bibitem{Bibi86I}
{\scshape C.~Sabbah} -- {\og {Proximit{\'e} {\'e}vanescente, I. La structure
  polaire d'un {$\mathcal{D}$}-module, Appendice en collaboration avec F.
  Castro}\fg}, \emph{Compositio Math.} \textbf{62} (1987), p.~283--328.

\bibitem{Bibi01c}
\bysame , \emph{{Polarizable twistor $\mathcal{D}$-modules}}, Ast{\'e}risque,
  vol. 300, Soci{\'e}t{\'e} Math{\'e}matique de France, Paris, 2005.

\bibitem{Sch-Sch94}
{\scshape P.~Schapira {\normalfont \smfandname} J.-P. Schneiders} --
  \emph{Index theorem for elliptic pairs}, Ast{\'e}risque, vol. 224,
  Soci{\'e}t{\'e} Math{\'e}matique de France, Paris, 1994.

\bibitem{Wang08}
{\scshape L.~Wang} -- {\og The constructibility theorem for differential
  modules\fg}, \smfphdthesisname, University of Illinois at Chicago, 2008,
  \url{http://indigo.uic.edu/handle/10027/13547}.

\end{thebibliography}
\end{document}